\documentclass[11pt,leqno]{amsart}

\usepackage{amsmath,amssymb,amsfonts,amstext,amsthm}
\usepackage{mathtools,latexsym,amscd}
\usepackage{tikz}
\usetikzlibrary{arrows.meta}
\usepackage{graphicx}
\usepackage[T1]{fontenc}
\usepackage[utf8]{inputenc}
\usepackage{microtype}
\usepackage{color}
\usepackage{parskip}
\usepackage{enumitem}
\usepackage{comment}
\usepackage{booktabs,array,adjustbox}
\usepackage[pagebackref,colorlinks]{hyperref}

\allowdisplaybreaks[1]

\setlist[enumerate]{leftmargin=2em,itemsep=3pt,topsep=5pt}

\theoremstyle{plain}
\newtheorem{theorem}{Theorem}
\newtheorem{lemma}[theorem]{Lemma}

\newtheorem{proposition}[theorem]{Proposition}
\newtheorem{corollary}[theorem]{Corollary}
\newtheorem*{corollary*}{Corollary}

\theoremstyle{definition}

\newtheorem{remark}[theorem]{Remark}
\newtheorem*{remark*}{Remark}
\newtheorem*{question*}{Question}

\newcommand{\R}{\mathbb R}
\newcommand{\Z}{\mathbb Z}
\newcommand{\F}{F_2}
\newcommand{\Hb}{H_b}
\newcommand{\rHb}{\overline H_b}
\newcommand{\rEH}{\overline{EH}_b}
\newcommand{\Sym}{\operatorname{Sym}}
\newcommand{\Homeo}{\operatorname{Homeo}}
\newcommand{\Diff}{\operatorname{Diff}}
\newcommand{\Symp}{\operatorname{Symp}}
\newcommand{\Ham}{\operatorname{Ham}}
\newcommand{\id}{\operatorname{id}}
\newcommand{\im}{\operatorname{im}}

\newcommand{\Int}{\operatorname{int}}
\newcommand{\eps}{\varepsilon}
\newcommand{\cG}{\mathcal T_M}
\newcommand{\norm}[1]{\left\lVert#1\right\rVert}
\newcommand{\pair}[2]{\left\langle#1,#2\right\rangle}

\hypersetup{
  pdftitle={Bounded cohomology of transformation groups in all degrees},
  pdfsubject={Reduced exact bounded cohomology of transformation groups},
  pdfkeywords={bounded cohomology, transformation groups, surfaces, cup products}
}

\begin{document}
\raggedbottom
\title[Bounded cohomology in all degrees]{Bounded cohomology of transformation groups in all degrees}
\author{Michael Brandenbursky}\thanks{M.B. was partially supported by the Israel Science Foundation grant 823/23.}
\address{Department of Mathematics, Ben Gurion University, Israel}
\email{brandens@bgu.ac.il}
\date{}

\begin{abstract}
Let $M$ be a compact connected oriented smooth manifold, of dimension
$n\ge2$. Let $\cG$ be one of the following transformation groups:
$\Homeo_0(M,\mu)$, $\Diff_0(M,\mu)$, 
$\Symp_0(M, \omega)$ (in case $M$ is symplectic) and $\Ham(M, \omega )$ 
(in case $M$ is symplectic). Denote by $\rHb^d(\cG)$ reduced 
bounded cohomology of $\cG$ in degree $d$, and by $\rEH^d(\cG)$ reduced exact bounded cohomology of $\cG$ in degree $d$. 
In this paper we prove that if $n=2$ and $M$ is a closed surface $\Sigma_g$ or a disc $\mathbb D$, then for every $d\ge 2$
$$\dim(\rHb^d(\cG))=\infty.$$ 
Moreover, $\dim\rEH^d(\cG)=\infty$ 
if $\cG$ is either $\Diff_0(\Sigma_g,\mu)$ or $\Ham(\Sigma_g, \omega)$, or $\Ham(\mathbb{D}, \omega)$; or $g>1$.
In case $n>2$, under certain conditions on $\pi_1(M)$, we prove that for every degree $d\ge 2$
$$\dim\rEH^d(\cG)=\infty.$$ 
In particular, these results hold for $\cG$ when $M$ is 
a closed, orientable smooth manifold admitting a Riemannian metric 
of strictly negative sectional curvature of an arbitrary dimension.
\end{abstract}

\maketitle
\markboth{Bounded cohomology in all degrees}{Bounded cohomology in all degrees}

\section{Introduction}
\label{sec:1}

Gromov's work \cite{Gromov} established bounded cohomology
as an important tool in geometry and topology.
Its applications include simplicial volume, group actions,
and symplectic geometry. Nevertheless, higher-degree
bounded cohomology remains difficult to compute, even for
free groups; see \cite{Frigerio}.

For transformation groups without a measure-preserving
condition, Monod--Nariman \cite{MonodNariman} obtained
complete computations and vanishing results.
In particular, their Theorem~1.10 gives
$$
H_b^q\bigl(\Homeo_0(S^2);\R\bigr)=H_b^q\bigl(\Diff_0(S^2);\R\bigr)=0\qquad(q=2,3).
$$
The area-preserving subgroups considered here exhibit
a different behaviour: their reduced bounded cohomology
is infinite-dimensional in every degree at least two.

In the measure-preserving setting, the constructions of
Gambaudo--Ghys \cite{GG} and Polterovich \cite{Polterovich}
provide quasimorphisms from braid groups and fundamental
groups. Such quasimorphisms have applications to
conjugation-invariant norms and dynamics; see, for example,
\cite{BMentropy,G-invariant}.
Brandenbursky--Marcinkowski \cite{BM} constructed transfer
maps in bounded cohomology and obtained infinite-dimensional
third reduced exact bounded cohomology for several
transformation groups, including those of hyperbolic
surfaces. Kimura \cite{Kimura} established analogous
configuration-space results for the disc and other
surfaces of non-negative Euler characteristic.
Nitsche \cite[Theorems~4.2--4.3]{Nitsche} obtained
infinite-dimensionality in every positive even degree
for volume-preserving homeomorphism groups of certain
manifolds of dimension at least three and for
area-preserving diffeomorphisms of the disc.

In this paper we obtain an all-degree conclusion for the
area-preserving transformation groups of every closed
oriented surface and of the disc. We also prove an
all-degree result for higher-dimensional manifolds under
the group-theoretic hypotheses stated below.

A principal difficulty in obtaining such results comes
from the very limited understanding of the higher-degree
bounded cohomology of free groups.
Under suitable hypotheses on the fundamental group,
Brandenbursky--Marcinkowski \cite[Theorems~A and~B]{BM}
proved that the reduced exact bounded cohomology of the
transformation groups considered there contains a linear
copy of $\rHb^d(F_2)$ for every $d\ge2$.
Thus, infinite-dimensionality of $\rHb^d(F_2)$ would
immediately give the corresponding conclusion for those
transformation groups in the same degree.
However, understanding the bounded cohomology of $F_2$
in degrees greater than three remains a fundamental open
problem; in particular, infinite-dimensionality of its
reduced bounded cohomology in these degrees is not known.
The degree-preserving transfer alone therefore does not
yield the desired all-degree conclusion.

Our construction avoids this difficulty by using symmetric
external products of degree-two and degree-three classes
on products of free groups, rather than requiring
higher-degree classes on $F_2$ itself.
We transfer these classes by integration over products
of manifolds and configuration spaces.
Suitable restriction maps recover the external products
and establish injectivity on the chosen spaces of classes.
Since every integer $d\ge2$ has the form $d=2m+q$,
with $m\ge0$ and $q\in\{2,3\}$, the known
infinite-dimensionality in degrees two and three
\cite{Brooks,Soma} suffices to obtain our all-degree results.

Let $M$ be a compact connected oriented smooth manifold of dimension
$n\ge2$ with possibly nonempty boundary. Let $\mu$ be a Borel probability
measure induced by a smooth positive volume form on $M$.
Let $\Homeo_0(M,\mu)$ denote the identity component of the group of measure-preserving
homeomorphisms of $M$ which are the identity on a neighborhood of
$\partial M$, and $\Diff_0(M,\mu)$ denote the identity component of the group of
volume-preserving diffeomorphisms $M$. If $(M,\omega)$ is symplectic, 
let $\Symp_0(M,\omega)$ denote the identity component of the group
of symplectomorphisms that are the identity near $\partial M$. 
Let $\Ham(M,\omega)$ denote the group of
Hamiltonian diffeomorphisms generated by Hamiltonian isotopies; when the
boundary is nonempty, the generating Hamiltonians have support in $\Int M$. 

Let $\Sigma_g$ denote the closed oriented surface of genus $g$, and fix
$x_0\in M$. We use the notation
$\pi=\pi_1(M,x_0),\quad G=\pi_M=\pi/Z(\pi),\quad \F=\langle a,b\rangle$,
where $Z(\pi)$ is the center of $\pi$ and $\F$ is the free group of rank
two. As in \cite{BM}, we use $\cG$ for any one of the transformation
groups under consideration. Theorem~\ref{thm:main} concerns $n\ge3$
under the hypotheses below. Theorem~\ref{thm:disc} concerns the disc
and the sphere, and Theorem~\ref{thm:surface} concerns all closed
oriented surfaces of positive genus.

Let $K$ be a discrete group. We denote its bounded cohomology by
$\Hb^d(K)$ and its reduced bounded cohomology by
$$
 \rHb^d(K)=\Hb^d(K)\big/\{u\in\Hb^d(K):\|u\|=0\},
$$
with the induced quotient norm. Let $c_K^d$ denote the comparison map
induced by inclusion of bounded cochains into ordinary cochains. Following
\cite{FPS}, let $EH_b^d(K)$ denote its kernel, called the
exact bounded cohomology of $K$. We write $\rEH^d(K)$ for the reduced
exact bounded cohomology:
$$
 \begin{aligned}
 &EH_b^d(K)=\ker\bigl(c_K^d:\Hb^d(K)\longrightarrow H^d(K;\R)\bigr),\\
 &\rEH^d(K)=\im\bigl(EH_b^d(K)\longrightarrow\rHb^d(K)\bigr).
 \end{aligned}
$$
Thus $\rEH^d(K)$ is $EH_b^d(K)$ modulo its zero-seminorm subspace,
with the norm induced from $\rHb^d(K)$. Consider the following conditions on $G=\pi/Z(\pi)$:

\begin{enumerate}[label=\textup{(\roman*)},leftmargin=2.4em]
\item There are a group $B$, an amenable group $A$, and injective homomorphisms
$A\to\F$ and $A\to B$ such that $G\cong\F*_A B$.
\item There is an epimorphism $r:G\twoheadrightarrow\F$.
\item The group $G$ contains a hyperbolically embedded subgroup isomorphic to $\F$.
\end{enumerate}

A surjection $\pi\twoheadrightarrow\F$ sends $Z(\pi)$ 
to the identity, so \textup{(ii)} is equivalent to Nitsche's epimorphism hypothesis on $\pi$ itself.

\begin{theorem}[Manifolds of dimension greater than two]\label{thm:main}
Let $M$ be a compact, connected, oriented smooth manifold of dimension
$n\ge3$, with normalized smooth volume measure $\mu$. Suppose that
$G=\pi_1(M)/Z(\pi_1(M))$ satisfies at least one of
\textup{(i)}, \textup{(ii)}, or \textup{(iii)}. Let $\cG$ denote either of
$$
 \Homeo_0(M,\mu),\qquad \Diff_0(M,\mu).
$$
If $M$ has a symplectic form $\omega$, one may instead take either of
$$
 \Symp_0(M,\omega),\qquad \Ham(M,\omega),
$$
with $\mu$ the normalized symplectic volume. 
\begin{enumerate}[label=\textup{(\alph*)},leftmargin=2.4em,itemsep=7pt]
\item \textbf{Infinite dimensionality.}
For every integer $d\ge2$,
\begin{equation}\label{eq:main}
 \dim_{\R}\rEH^d(\cG)=\infty,
 \qquad \dim_{\R}\rHb^d(\cG)=\infty.
\end{equation}

\item \textbf{Injective maps.}
There are a linear subspace $K_2\subset\rHb^2(\F)$ of codimension at most
one and injective linear maps
\begin{align}
 &\Phi_{0,2}:K_2
 \longrightarrow\rEH^2(\cG),\label{eq:baseinj}\\
 &\Phi_{0,3}:\rHb^3(\F)
 \longrightarrow\rEH^3(\cG),\notag\\
 &\Phi_{m,q}:\rHb^q(\F)
 \longrightarrow\rEH^{2m+q}(\cG)
 \qquad(m\ge1,\ q\in\{2,3\}).\label{eq:highinj}
 \end{align}
Under \textup{(ii)} or \textup{(iii)}, one may take $K_2=\rHb^2(\F)$.
\end{enumerate}
\end{theorem}
Let $M$ be a closed, connected, orientable smooth manifold of dimension $n\geq 3$ 
admitting a Riemannian metric of strictly negative sectional curvature. 
Then $\pi_1(M)$ is center-free and non-elementary torsion-free word-hyperbolic group, and hence it
admits a hyperbolically embedded $F_2$ \cite[Theorem~6.14(c)]{DGO}. 
Now Theorem \ref{thm:main} implies the following

\begin{corollary*}
Let $M$ be a closed, connected, orientable smooth manifold of dimension $n\geq3$ 
admitting a Riemannian metric of strictly negative sectional curvature. 
Let $\mu$ be the normalized measure induced by any smooth positive volume form. Then, for every $d\ge2$,
$$
\dim\overline{EH}_b^d\bigl(\operatorname{Homeo}_0(M,\mu)\bigr)=
\infty\quad\text{and}\quad\dim\overline{EH}_b^d\bigl(\operatorname{Diff}_0(M,\mu)\bigr)=\infty.
$$
In particular, their reduced bounded cohomology is infinite-dimensional in every degree at least two.
If $M$ also admits a symplectic form $\omega$, the same reduced exact conclusions hold for
$$
\operatorname{Symp}_0(M,\omega)\quad\text{and}\quad\operatorname{Ham}(M,\omega),
$$
using the normalized symplectic volume.
\end{corollary*}

The second theorem treats the closed two-disc $\mathbb D$ and the sphere
$S^2$, using three-point and four-point configurations, respectively.

\begin{theorem}[The disc and the sphere]\label{thm:disc}
Let $S$ be $\mathbb D$ or $S^2$, equipped with an area
form $\omega$, and let $\mu$ be its induced measure. Let
$d\ge2$ be an integer. Then
$$
 \dim\rHb^d\bigl(\Homeo_0(S,\mu)\bigr)=\infty,\qquad\dim\rEH^d\bigl(\Ham(S,\omega)\bigr)=\infty.
$$
\end{theorem}

For positive-genus surfaces we use a mixed transfer. Its distinguished
factor is the compact one-point transfer when $g\ge2$, and the
two-point transfer when $g=1$. For torus homeomorphisms, a family of
two-point quasimorphisms supplies exactness in every even degree.
For these homeomorphisms, in odd degrees we assert only the reduced
bounded conclusion.
Write $T=T^2=\Sigma_1$. For $n=2,3$, let $\operatorname{Mod}(S,n)$
denote the mapping class group with $n$ marked points, allowing their
permutations, and set
$$
 K(S,n)=\ker\bigl(\operatorname{Mod}(S,n)
                        \longrightarrow\operatorname{Mod}(S)\bigr).
$$
The group $K(T,2)$ in the next theorem is identified with
$B_2(T)/Z(B_2(T))$ in the torus proof. 

\begin{theorem}[Closed surfaces of positive genus]\label{thm:surface}
Let $S=\Sigma_g$ be a closed connected oriented surface of genus $g\ge1$.
Let $\omega$ be a smooth positive area form and $\mu$ its normalized
area measure. Let $\mathcal T_S$ denote any one of
$$
 \Homeo_0(S,\mu),\quad
 \Symp_0(S,\omega),\quad\Ham(S,\omega).
$$
\begin{enumerate}[label=\textup{(\alph*)},leftmargin=2.4em,itemsep=7pt]
\item For every integer $d\ge2$,
$$
 \dim_{\R}\rHb^d(\mathcal T_S)=\infty.
$$
If $g\ge2$, or if $\mathcal T_S$ is one of the two smooth groups,
or if $d$ is even, then also
$$
 \dim_{\R}\rEH^d(\mathcal T_S)=\infty.
$$
\item For every $m\ge0$ and $q\in\{2,3\}$, there are injective
bounded linear maps, bounded below, as follows. For $g\ge2$ they are
$$
 \Phi_{m,q}:\rHb^q(\F)\longrightarrow\rEH^{2m+q}(\mathcal T_S).
$$
For $g=1$ they are
$$
 \Phi^T_{m,q}:\rHb^q(K(T,2))\longrightarrow\rHb^{2m+q}(\mathcal T_T).
$$
The latter maps have reduced exact image for each of the smooth groups.
For $\mathcal T_T=\Homeo_0(T,\mu)$, there is a single infinite-dimensional
subspace $\mathcal E_T\subset\rHb^2(K(T,2))$ of codimension at most one
such that every restriction $\Phi^T_{m,2}|_{\mathcal E_T}$ has reduced
exact image.
\end{enumerate}
\end{theorem}

\begin{remark*}[Degree two]
For a discrete group $K$, the map $f\mapsto[\delta f]$ identifies the
space of real homogeneous quasimorphisms modulo homomorphisms
$K\to\R$ with $EH_b^2(K)$; see \cite[Sections~2.3--2.4]{Frigerio}.
Since $H_b^2(K)$ is normed, this is also $\rEH^2(K)$ by
\cite[Theorem~6.5]{Frigerio}. Thus the degree-two conclusions of
Theorems~\ref{thm:main} and~\ref{thm:surface}, and the smooth conclusions
of Theorem~\ref{thm:disc}, give infinite-dimensional spaces of homogeneous
quasimorphisms even after quotienting by homomorphisms. This is a
bounded-cohomological proof of that consequence, not an assertion that
the existence of such quasimorphisms is new. Relevant constructions and
injectivity results already appear in \cite{GG,Polterovich,BM,BMentropy,Kimura}.
For torus homeomorphisms, the required extension of two-point quasimorphisms
follows from \cite[Theorem~1.1]{BKMS}; see also \cite{BKS}.
\end{remark*}

\begin{remark*}[Norm estimates]
For each map in Theorems~\ref{thm:main}\textup{(b)}
and~\ref{thm:surface}\textup{(b)}, there is a constant $c_{m,q}>0$
such that
$$
\|\Phi_{m,q}(v)\|\ge c_{m,q}\|v\|\qquad\text{for every $v$ in the domain of $\Phi_{m,q}$.}
$$
The same statement holds with $\Phi^T_{m,q}$ in the torus case.
The constant may depend on $M$, $\cG$, $m$, $q$, and the construction, but not on $v$.
All the maps in Theorem~\ref{thm:surface} can be chosen bounded as well.
In Theorem~\ref{thm:main}, these maps can be chosen bounded
under \textup{(i)} or \textup{(ii)}. Under \textup{(iii)}, 
the construction gives injective linear maps with the stated lower norm bounds; 
no upper norm bound is asserted.
\end{remark*}

We also record an ordinary-cohomology consequence of the
classical mass-flow construction: a nonzero real character
of $\pi_1(M)/Z(\pi_1(M))$ yields infinite-dimensional ordinary
cohomology in every positive degree for the volume-preserving
homeomorphism and diffeomorphism groups, see Theorem \ref{thm:ordinary-cohomology}.

\emph{Symmetric external products.}
For a group $K$ and $\sigma\in S_k$, we define a
permutation of $K^k$ by
$\sigma(g_1,\ldots,g_k)=(g_{\sigma^{-1}(1)},\ldots,g_{\sigma^{-1}(k)})$.
Set
$$\Sym_k\xi=\sum_{\sigma\in S_k}\sigma^*\xi.$$

A class on $K^k$ is called \emph{symmetric} if every $\sigma^*$ fixes it.
For every class $\xi$, the sum $\sum_{\sigma\in S_k}\sigma^*\xi$ is symmetric. 
If $\xi$ is already symmetric, then $\Sym_k\xi=k!\xi$. Thus the normalized operator $\frac{1}{k!}\Sym_k$ 
acts as the identity on the symmetric subspace and is an idempotent linear projection onto that subspace.
A transposition sends $u\times v$ to $(-1)^{\deg u\deg v}v\times u$; this need not equal
$u\times v$. With degree-two fixed factors and at most one other factor,
all permutation signs are positive.

Now we sketch the proof of Theorem~\ref{thm:main}, proofs of Theorems \ref{thm:disc} and \ref{thm:surface} 
use similar ideas, thus we do not sketch their proofs in the introduction.

\begin{proof}[Sketch of the proof of Theorem~\ref{thm:main}]
We construct the maps in part~\textup{(b)}
using only classes of degrees two and three on $\F$. 
Both $\rHb^2(\F)$ and $\rHb^3(\F)$ are infinite-dimensional, see, e.g., \cite{Brooks,Soma}.
For $m\ge1$ and $q\in\{2,3\}$, put $k=m+1$ and $d=2m+q$. Consider the composition
$$
 \rHb^q(\F) \xrightarrow{\ \widetilde T_m^q\ }
 \rHb^d(G^k)\xrightarrow{\ k!^{-1}\overline\Gamma_{b,k}^{\,d}\ }\rHb^d(\cG).
$$
The first map is the unnormalized sum of the extended external products, see Section \ref{sec:2}.
The second is the Gambaudo--Ghys map over $M^k$, divided by $k!$, see Section \ref{sec:3}.
Thus we transfer the normalized symmetric class, as in
\eqref{eq:transfer-symmetrization}. We explain why their
composition is injective and why its image lies in reduced exact bounded cohomology.

\begin{itemize}[leftmargin=1.5em,itemsep=7pt,topsep=6pt]
\item
Each of the three hypotheses provides an embedding $i:\F\hookrightarrow G$
and linear extension maps $J_2,J_3$ satisfying
$$i^*J_s=\id\qquad(s=2,3).$$
Lemmas~\ref{lem:exact} and~\ref{lem:othersections} construct these maps and
control their images under comparison. The image of $J_3$ lies in reduced exact bounded cohomology.
The same holds for $J_2$ on a subspace
$K_2\subset\rHb^2(\F)$ of codimension at most one; under \textup{(ii)} or
\textup{(iii)}, this is the whole space. Lemma~\ref{lem:brooks} chooses in
$K_2$ a nonzero finite linear combination $\theta$ of Brooks classes such that
$$\theta\smile u=0\qquad\bigl(u\in\rHb^s(\F),\ s\ge1\bigr).
$$
The vanishing is the theorem of Amontova--Bucher \cite[Theorem~A\textup{(b)}]{AB}.
External products with this fixed class increase the degree. Its cup-product
vanishing property eliminates the terms with repeated coordinate assignments
in the calculation.

\item
First consider two factors. For $v\in\rHb^q(\F)$, take
$$\theta\times v+v\times\theta\in\rHb^{q+2}(\F\times\F).
$$
This depends injectively on $v$. To see it, choose a degree-two
$\ell^1$-cycle $z$ with $\langle\theta,z\rangle=1$
(Lemma~\ref{lem:brooks}). The dual of the chain map $b\mapsto z\times b$
evaluates the first factor on $z$ and leaves a reduced bounded-cohomology
class on the second factor. Applied to the displayed sum, it gives
$$v\quad(q=3),\qquad v+\langle v,z\rangle\theta\quad(q=2).$$
The second expression also determines $v$, since
$\langle\theta,z\rangle=1$ (see ~\eqref{eq:explicitQ}). 

For $k=m+1$ factors, use the symmetric sum
$$T_m^q(v)=\Sym_k(\theta^{\times m}\times v).$$
Here $\theta^{\times m}$ denotes $m$ external factors.
Proposition~\ref{prop:symmetric-products}, using Lemma~\ref{lem:crossdual}, proves that
$T_m^q$ has a bounded left inverse. Extend each factor separately and set
$$
\widetilde T_m^q(v)=\operatorname{Sym}_k\bigl(J_2(\theta)^{\times m}\times J_q(v)\bigr).
$$
For the product embedding $i^k(w_1,\ldots,w_k)=(i(w_1),\ldots,i(w_k))$,
naturality gives
$$(i^k)^*\widetilde T_m^q(v)=T_m^q(v).$$
Thus the extended classes remain independent. No compatibility of the
maps $J_s$ with cup products is required.

\item
Each point in $M$ and an isotopy of a transformation, together with paths to a basepoint,
produce an element of $G$ in a natural way. For $k$ points this gives
an element of $G^k$. Integrating a bounded cocycle on these elements over
$M^k$ defines the second map in the composition. Lemma~\ref{lem:transfer}
proves that it is a bounded cochain construction and commutes with
comparison. For each fixed tuple of transformations the cocycle
has finite image, so the same integral is defined even on ordinary
cochains that are not bounded.

Every summand of $\widetilde T_m^q(v)$ contains the factor
$J_2(\theta)\in\rEH^2(G)$, since $m\ge1$. Its comparison image is therefore
zero, and the transferred class belongs to reduced exact bounded cohomology as well. We denote this composition
by $\Phi_{m,q}$.

\item
It remains to show that the transfer is injective on the subspace of
external-product classes just constructed.
Lemma~\ref{lem:push} constructs homomorphisms
$\rho_\eps:\F^k\to\cG$ by finger-pushing isotopies in $k$ disjoint
figure-eight neighborhoods. On a subset $U_{j,\eps}$ of the $j$th
neighborhood, the associated homomorphism is $i\circ p_j$, up to conjugation,
where $p_j:\F^k\to\F$ is the $j$th projection. The measure of $U_{j,\eps}$ tends to a positive number $\lambda_j$.
Other fixed subsets give cyclic or trivial associated homomorphisms; the
remaining transition subset has measure tending to zero.

For two coordinates, the products
$U_{1,\eps}\times U_{2,\eps}$ and $U_{2,\eps}\times U_{1,\eps}$ 
contribute the same class $\theta\times v+v\times\theta$. If both points lie in $U_{j,\eps}$,
both cocycle coordinates use the same projection $p_j$. The pullback then factors through the
diagonal $\Delta(w)=(w,w)$, and
$$
\Delta^*(\theta\times v)=\theta\smile v=0.
$$
The other summand vanishes for the same reason. 
Homomorphisms with cyclic or trivial image also give zero in positive-degree bounded cohomology.
Finally, boundedness of the cocycle makes the contribution of the
transition subset tend to zero.

For $k$ coordinates, every repeated assignment puts at least two factors
on the same free group. At least one of them is $\theta$, since only one
factor varies, so that internal cup product vanishes. The $k!$ bijective assignments
give the remaining contributions and have the same product measure.
Proposition~\ref{prop:approx} proves the resulting convergence
$$
\rho_\eps^*\Phi_{m,q}(v)\longrightarrow C_kT_m^q(v),\qquad C_k=\prod_{j=1}^k\lambda_j>0,
$$
in reduced norm, first for two coordinates and then for general $k$.

\begin{samepage}
As in the proof outline in \cite[Section~3]{BM}, the comparison between
transfer and restriction can be summarized by the diagram
$$
\begin{tikzpicture}[>=Stealth, baseline=(current bounding box.center)]
  \node (A1) at (0,2.2) {$\rHb^q(\F)$};
  \node (B1) at (4.7,2.2) {$\rHb^d(G^k)$};
  \node (C1) at (10.1,2.2) {$\rHb^d(\cG)$};
  \node (A2) at (0,0) {$\rHb^q(\F)$};
  \node (B2) at (4.7,0) {$\rHb^d(\F^k)$};

  \draw[->] (A1) -- node[above] {$\widetilde T_m^q$} (B1);
  \draw[->] (B1) -- node[above] {$k!^{-1}\overline\Gamma_{b,k}^{\,d}$} (C1);
  \draw[->] (A1) -- node[left] {$=$} (A2);
  \draw[->] (A2) -- node[below] {$C_kT_m^q$} (B2);
  \draw[->] (B1) -- node[right] {$C_k(i^k)^*$} (B2);
  \draw[->] (C1) -- node[right] {$\rho_\eps^*$} (B2);
  \draw[->, bend left=16] (A1) to node[above] {$\Phi_{m,q}$} (C1);
\end{tikzpicture}
$$
\end{samepage}
This diagram commutes in the limit $\eps\to0$ on the source space
$\rHb^q(\F)$: for each $v$, the two paths from the upper-left corner to the
lower-right corner have values whose difference tends to zero in reduced norm.
Indeed, the upper path gives $\rho_\eps^*\Phi_{m,q}(v)$, whereas the lower
path gives $C_kT_m^q(v)$ because $(i^k)^*\widetilde T_m^q(v)=T_m^q(v)$.

We do not assert exact commutativity for fixed $\eps$, or convergence
on arbitrary classes in $\rHb^d(G^k)$.
If $\Phi_{m,q}(v)=0$, the left-hand side is zero for every $\eps$.
Thus $T_m^q(v)$ has zero reduced norm and is the zero class.
Injectivity of $T_m^q$ gives $v=0$. This proves the required injectivity.
\end{itemize}

Degrees two and three use the same argument with one coordinate and no
fixed external factor: apply the transfer to $J_2|_{K_2}$ and to $J_3$,
respectively. The degrees $2m+2$ and $2m+3$ exhaust all integers at least
two. The injective maps therefore imply part~\textup{(a)}.
Lemma~\ref{lem:push} gives the isotopies in the volume-preserving groups
and, in the symplectic case and in $\Ham$ itself. 
The argument consequently applies to every group in the theorem.
\end{proof}

\textbf{Acknowledgments.}
The author acknowledges the support of Israel Science Foundation grant
823/23. He thanks the Department of Mathematics at Ben Gurion University for
excellent working conditions. The author used an LLM for computational exploration, algebraic checks,
and editorial assistance. The mathematical ideas, research direction, and the transfer method
approach were developed exclusively by the author.


\section{Bounded cohomology of products of groups\texorpdfstring{\\}{ } and transfer maps}\label{sec:2}

In this section establish the common algebraic and transfer lemmas. 

\emph{Cup and external products.}
We first fix our cochain conventions. Let $K$ be a discrete group, and let $f$ and $h$ be
inhomogeneous cochains on $K$ of degrees $p$ and $q$, respectively.
Their cup product is defined by
\begin{equation}\label{eq:cupcochain}
 (f\smile h)(g_1,\ldots,g_{p+q})
 =f(g_1,\ldots,g_p)\,h(g_{p+1},\ldots,g_{p+q}).
\end{equation}
See \cite[Section~2.2]{AB}. This product satisfies the Leibniz identity
$$\delta(f\smile h)=\delta f\smile h+(-1)^p f\smile\delta h.$$
If both cochains are bounded, then their cup product is bounded and
$$\|f\smile h\|_\infty\le\|f\|_\infty\|h\|_\infty.$$

For discrete groups $K$ and $L$, let
$p_K:K\times L\to K$ and $p_L:K\times L\to L$
denote the coordinate projections. For cochains $f$ on $K$ and
$h$ on $L$, define their external product by
$f\times h=p_K^*f\smile p_L^*h$.
We use the same notation for the induced products on ordinary
and bounded cohomology. For bounded-cohomology classes of the
appropriate degrees, the canonical seminorms satisfy
\begin{equation}\label{eq:productnorm}
 \|u\smile v\|\le\|u\|\|v\|,\qquad\|u\times v\|\le\|u\|\|v\|.
\end{equation}
See \cite[Section~2.3]{Loh} and \cite[Section~2]{Nitsche}.
Our convention does not include the factor
$(-1)^{pq}$ appearing in L\"oh's definition. This difference
has no effect when one of the degrees is even.
The following lemma is proved in 
\begin{lemma}[\cite{FFPS}]\label{lem:zero}
For every discrete group $K$, a class of zero seminorm in $\Hb^d(K)$ has zero image in $H^d(K;\R)$. Consequently the comparison map induces a linear map
$$
\overline{c}_K^d:\rHb^d(K)\longrightarrow H^d(K;\R),\qquad\rEH^d(K)=\ker\overline{c}_K^d.
$$
External and cup products descend to reduced bounded cohomology and are compatible with these comparison maps.
\end{lemma}

The following lemma is a consequence of the isometric section theorem
of Bucher--Burger--Frigerio--Iozzi--Pagliantini--Pozzetti
\cite{BBFIPP}; see also \cite[Theorem~9.7]{Frigerio}.
The assertions about comparison maps follow from the ordinary
Mayer--Vietoris sequence of the amalgam.

\begin{lemma}\label{lem:exact}
Assume \textup{(i)} and write
\begin{equation}\label{eq:amalgam}
 G\cong\F*_A B.
\end{equation}
Let $i:\F\hookrightarrow G$ and $j:B\hookrightarrow G$ denote
the vertex inclusions. For every $q\ge2$, there are isometric
linear maps
$$
 \begin{aligned}
 &J_b^q:\Hb^q(\F)\longrightarrow\Hb^q(G),\\
 &J_q:\rHb^q(\F)\longrightarrow\rHb^q(G)
 \end{aligned}
$$
such that
\begin{equation}\label{eq:section}
 \begin{aligned}
 &i^*J_b^q=\id,\qquad j^*J_b^q=0,\\
 &i^*J_q=\id,\qquad j^*J_q=0.
 \end{aligned}
\end{equation}
Moreover,
$\overline c_G^q\circ J_q=0\quad(q\ge3),\qquad\operatorname{rank}(\overline c_G^2\circ J_2)\le1$.
In particular, $K_2=\ker(\overline c_G^2\circ J_2)$
has codimension at most one in $\rHb^2(\F)$.
\end{lemma}
\begin{proof}
Since $G$ is countable and $A$ is amenable, the isometric section
theorem gives a linear map
$$\Theta^q:\Hb^q(\F)\oplus\Hb^q(B)\longrightarrow\Hb^q(G)$$
which is an isometric right inverse of the combined restriction
map $(i^*,j^*)$. Here the direct sum carries the maximum seminorm.
Define $J_b^q(u)=\Theta^q(u,0)$. Then
$$
i^*J_b^q(u)=u,\qquad j^*J_b^q(u)=0,\qquad\|J_b^q(u)\|=\|u\|.
$$
Thus $J_b^q$ preserves zero-seminorm classes and induces an
isometric map $J_q$ on reduced bounded cohomology. This proves
\eqref{eq:section}.

An amenable subgroup of a free group is either trivial or infinite
cyclic. Consequently,
\begin{equation}\label{eq:edgecohom}
 \begin{aligned}
 &A=1\ \text{or}\ A\cong\Z,\\
 &H^s(A;\R)=0\quad(s\ge2),
 \qquad \dim H^1(A;\R)\le1.
 \end{aligned}
\end{equation}
We also have $H^q(\F;\R)=0$ for $q\ge2$, see \cite[Section~1.3]{Hatcher}.

Let $\iota_F:A\to\F$ and $\iota_B:A\to B$ denote the edge
inclusions. The ordinary Mayer--Vietoris sequence contains
$$
 \begin{aligned}
 &H^{q-1}(\F;\R)\oplus H^{q-1}(B;\R)
   \xrightarrow{\ d^{q-1}\ } H^{q-1}(A;\R)\\
 &\xrightarrow{\ \partial_A^{q-1}\ } H^q(G;\R)
   \xrightarrow{\ R^q\ }
   H^q(\F;\R)\oplus H^q(B;\R),
 \end{aligned}
$$
where
$$d^{q-1}(x,y)=\iota_F^*x-\iota_B^*y,\qquad R^q(w)=(i^*w,j^*w).$$
Here $\partial_A^{q-1}$ is the connecting homomorphism; see
\cite[Chapter~II, Section~7]{Brown}.

Let $u\in\rHb^q(\F)$ and put
$w=\overline c_G^q(J_q(u))$. Naturality and
\eqref{eq:section} give
$$
 \begin{aligned}
 &i^*w
   =\overline c_{\F}^q(i^*J_q(u))
   =\overline c_{\F}^q(u)=0,\\
 &j^*w
   =\overline c_B^q(j^*J_q(u))
   =\overline c_B^q(0)=0.
 \end{aligned}
$$
Exactness of the Mayer--Vietoris sequence therefore gives
\begin{equation}\label{eq:mv}
 \im(\overline c_G^q\circ J_q)
 \subseteq \ker R^q
 =\im\partial_A^{q-1}.
\end{equation}

For $q\ge3$, the domain $H^{q-1}(A;\R)$ of the connecting
homomorphism is zero by \eqref{eq:edgecohom}. Hence
\eqref{eq:mv} implies $\overline c_G^q\circ J_q=0$.

For $q=2$, the same inclusion gives
$\operatorname{rank}(\overline c_G^2\circ J_2)\le \dim H^1(A;\R)\le1$.
More precisely, exactness identifies the possible obstruction
space as
\begin{equation}\label{eq:obstructionquotient}
 \ker R^2
 \cong
 \frac{H^1(A;\R)}
 {\im\bigl(
 H^1(\F;\R)\oplus H^1(B;\R)
 \xrightarrow{\ d^1\ }H^1(A;\R)
 \bigr)}.
\end{equation}
Thus the comparison image of $J_2$ lies in one fixed subspace
of dimension at most one. Its kernel $K_2$ consequently has
codimension at most one, as claimed.
\end{proof}

\begin{lemma}\label{lem:othersections}
Under \textup{(ii)} or \textup{(iii)} there are an embedding $i:\F\hookrightarrow G$ and linear maps
$$J_q:\rHb^q(\F)\longrightarrow\rEH^q(G),\qquad i^*J_q=\id\quad(q=2,3).$$
Under \textup{(ii)} these maps are isometric homomorphism pullbacks. Under \textup{(iii)} 
they are algebraic linear right inverses of restriction.
\end{lemma}
\begin{proof}
Under \textup{(ii)}, freeness gives a section $i:\F\to G$ of $r$, so $r\circ i=\id$. Put $J_q=r^*$. 
Naturality and the norm bound for pullback give
$$
i^*J_q=\id,\qquad\|v\|\le\|J_q(v)\|\le\|v\|,\qquad\overline c_G^qJ_q=r^*\overline c_{\F}^q=0.
$$
Here $H^q(\F;\R)=0$ for $q\ge2$. 

Under \textup{(iii)}, the exact extension theorem of Frigerio--Pozzetti--Sisto 
\cite[Corollary~2 and Proposition~5.1]{FPS} states that $i^*:EH_b^q(G)\longrightarrow EH_b^q(\F)$
is onto for $q\ge2$. Since $EH_b^q(\F)=H_b^q(\F)$, lifting a representative of 
any reduced bounded-cohomology class and then passing to the quotient gives a surjection
$$
 i^*:\rEH^q(G)\twoheadrightarrow \rHb^q(\F).
$$
Choose lifts of a Hamel basis and extend linearly to obtain $J_q$ for $q=2,3$. 
Its image lies in reduced exact bounded cohomology and $i^*J_q=\id$; 
compare \cite[proof of Theorem~B]{BM}. 
The inequality $\|J_q(v)\|\ge\|v\|$ follows from the norm bound for restriction.
\end{proof}

For the remainder of the proof, fix $i,J_2,J_3$ as in
Lemma~\ref{lem:exact} or Lemma~\ref{lem:othersections}. We use the following
subspace for the degree-two assertion of Theorem~\ref{thm:main}:
\begin{equation}\label{eq:commondata}
 \begin{aligned}
 &i^*J_q=\id\quad(q=2,3),\qquad
 J_3\bigl(\rHb^3(\F)\bigr)\subset\rEH^3(G),\\
 &K_2:=\ker\bigl(\overline c_G^2\circ J_2\bigr)
 =\{v\in\rHb^2(\F):\overline c_G^2(J_2(v))=0\}.
 \end{aligned}
\end{equation}
Thus $K_2$ consists precisely of the reduced bounded-cohomology classes whose chosen extension lies in $\rEH^2(G)$.
Lemma~\ref{lem:exact} gives
$\dim(\rHb^2(\F)/K_2)\le1$ under \textup{(i)}.
Under \textup{(ii)} or \textup{(iii)}, Lemma~\ref{lem:othersections} gives
$K_2=\rHb^2(\F)$.
The approximation below will be applied to individual classes in the images of $J_q$.

\begin{lemma}\label{lem:brooks}
There exist a nonzero class $\theta\in \rHb^2(\F)$, 
a bounded representative given by the coboundary of a homogeneous quasimorphism, 
and an $\ell^1$ bar cycle $z\in Z_2^{\ell^1}(\F)$ such that
$$
J_2(\theta)\in\rEH^2(G),\qquad\pair{\theta}{z}=1,\qquad\theta\smile v=0\quad(v\in\rHb^{>0}(\F)).
$$
The class $\theta$ may be chosen to be a finite linear combination of Brooks classes.
\end{lemma}

\begin{proof}
The algebraic span of Brooks classes in $H_b^2(\F)$ is infinite-dimensional 
by Brooks--Mitsumatsu \cite{Brooks,Mitsumatsu}; see also \cite[Section~2.6]{Frigerio}. 
By Matsumoto--Morita, the seminorm on $H_b^2(\F)$ is a norm and the Kronecker map is an isomorphism
\begin{equation}\label{eq:degree2duality}
H_b^2(\F)=\rHb^2(\F)\xrightarrow{\ \cong\ }\bigl(\overline H_2^{\ell^1}(\F)\bigr)',
\qquad u\longmapsto\langle u,\,\cdot\,\rangle.
\end{equation}
Here $\overline H_2^{\ell^1}(\F)$ denotes reduced $\ell^1$-homology. 
We use \cite[Theorem~2.3 and Corollary~2.7]{MM}; 
the needed statements are also given in \cite[Theorem~6.5 and Corollary~6.7]{Frigerio}. 
The same degree-two duality is used in \cite[proof of Theorem~4.2]{Nitsche}. 

By \eqref{eq:commondata}, the intersection of $K_2$ with the span of 
Brooks classes has codimension at most one. 
Choose a nonzero class $\theta$ in this intersection. 
It is a finite linear combination of Brooks classes 
and has a representative $\delta q$, where $q$ is a 
homogeneous quasimorphism: replace each Brooks quasimorphism by 
its homogenization.

By \eqref{eq:degree2duality}, some $\ell^1$-cycle evaluates nontrivially on $\theta$. 
Multiplying that cycle by a scalar gives $z\in Z_2^{\ell^1}(\F)$ with $\langle\theta,z\rangle=1$. 
The condition $\theta\in K_2$ gives exactness of $J_2(\theta)$. 
Finally, \cite[Theorem~A\textup{(b)}]{AB} states that the cup product of 
a Brooks class with every bounded-cohomology class of positive degree is zero. 
Apply this to the finite linear combination defining $\theta$ 
and pass to reduced bounded cohomology using Lemma~\ref{lem:zero}.
\end{proof}

The following lemma is the bounded version of the classical
slant product with a fixed homology class; see
\cite[Chapter~6, Sections~1 and~10]{Spanier}.
For its bounded chain maps and homotopies, see
\cite[Lemma~2.6 and the proof of Proposition~2.5]{Loh}.
We retain the notation $L_Z^*$ for the dual of cross product
with a fixed cycle $Z$.

Let $C_*^{\ell^1}(K)$ denote the standard unnormalized
$\ell^1$-chain complex of a group $K$ with trivial real
coefficients, and let $C_b^*(K)$ denote its continuous dual.
Write $Z_p^{\ell^1}(K)=\ker\partial_p$.
We use the standard shuffle cross product
$$
\times:C_r^{\ell^1}(K)\otimes C_t^{\ell^1}(L)\longrightarrow C_{r+t}^{\ell^1}(K\times L).
$$
For chains $x$ and $y$ of degrees $r$ and $t$, respectively,
it satisfies
\begin{equation}\label{eq:shufflebound}
 \|x\times y\|_1
 \le\binom{r+t}{r}\|x\|_1\|y\|_1
\end{equation}
and
\begin{equation}\label{eq:shufflechain}
 \partial(x\times y)
 =\partial x\times y+(-1)^r x\times\partial y.
\end{equation}
See \cite[Section~2.3]{Loh}. The binomial coefficient in
\eqref{eq:shufflebound} is the number of $(r,t)$-shuffles.

\begin{lemma}\label{lem:crossdual}
Let $K,L$ be discrete groups, let $p\ge0$ be even, and let
$Z\in Z_p^{\ell^1}(K)$. Cross product with $Z$ defines
$$
L_Z:C_q^{\ell^1}(L)\longrightarrow C_{p+q}^{\ell^1}(K\times L),\qquad L_Z(b)=Z\times b.
$$
Its dual cochain map
$$
L_Z^*:C_b^{p+q}(K\times L)\longrightarrow C_b^q(L),\qquad (L_Z^*c)(b)=c(Z\times b),
$$
induces a bounded linear map $L_Z^*:\rHb^{p+q}(K\times L)\longrightarrow\rHb^q(L)$
for every $q\ge0$, with
\begin{equation}\label{eq:dualbound}
\|L_Z^*\|\le\binom{p+q}{p}\|Z\|_1.
\end{equation}
For $u\in\rHb^s(K)$ and $v\in\rHb^t(L)$ with
$s+t=p+q$, we have
\begin{equation}\label{eq:dualproduct}
 L_Z^*(u\times v)=
 \begin{cases}
  \langle u,Z\rangle\,v,&s=p,\\
  0,&s\ne p.
 \end{cases}
\end{equation}
\end{lemma}

\begin{proof}
Since $\partial Z=0$ and $p$ is even,
\eqref{eq:shufflechain} shows that $L_Z$ commutes with
the boundary. Its dual therefore commutes with the
coboundary, and \eqref{eq:shufflebound} gives
\eqref{eq:dualbound}. Being bounded, it preserves closures
of coboundaries and hence induces the asserted map on
reduced bounded cohomology.

The classical slant-product calculation gives
\eqref{eq:dualproduct}; see
\cite[Chapter~6, Sections~1 and~10]{Spanier}.
The same calculation applies to $\ell^1$-chains and bounded
cochains because the Eilenberg--Zilber homotopies are bounded
in every degree; see
\cite[Lemma~2.6 and the proof of Proposition~2.5]{Loh}.
In particular, evaluating the homotopy on the fixed cycle $Z$
makes the cochain error a bounded coboundary. Thus the identity
holds in bounded cohomology and passes to reduced bounded
cohomology. 
\end{proof}

\begin{proposition}\label{prop:symmetric-products}
Let $\theta\in\rHb^2(\F)$ and $z\in Z_2^{\ell^1}(\F)$ satisfy
$\langle\theta,z\rangle=1$. For $m\ge0$, $q\in\{2,3\}$, put $k=m+1$ and
$d=2m+q$, and define
\begin{equation}\label{eq:T}
 \begin{aligned}
 &T_m^q:\rHb^q(\F)\longrightarrow\rHb^d(\F^k),\\
 &T_m^q(v)=\Sym_k(\theta^{\times m}\times v).
 \end{aligned}
\end{equation}
Each $T_m^q$ has a bounded left inverse. In particular it is bounded below, 
and its image is an infinite-dimensional space of symmetric classes, 
and each nonzero class in the image has positive norm. 
\end{proposition}

\begin{proof}
Cross-product independence via Kronecker pairings is used in
\cite[Proposition~3.3]{Loh}, and a graded-symmetric version occurs
in \cite[Lemma~4.1(ii)]{Nitsche}. We apply
Lemma~\ref{lem:crossdual} to construct a bounded left inverse
of $T_m^q$ for each $q\in\{2,3\}$. The construction uses only
the fixed degree-two cycle $z$ with $\langle\theta,z\rangle=1$.
In particular, for $q=3$ it does not require a degree-three
homology class detecting the class $v$.

Let $\sigma$ act on $\F^k$ by permuting its coordinates. A sum over all such pullbacks is invariant under every coordinate permutation, so every class in the image of $T_m^q$ is symmetric. It remains to prove injectivity with a lower norm estimate.

Put $Z_m=z^{\times m}$. Since an external product of cycles
is a cycle, we have $Z_m\in Z_{2m}^{\ell^1}(\F^m)$. 
Repeated application of \eqref{eq:shufflebound} shows that it has finite norm; more precisely
the cross-product norm estimate gives
$$\|Z_m\|_1=\|Z_{m-1}\times z\|_1\le\binom{2m}{2}\|Z_{m-1}\|_1\|z\|_1,$$
because $Z_{m-1}$ has degree $2m-2$ and $z$ has degree two.
Starting with $Z_1=z$, we obtain
$$
\|Z_m\|_1\le\left(\prod_{j=1}^{m}\binom{2j}{2}\right)\|z\|_1^m=
\left(\prod_{j=1}^{m}\frac{(2j)!}{(2j-2)!\,2!}\right)\|z\|_1^m=\frac{(2m)!}{(2!)^m}\|z\|_1^m.
$$
In particular, $Z_m$ has finite $\ell^1$-norm.

We next compute the pairing of $\theta^{\times m}$ with $Z_m$.
By the definition of $L_Z^*$, for every class $a$ and every
cycle $w$ of the appropriate degrees,
$$\langle a,Z\times w\rangle=\langle L_Z^*a,w\rangle.
$$
For $m\ge2$, apply Lemma~\ref{lem:crossdual} with
$K=\F^{m-1}$, $L=\F$, and $Z=Z_{m-1}$. Since both
$\theta^{\times(m-1)}$ and $Z_{m-1}$ have degree $2m-2$,
\eqref{eq:dualproduct} gives
$$
L_{Z_{m-1}}^*(\theta^{\times m})=\langle\theta^{\times(m-1)},Z_{m-1}\rangle\,\theta.
$$
As $Z_m=Z_{m-1}\times z$, it follows that
$$
\langle\theta^{\times m},Z_m\rangle=\langle L_{Z_{m-1}}^*(\theta^{\times m}),z\rangle
=\langle\theta^{\times(m-1)},Z_{m-1}\rangle\langle\theta,z\rangle.
$$
Starting with $Z_1=z$ and $\langle\theta,z\rangle=1$,
induction therefore gives
\begin{equation}\label{eq:productpairtheta}
 \langle\theta^{\times m},Z_m\rangle
 =\langle\theta,z\rangle^m=1.
\end{equation}
For $m=0$, the same identity holds with the degree-zero
unit $\theta^{\times0}=1$ and $Z_0=[\,]$.

Apply Lemma~\ref{lem:crossdual} with $K=\F^m$, $L=\F$, $p=2m$ and $Z=Z_m$. It gives the map
$$
 L_{Z_m}^{*}:\rHb^{2m+q}(\F^m\times\F)\longrightarrow \rHb^q(\F).
$$
There are $m!$ permutations in \eqref{eq:T} for each prescribed 
position of the variable class $v$: the other $m$ factors all equal $\theta$. Moving an even-degree factor 
past another factor produces no sign. Consequently these $m!$ 
copies of each term agree in reduced bounded cohomology. 
Graded commutativity is valid for bounded cohomology 
because the usual shuffle/flip homotopy is bounded in 
each degree; see \cite[Section~2.3]{Loh} and the 
bounded chain maps in Lemma~\ref{lem:crossdual}.

Suppose $q=3$. If $v$ is in the last coordinate, 
the first $m$ coordinates carry $\theta^{\times m}$, 
whose evaluation on $Z_m$ is one by \eqref{eq:productpairtheta}. 
Each of these $m!$ terms therefore gives $v$. 
If $v$ is in one of the first $m$ coordinates, 
the degree on those coordinates is 
$$2(m-1)+3=2m+1.$$
Since $Z_m$ has degree $2m$, the second line of \eqref{eq:dualproduct} gives zero. We obtain
\begin{equation}\label{eq:left3}
 L_{Z_m}^{*}T_m^3(v)=m!v.
\end{equation}
In particular, $m!^{-1}L_{Z_m}^{*}$ is a left inverse of $T_m^3$.

Suppose $q=2$. Evaluation on the fixed cycle $z$ defines
a bounded linear functional on $\rHb^2(\F)$, since $|\langle v,z\rangle|\le\|v\|\|z\|_1$.
As before, the $m!$ terms with $v$ last contribute $m!v$. If $v$ is in the $j$th of the first $m$ coordinates, the product pairing on these coordinates is $\langle v,z\rangle$: all the other evaluations are $\langle\theta,z\rangle=1$. The last coordinate contains $\theta$. There are $m!$ terms for each of these $m$ choices, so
\begin{equation}\label{eq:left2}
 L_{Z_m}^{*}T_m^2(v)=m!\bigl(v+m\langle v,z\rangle\theta\bigr).
\end{equation}
Define the following bounded linear maps on $\rHb^{2m+q}(\F^{m+1})$, for $q=3$ and $q=2$, respectively:
\begin{equation}\label{eq:explicitQ}
Q_{m,3}(\eta)=\frac1{m!}L_{Z_m}^{*}(\eta),\quad
Q_{m,2}(\eta)=\frac1{m!}\left(L_{Z_m}^{*}(\eta)
-\frac{m}{m+1}\bigl\langle L_{Z_m}^{*}(\eta),z\bigr\rangle\theta\right).
\end{equation}
Equation~\eqref{eq:left3} gives $Q_{m,3}T_m^3=\id$. To check the second identity, pair \eqref{eq:left2} with $z$. Since $\langle\theta,z\rangle=1$, this gives
$$
\bigl\langle L_{Z_m}^{*}T_m^2(v),z\bigr\rangle=m!(m+1)\langle v,z\rangle.
$$
Substituting this and \eqref{eq:left2} into \eqref{eq:explicitQ} yields
$$
Q_{m,2}T_m^2(v)=v+m\langle v,z\rangle\theta-m\langle v,z\rangle\theta=v.
$$
Thus $Q_{m,q}T_m^q=\id$. Both maps are bounded by Lemma~\ref{lem:crossdual} 
and the bound on the pairing with $z$. For $m=0$ they are the identity maps. For example,
$$
\|Q_{m,2}\|\le\frac1{m!}\left(1+\frac m{m+1}\|z\|_1\|\theta\|\right)\binom{2m+2}{2m}\|Z_m\|_1,
$$
and Lemma~\ref{lem:crossdual} gives the corresponding bound for $Q_{m,3}$.

It follows that $\|T_m^q(v)\|\ge\frac1{\|Q_{m,q}\|}\|v\|$.
Thus $T_m^q$ is injective and satisfies the asserted lower norm estimate.

The space $\rHb^2(\F)$ is infinite-dimensional by \cite{Brooks,Mitsumatsu}. For $\rHb^3(\F)$ we use \cite[Corollary~3]{FPS} (see \cite{Soma} as well): the third reduced bounded cohomology of an acylindrically hyperbolic group has continuum dimension. It applies to $\F$, whose action on its Cayley tree is acylindrical and non-elementary. Another construction of positive-norm degree-three classes is given in \cite[Sections~6--7]{Farre}. Thus both spaces $\rHb^2(\F)$ and $\rHb^3(\F)$, and hence all their images under $T_m^q$, are infinite-dimensional. This proves the proposition.
\end{proof}


\section{Proof of Theorem \ref{thm:main}}
\label{sec:3}

\begin{lemma}\label{lem:transfer}
For each $k\ge1$ and each transformation group $\cG$ 
in Theorem~\ref{thm:main} or~\ref{thm:surface}, the diagonal action on $M^k$ defines compatible maps on reduced bounded cohomology and ordinary cohomology:
$$
 \overline\Gamma_{b,k}^{\,d}:\rHb^d(G^k)\longrightarrow\rHb^d(\cG),
 \qquad
 \Gamma_k^d:H^d(G^k;\R)\longrightarrow H^d(\cG;\R),
$$
with $\norm{\overline\Gamma_{b,k}^{\,d}}\le1$. The corresponding map on bounded cohomology is denoted by $\Gamma_{b,k}^d$. The maps commute with comparison. In particular, the reduced bounded-cohomology transfer preserves reduced exact bounded cohomology:
$$
 \overline\Gamma_{b,k}^{\,d}\bigl(\rEH^d(G^k)\bigr)\subset\rEH^d(\cG).
$$
\end{lemma}
\begin{proof}
This is the cocycle induction of \cite[Sections~1.B and~2]{BM}, applied to the diagonal action on $M^k$. Its covering-space formulation and compatibility with ordinary cohomology are given in \cite[Section~3]{Nitsche}; see also \cite{Polterovich,GG}. We specify our conventions and check the finite-image hypothesis for the chosen section. Group multiplication is composition, $fg=f\circ g$.

Following \cite[Section~1.A]{BM}, we call a function $c:K^{d+1}\to\R$
\emph{homogeneous} if
$$c(g_0h,\ldots,g_dh)=c(g_0,\ldots,g_d)\qquad(g_0,\ldots,g_d,h\in K).
$$
Its coboundary is the alternating sum obtained by deleting one argument.
This complex is isometrically identified with the inhomogeneous complex
used in Lemma~\ref{lem:crossdual}: the homogeneous cochain corresponding
to an inhomogeneous cochain $f$ is
$$c(g_0,\ldots,g_d)=f(g_0g_1^{-1},g_1g_2^{-1},\ldots,g_{d-1}g_d^{-1}).$$
Fix an auxiliary complete Riemannian metric on $\Int M$ and
a basepoint $x_0\in\Int M$. 
Outside the cut locus of $x_0$, let $\alpha_{x_0,x}$ be the
unique minimizing geodesic from $x_0$ to $x$. On the cut locus,
which has $\mu$-measure zero, choose the reference paths
measurably. Put $\alpha_{x,x_0}=\alpha_{x_0,x}^{-1}$.
This gives the measurable family of reference paths used in
\cite[Section~2.B]{BM}.

For $f\in\cG$, choose an isotopy $\{f_t\}_{t\in[0,1]}$ from
the identity to $f$ and define
\begin{equation}\label{eq:cocycle}
\gamma(f,x)=\bigl[\alpha_{x_0,x}*\{f_t(x)\}_{t\in[0,1]}*\alpha_{f(x),x_0}\bigr]^{-1}\in G.
\end{equation}
Here $*$ denotes path concatenation from left to right, and
the brackets denote the homotopy class projected to
$G=\pi_1(M,x_0)/Z(\pi_1(M,x_0))$. This class is independent of
the chosen isotopy, since evaluation of a loop of isotopies
has image in the center of $\pi_1(M,x_0)$; see
\cite[Proposition~2.1]{BM}. The resulting measurable cocycle
satisfies
\begin{equation}\label{eq:cocycleidentity}
 \gamma(fg,x)=\gamma(f,gx)\gamma(g,x)
\end{equation}
for every $f,g\in\cG$ and almost every $x\in M$.

By \cite[Lemma~2.2]{BM}, for each $f\in\cG$ there is
a finite set $S_f\subset G$ such that
\begin{equation}\label{eq:finiteimage}
\gamma(f,x)\in S_f\qquad\text{for almost every }x\in M.
\end{equation}

For $\mathbf x=(x_1,\ldots,x_k)$ set
$$\gamma_k(f,\mathbf x)=\bigl(\gamma(f,x_1),\ldots,\gamma(f,x_k)\bigr)\in G^k.$$
This is a cocycle for the diagonal action on $M^k$, by \eqref{eq:cocycleidentity}. Moreover,
$$\{\gamma_k(f,\mathbf x):\mathbf x\in M^k\}\subset S_f^k,$$
so $\gamma_k(f,\cdot)$ has finite image for every fixed $f$.

For a homogeneous bounded $d$-cochain $c$ on $G^k$, define $\Gamma'_{b,k}c$ by
\begin{equation}\label{eq:transferformula}
 (\Gamma'_{b,k}c)(f_0,\ldots,f_d)
 =\int_{M^k}
 c\bigl(\gamma_k(f_0,\mathbf x),\ldots,
             \gamma_k(f_d,\mathbf x)\bigr)\,d\mu^k(\mathbf x).
\end{equation}
The integrand is measurable and bounded. Since $\mu(M)=1$, we have
$\|\Gamma'_{b,k}c\|_\infty\le\|c\|_\infty$.
By the standard cocycle-induction calculation 
\cite[Proposition~1.1 and Section~2.C]{BM}, this is a homogeneous cochain map and
$$
 \delta\Gamma'_{b,k}=\Gamma'_{b,k}\delta.
$$
Indeed, the cocycle identity and the change of variable $\mathbf x\mapsto h\mathbf x$ give right homogeneity, 
and the differential commutes with integration. 
The norm bound follows from $\mu^k(M^k)=1$. 
Therefore the maps $\Gamma_{b,k}^d$ and 
$\overline\Gamma_{b,k}^{\,d}$ are defined and have norm at most one.

Changing the measurable reference paths does not change the
induced map on bounded cohomology, by the bounded
change-of-section homotopy; see
\cite[Lemma~3.3 and Remark~3.4]{Nitsche}. 
This proves that the induced map is independent of the section.

We next define the map on ordinary cohomology and explain 
why finite image is needed. Fix $(f_0,\ldots,f_d)$ and 
let $c:(G^k)^{d+1}\to\R$ be an ordinary homogeneous cochain, not necessarily bounded. By \eqref{eq:finiteimage}, the map
$$
\mathbf x\longmapsto\bigl(\gamma_k(f_0,\mathbf x),\ldots,\gamma_k(f_d,\mathbf x)\bigr)
$$
has image in the finite set $\mathcal A= S_{f_0}^k\times\cdots\times S_{f_d}^k$.
For $\mathbf a=(a_0,\ldots,a_d)\in\mathcal A$, let
$$
E_{\mathbf a}=\{\mathbf x\in M^k:\gamma_k(f_i,\mathbf x)=a_i\text{ for }i=0,\ldots,d\}.
$$
These are measurable sets and form a finite partition after empty sets are omitted. Thus the integrand in \eqref{eq:transferformula} has finite image in $\R$, and
$$
(\Gamma'_k c)(f_0,\ldots,f_d)=\sum_{\mathbf a\in\mathcal A}\mu^k(E_{\mathbf a})\,c(a_0,\ldots,a_d).
$$
In particular, its absolute value is at most
$\max_{\mathbf a\in\mathcal A}|c(\mathbf a)|<\infty$, since $\mu^k(M^k)=1$. 
This is an argument for each fixed tuple; it does not assert that $\Gamma'_k c$ 
is a bounded cochain. The same homogeneity and differential calculations are 
valid because all the integrals involved are finite. They give the ordinary 
cochain map $\Gamma'_k$ and the induced map $\Gamma_k^d$. Since $\Gamma'_k$ 
agrees with $\Gamma'_{b,k}$ on bounded cochains, we have
$$\overline c_{\cG}^d\overline\Gamma_{b,k}^{\,d}=\Gamma_k^d\overline c_{G^k}^d.$$
In particular, if a bounded cocycle $c$ has an ordinary primitive $b$, then
$$\Gamma'_{b,k}c=\delta(\Gamma'_k b).$$
The cochain $\Gamma_k^{\prime}b$ is well-defined by the finite-image property. 
Therefore the induced map preserves reduced exact bounded cohomology.
\end{proof}

\begin{remark*}[Permutation invariance of product transfer]
\label{rem:transfer-symmetry}
The product transfer in Lemma~\ref{lem:transfer} is unchanged by
permuting the group coordinates. More generally, let $(X,\nu)$ be a
probability space with a measure-preserving action and a measurable
$K$-valued cocycle. Apply that same cocycle to each of $k$ independent
points of $X$, and let $\Gamma_{X,k}^d$ denote the resulting transfer
from $\rHb^d(K^k)$. Then
\begin{equation}\label{eq:transfer-symmetrization}
 \Gamma_{X,k}^d(\sigma^*\xi)=\Gamma_{X,k}^d(\xi),\qquad
 \Gamma_{X,k}^d(\xi)=\frac1{k!}\Gamma_{X,k}^d(\Sym_k\xi).
\end{equation}
This is an equality of cochain maps before passing to cohomology. Indeed,
for a bounded homogeneous cochain $c$ and a fixed transformation tuple
$\mathbf f=(f_0,\ldots,f_d)$, let $C_{\mathbf f}(\mathbf x)$ denote
its transfer integrand. The coordinatewise cocycle satisfies
$\gamma_k(f,\sigma\mathbf x)=\sigma\gamma_k(f,\mathbf x)$.
Consequently,
$$
(\Gamma'_{X,k}(\sigma^*c))(\mathbf f)=\int_{X^k}C_{\mathbf f}(\sigma\mathbf x)\,d\nu^k(\mathbf x)
=\int_{X^k}C_{\mathbf f}(\mathbf x)\,d\nu^k(\mathbf x)=(\Gamma'_{X,k}c)(\mathbf f).
$$
The second equality uses the measure-preserving permutation of the
integration variables. Summing over $S_k$ proves
\eqref{eq:transfer-symmetrization}. The argument uses boundedness and
finite measure, not compactness. It therefore applies to $M^k$,
$X_3(\mathbb D)^k$, $X_4(S^2)^k$, and $X_3(S)^k$ below.

We consistently transfer the normalized symmetric class
$k!^{-1}\Sym_k\xi$. This has exactly the same transfer as the ordered
class $\xi$. For products of two different kinds of integration spaces,
only equal factors are permuted. In particular, for
$X_3(S)^m\times S$ for $g\ge2$ and $X_3(T)^m\times X_2(T)$ for
$g=1$ in Theorem~\ref{thm:surface}, the first $m$ three-point factors
are symmetrized. The last one-point or two-point factor is kept fixed;
the normalization is $1/m!$.
\end{remark*}

\begin{lemma}\label{lem:push}
Assume $n\ge3$. Let $i:\F\hookrightarrow G$ be as above and fix $k\ge1$. For every sufficiently small $\eps>0$ there is a homomorphism
$$\rho_\eps:\F^k\longrightarrow\cG$$
and a measurable $\rho_\eps(\F^k)$-invariant subset $B_\eps\subset M$ such that
$$\delta_\eps=\mu(B_\eps)\longrightarrow0.$$
Every element of $\rho_\eps(\F^k)$ fixes $M\setminus B_\eps$ pointwise. 
Up to sets of measure zero, this complement has a measurable partition 
into subsets $U_{j,\eps}$, $V_{j,a,\eps}$, $V_{j,b,\eps}$, $1\le j\le k$, and $W$. 
For each point in these subsets the homomorphism 
$w\mapsto\gamma(\rho_\eps(w),x)$ is conjugate, respectively, to
$$
i\circ p_j,\qquad i\circ h_a\circ p_j,\qquad i\circ h_b\circ p_j,\qquad 1,
$$
where $p_j:\F^k\to\F$ is the projection and
$$h_a(a)=a,\quad h_a(b)=1,\qquad h_b(a)=1,\ h_b(b)=b.$$
Moreover, $\lambda_{j,\eps}=\mu(U_{j,\eps})\to\lambda_j>0$.

The homomorphisms can be chosen in $\Diff_0(M,\mu)$. 
In the symplectic case they can be chosen in $\Ham(M,\omega)$ itself, 
with Hamiltonian isotopies supported in the interior.
\end{lemma}

\begin{proof}
We use the finger-pushing construction of
\cite[Lemma~3.1]{BM}, performed in $k$ disjoint neighborhoods
as in \cite[proof of Theorem~4.2]{Nitsche}.
Since $n\ge3$, general position gives $k$ pairwise disjoint
figure-eight embeddings in $\Int M$, whose based loops
represent $i(a)^{-1}$ and $i(b)^{-1}$ in $G$.
The inverses agree with the cocycle convention
\eqref{eq:cocycle}. Choose tubular neighborhoods
$N_{j,a},N_{j,b}$ with compact closures in $\Int M$,
such that the unions $N_{j,a}\cup N_{j,b}$ are pairwise
disjoint and each intersection $N_{j,a}\cap N_{j,b}$
is contractible and has positive measure.
See Figure~\ref{fig:eight}.

\begin{figure}[tbp]
\centering
\begin{adjustbox}{max width=\linewidth}
\includegraphics{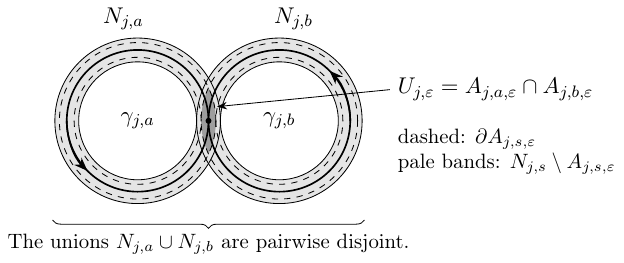}
\end{adjustbox}
\caption{Tubular neighborhoods of a figure-eight embedding.
Their contractible intersection contains the points whose
trajectories record both generators, with the same
conjugating element. The normal directions are omitted.}
\label{fig:eight}
\end{figure}

In product-volume coordinates
$N=S^1\times D^{n-1}$, with $S^1=\R/\Z$,
choose compact transverse cores $K_\eps$ exhausting
$D^{n-1}$ and smooth compactly supported cutoffs
$\chi_\eps$ equal to one near $K_\eps$. Define
\begin{equation}\label{eq:volpush}
 P_\eps^t(\vartheta,y)=(\vartheta+t\chi_\eps(y),y),\qquad 0\le t\le1.
\end{equation}
This isotopy preserves volume and extends by the identity
outside $N$. On $A_\eps=S^1\times K_\eps$, every trajectory
makes one positive circuit, so the time-one map is the
identity there. Moreover, $\mu(N\setminus A_\eps)\to0$.

For symplectic $M$ of dimension $2\ell$, use the standard
isotropic tubular model
$$
N=S^1\times(-r,r)_p\times D^{2\ell-2}_z,\qquad \omega=d\vartheta\wedge dp+\omega_0;
$$
see \cite{Weinstein} and \cite[Section~4.3]{BK}.
Choose $\chi_\eps(p,z)$ compactly supported in the transverse
domain and equal to one near compact cores $K_\eps$
exhausting that domain. Set
\begin{equation}\label{eq:Ham}
H_\eps(\vartheta,p,z)=-p\chi_\eps(p,z),\qquad \iota_{X_{H_\eps}}\omega=-dH_\eps.
\end{equation}
On $A_\eps=S^1\times K_\eps$, the vector field is
$\partial_\vartheta$, so its flow has the same full-circuit
property. Extension by zero gives a Hamiltonian isotopy
on $M$. All cutoff effects lie in $N\setminus A_\eps$,
whose measure tends to zero. See Figure~\ref{fig:annulus}.

\begin{figure}[tbp]
\centering
\begin{adjustbox}{max width=\linewidth}
\includegraphics{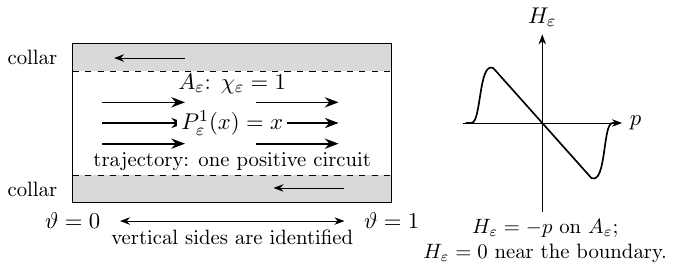}
\end{adjustbox}
\caption{The finger-pushing isotopy on an annular slice
and its Hamiltonian. The vertical sides are identified.
Every point of $A_\eps$ returns at time one after one
positive circuit. The transition region
$N\setminus A_\eps$ has measure tending to zero.}
\label{fig:annulus}
\end{figure}

We apply the appropriate construction in every tube.
Write $A_{j,s,\eps}$ for its core and $f_{j,s,\eps}$
for its time-one map, where $s\in\{a,b\}$.
Now, let 
$$\rho_{j,\eps}:\F\to\cG$$ 
send $a,b$ to $f_{j,a,\eps},f_{j,b,\eps}$. Disjoint supports imply that
$$
\rho_\eps(w_1,\ldots,w_k)=\rho_{1,\eps}(w_1)\cdots\rho_{k,\eps}(w_k)
$$
is a homomorphism. Set
$$
B_\eps=\bigcup_{j=1}^k\bigcup_{s\in\{a,b\}}(N_{j,s}\setminus A_{j,s,\eps}).
$$
Then $\mu(B_\eps)\to0$. Every generator fixes
$M\setminus B_\eps$ pointwise, so $B_\eps$ is invariant.
Up to boundaries of measure zero, its complement is
partitioned by
$$
 \begin{aligned}
 &U_{j,\eps}=A_{j,a,\eps}\cap A_{j,b,\eps},\quad V_{j,a,\eps}=A_{j,a,\eps}\setminus N_{j,b},\quad
 V_{j,b,\eps}=A_{j,b,\eps}\setminus N_{j,a},\\
 &W=M\setminus\bigcup_{j=1}^k(N_{j,a}\cup N_{j,b}).
 \end{aligned}
$$

For $x\in U_{j,\eps}$, use a path $c_x$ from the common vertex
to $x$ inside the contractible tube intersection to identify
both generator trajectories. Hence the same conjugating
element $u_x\in G$ works for both generators. Since all
time-one maps fix $x$, the cocycle identity gives
$$
\gamma(\rho_\eps(w),x)=u_xi(w_j)u_x^{-1}\qquad(w=(w_1,\ldots,w_k)\in\F^k).
$$
On $V_{j,a,\eps}$ only the $a$-isotopy moves $x$,
giving a conjugate of $i\circ h_a\circ p_j$.
Similarly, on $V_{j,b,\eps}$ we obtain
$i\circ h_b\circ p_j$, and on $W$ the cocycle is trivial.
The conjugators can be chosen measurably by enumerating
the countable group $G$ and selecting the first element
that gives the required identities on the generators.

Finally, exhaustion of the tubes by their cores gives
$$
\mu(U_{j,\eps})\longrightarrow\mu(N_{j,a}\cap N_{j,b})=\lambda_j>0.
$$
All isotopies have compact support in $\Int M$.
The volume-preserving construction takes values in
$\Diff_0(M,\mu)$ and hence in $\Homeo_0(M,\mu)$;
the Hamiltonian construction takes values in
$\Ham(M,\omega)$ and hence in $\Symp_0(M,\omega)$.
\end{proof}

We recall the notation needed for the next proposition. Let $\cG$ be one
of the transformation groups in Theorem~\ref{thm:main}. Fix the embedding
$i:\F\hookrightarrow G$ and the linear maps
$$
J_s:\rHb^s(\F)\longrightarrow\rHb^s(G),\qquad i^*J_s=\id\quad(s=2,3),
$$
chosen in Lemma~\ref{lem:exact} or Lemma~\ref{lem:othersections}.
Let $\theta\in\rHb^2(\F)$ be the fixed nonzero finite linear combination
of Brooks classes from Lemma~\ref{lem:brooks}. Recall that
$$
J_2(\theta)\in\rEH^2(G),\qquad\theta\smile u=0\quad\bigl(u\in\rHb^s(\F),\ s\ge1\bigr).
$$

Fix integers $m\ge0$ and $q\in\{2,3\}$, and put $k=m+1$ and $d=2m+q$.
Applying $J_2$ and $J_q$ to the individual factors in
Proposition~\ref{prop:symmetric-products} defines the linear map
$$
 \begin{aligned}
 &\widetilde T_m^q:\rHb^q(\F)\longrightarrow\rHb^d(G^k),\\
 &\widetilde T_m^q(v)=
 \Sym_k\bigl(J_2(\theta)^{\times m}\times J_q(v)\bigr).
 \end{aligned}
$$
Here $J_2(\theta)^{\times m}$ is the external product of $m$ copies of
$J_2(\theta)$. The operator $\Sym_k$ is the unnormalized sum of coordinate
pullbacks fixed above. If $m=0$, then $\widetilde T_0^q=J_q$.
Write $i^k:\F^k\hookrightarrow G^k$ for the product embedding
$i^k(w_1,\ldots,w_k)=(i(w_1),\ldots,i(w_k))$, so $i^2=i\times i$.
Naturality gives
$$
(i^k)^*\widetilde T_m^q(v)=\Sym_k(\theta^{\times m}\times v)=T_m^q(v),
$$
where $T_m^q$ is the map of Proposition~\ref{prop:symmetric-products}.

Fix the homomorphisms
$\rho_\eps:\F^k\to\cG$ and the subsets $B_\eps,U_{j,\eps}$ supplied by Lemma~\ref{lem:push}.
Recall that $\delta_\eps=\mu(B_\eps)\to0$ and
$\lambda_{j,\eps}=\mu(U_{j,\eps})\to\lambda_j>0$. Put
\begin{equation}\label{eq:pushconstants}
 c_{k,\eps}=\prod_{j=1}^k\lambda_{j,\eps},\qquad C_k=\prod_{j=1}^k\lambda_j>0.
\end{equation}
Thus $c_{k,\eps}\to C_k$. We use the product transfer
$\overline\Gamma_{b,k}^{\,d}:\rHb^d(G^k)\to\rHb^d(\cG)$
of Lemma~\ref{lem:transfer}.

\begin{proposition}\label{prop:approx}
With the notation above, for every $\xi\in\im\widetilde T_m^q$ and every
sufficiently small $\eps>0$,
\begin{equation}\label{eq:approx}
\norm{\rho_\eps^*\overline\Gamma_{b,k}^{\,d}(\xi)-k!C_k(i^k)^*\xi}\le
\left(k\delta_\eps+k!|c_{k,\eps}-C_k|\right)\norm{\xi},
\end{equation}
where $c_{k,\eps}$ and $C_k$ are defined in \eqref{eq:pushconstants}.
In particular, the left-hand side tends to zero for each fixed $\xi$.
\end{proposition}
\begin{proof}
We compute the pullback directly from the transfer formula
\eqref{eq:transferformula}, using the partition and the
associated homomorphisms of Lemma~\ref{lem:push}.
We first treat two coordinates, identify the nonzero
contributions, and estimate the integral over the exceptional
set. We then give the argument for arbitrary $k$.

\emph{The case $k=2$.}
Here $m=1$, $d=2+q$, and $q\in\{2,3\}$. Let
$$
v\in\rHb^q(\F),\quad\xi=\widetilde T_1^q(v)=J_2(\theta)\times J_q(v)+J_q(v)\times J_2(\theta).
$$
Recall that restriction along $i\times i$ is the map
$$(i\times i)^*:\rHb^d(G\times G)\longrightarrow\rHb^d(\F\times\F).$$
The maps $J_2,J_3$ chosen in Lemma~\ref{lem:exact} or Lemma~\ref{lem:othersections}
are right inverses to restriction along $i$:
$$i^*J_s=\id\quad(s=2,3),$$
as recorded in \eqref{eq:commondata}. In particular, $i^*J_2(\theta)=\theta$ and $i^*J_q(v)=v$. 
Naturality of the external product therefore gives
$$
(i\times i)^*\xi=(i^*J_2(\theta))\times(i^*J_q(v))+(i^*J_q(v))\times(i^*J_2(\theta))=\theta\times v+v\times\theta.
$$
Let $\tau:\F\times\F\to\F\times\F$ exchange the factors. 
Then $\tau^*\bigl((i\times i)^*\xi\bigr)=(i\times i)^*\xi$; 
there is no sign when the factors are exchanged, 
since the degree of $\theta$ is two. For $k=2$, 
the homomorphism $\rho_\eps$ of Lemma~\ref{lem:push} is
$$
\rho_\eps:\F\times\F\longrightarrow\cG,\qquad\rho_\eps(g,h)=\rho_{1,\eps}(g)\rho_{2,\eps}(h).
$$
Thus $g$ determines the pushes in $N_{1,a}\cup N_{1,b}$, and $h$
determines the pushes in the disjoint union $N_{2,a}\cup N_{2,b}$. Set
$$\mathcal E_\eps=(B_\eps\times M)\cup(M\times B_\eps).$$
The set $\mathcal E_\eps$ is invariant under the diagonal 
action of $\rho_\eps(\F^2)$, since $B_\eps$ is invariant. As $\mu(M)=1$,
$$\mu^2(\mathcal E_\eps)=2\delta_\eps-\delta_\eps^2\le2\delta_\eps.$$
For a homogeneous bounded $d$-cochain $c$ on $G^2$, denote the integral over this subset by
$$
(\mathcal R_\eps c)(w_0,\ldots,w_d)=\int_{\mathcal E_\eps}c\bigl(\gamma_2(\rho_\eps(w_0),\mathbf x),
\ldots,\gamma_2(\rho_\eps(w_d),\mathbf x)\bigr)\,d\mu^2(\mathbf x).
$$
Here $w_0,\ldots,w_d\in\F^2$. The cocycle identity \eqref{eq:cocycleidentity}, 
followed by the measure-preserving change of variable 
$\mathbf x\mapsto\rho_\eps(w)\mathbf x$ for fixed $w\in\F^2$, 
proves right invariance of this cochain. 
The change of variable preserves $\mathcal E_\eps$. 
The differential commutes with the integral, so $\mathcal R_\eps$ is a cochain map. Moreover,
$$\|\mathcal R_\eps c\|_\infty\le2\delta_\eps\|c\|_\infty.$$
It therefore induces a map, denoted by the same symbol, on reduced bounded cohomology, and
$$\|\mathcal R_\eps(\xi)\|\le2\delta_\eps\|\xi\|.$$
Thus all pairs with a coordinate in $B_\eps$ are accounted for by this error term.

Up to a null set, the complement of $\mathcal E_\eps$ 
is partitioned into products $C=X\times Y$, where $X,Y$ belong to
$\{U_{1,\eps},U_{2,\eps},V_{1,a,\eps},V_{1,b,\eps},V_{2,a,\eps},V_{2,b,\eps},W\}$.
Every time-one map fixes every point in $C$. 
It follows that for $(x,y)\in C$, the map
$$w\longmapsto\gamma_2(\rho_\eps(w),(x,y))$$
is a homomorphism. By Lemma~\ref{lem:push}, it has the form
$$w\longmapsto u_{x,y}\varphi_C(w)u_{x,y}^{-1},$$
where $\varphi_C:\F^2\to G^2$ is independent of $(x,y)\in C$, and $u_{x,y}$ is measurable.

Inner automorphisms induce the identity on bounded cohomology,
through a cochain homotopy bounded independently of the
conjugating element; see \cite[Lemma~A.2]{Loh}.
Integrating this homotopy shows that the contribution of $C$
is $\mu^2(C)\varphi_C^*$, also on reduced bounded cohomology.
We therefore omit the conjugations in the following calculations.

If $X$ or $Y$ is one of the subsets $V_{j,a,\eps}$, $V_{j,b,\eps}$, or $W$, 
the corresponding coordinate of $\varphi_C$ 
factors through a cyclic or trivial group. For any external product of positive-degree classes,
$$
\varphi_C^*(\widetilde\alpha\times\widetilde\beta)=
\varphi_{C,1}^*\widetilde\alpha\smile\varphi_{C,2}^*\widetilde\beta,
$$
where $\varphi_{C,1},\varphi_{C,2}:\F^2\to G$ are its coordinate homomorphisms. 
The factor corresponding to that cyclic or trivial coordinate is zero, 
since the positive-degree bounded cohomology of an 
amenable group vanishes \cite[Corollary~3.7]{Frigerio}. 
Both summands of $\xi$ are such external products. 
Hence all these subsets contribute zero to the pullback of $\xi$. 
The vanishing is in reduced bounded cohomology.

It remains to compute the four products
$$
 U_{1,\eps}\times U_{2,\eps},\quad
 U_{2,\eps}\times U_{1,\eps},\quad
 U_{1,\eps}\times U_{1,\eps},\quad
 U_{2,\eps}\times U_{2,\eps}.
$$
For this calculation let $\widetilde\alpha\in\rHb^r(G)$ and
$\widetilde\beta\in\rHb^s(G)$, with $r,s>0$ and $r+s=d$, and put
$$\alpha=i^*\widetilde\alpha,\qquad\beta=i^*\widetilde\beta.
$$
If $x\in U_{1,\eps}$ and $y\in U_{2,\eps}$, 
the two associated homomorphisms depend on $g$ and $h$ separately. The associated homomorphism is
$$
\varphi_{12}:\F\times\F\longrightarrow G\times G,\qquad \varphi_{12}(g,h)=(i(g),i(h)),
$$
and therefore
$$
\varphi_{12}^*(\widetilde\alpha\times\widetilde\beta)
=(i\times i)^*(\widetilde\alpha\times\widetilde\beta)=\alpha\times\beta.
$$
For $x\in U_{2,\eps}$ and $y\in U_{1,\eps}$, we have
$\varphi_{21}(g,h)=(i(h),i(g))=(i\times i)\tau(g,h)$, and so
$\varphi_{21}^*(\widetilde\alpha\times\widetilde\beta)=\tau^*(\alpha\times\beta)$.

Now suppose that $x,y\in U_{1,\eps}$. Both associated homomorphisms depend on $g$, while the pushes determined by $h$ are supported in the disjoint second neighborhood. Thus
$$
\varphi_{11}:\F\times\F\longrightarrow G\times G,\qquad \varphi_{11}(g,h)=(i(g),i(g)).
$$
Let $p_1,p_2:\F\times\F\to\F$ be the projections, and let
$\Delta:\F\to\F\times\F$ be given by $\Delta(g)=(g,g)$. Then
$\F\times\F\xrightarrow{\ p_1\ }\F\xrightarrow{\ \Delta\ }\F\times\F\xrightarrow{\ i\times i\ }G\times G$
is a factorization of $\varphi_{11}$. Since
$\alpha\times\beta=p_1^*\alpha\smile p_2^*\beta$ and
$p_1\Delta=p_2\Delta=\id_{\F}$, naturality gives
$\Delta^*(\alpha\times\beta)=(p_1\Delta)^*\alpha\smile(p_2\Delta)^*\beta=\alpha\smile\beta$.
Consequently,
$$
\varphi_{11}^*(\widetilde\alpha\times\widetilde\beta)=p_1^*\Delta^*(i\times i)^*
(\widetilde\alpha\times\widetilde\beta)=p_1^*\Delta^*(\alpha\times\beta)=p_1^*(\alpha\smile\beta).
$$
The same calculation on $U_{2,\eps}\times U_{2,\eps}$ gives $p_2^*(\alpha\smile\beta)$.

The measure of each of the first two products is
$c_{2,\eps}=\lambda_{1,\eps}\lambda_{2,\eps}$, 
while the measures of the last two are $\lambda_{1,\eps}^2$ and $\lambda_{2,\eps}^2$. 
Including the error term, the calculation gives
$$
\begin{aligned}
&\rho_\eps^*\overline\Gamma_{b,2}^{\,d}(\widetilde\alpha\times\widetilde\beta)\\
&=c_{2,\eps}\bigl(\alpha\times\beta+\tau^*(\alpha\times\beta)\bigr)\\
&\mathrel{+}\lambda_{1,\eps}^2p_1^*(\alpha\smile\beta)+\lambda_{2,\eps}^2p_2^*(\alpha\smile\beta)
\mathrel{+}\mathcal R_\eps(\widetilde\alpha\times\widetilde\beta).
\end{aligned}
$$
Figure~\ref{fig:cells} shows the four products and their pullbacks.

\begin{figure}[htbp]
\centering
\begin{adjustbox}{max width=\linewidth}
\includegraphics{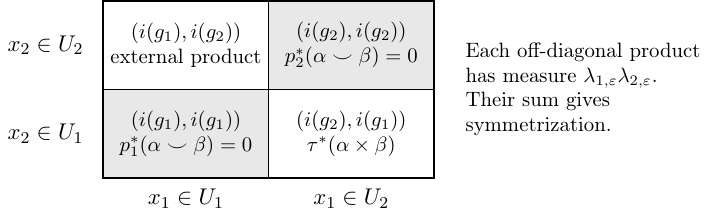}
\end{adjustbox}
\caption{The four products $U_{j,\eps}\times U_{l,\eps}$ for $k=2$. On the shaded products the coordinate homomorphism factors through the diagonal, giving $p_j^*(\alpha\smile\beta)$. This class is zero when one factor is $\theta$. The other two products give $\alpha\times\beta$ and $\tau^*(\alpha\times\beta)$, where $\tau$ exchanges the factors. Tuples with a coordinate in $B_\eps$ belong to $\mathcal E_\eps$.}
\label{fig:cells}
\end{figure}

Apply this identity to the two summands of $\xi$. Their restrictions are
$\theta\times v$ and $v\times\theta$. By Lemma~\ref{lem:brooks} and graded commutativity,
$$\theta\smile v=0,\qquad v\smile\theta=(-1)^{2q}\theta\smile v=0.$$
Hence the two products in which both points belong to the same $U_{j,\eps}$ contribute zero. 
The remaining two products give
$$
\rho_\eps^*\overline\Gamma_{b,2}^{\,d}(\xi)=c_{2,\eps}\bigl((i\times i)^*\xi+
\tau^*\bigl((i\times i)^*\xi\bigr)\bigr)+\mathcal R_\eps(\xi)=2c_{2,\eps}(i\times i)^*\xi+\mathcal R_\eps(\xi).
$$
Subtract $2C_2(i\times i)^*\xi$, where $C_2=\lambda_1\lambda_2>0$. 
Since restriction is norm non-increasing,
$$
\begin{aligned}
&\|\rho_\eps^*\overline\Gamma_{b,2}^{\,d}(\xi)-2C_2(i\times i)^*\xi\|\\
&\le2|c_{2,\eps}-C_2|\|(i\times i)^*\xi\|+\|\mathcal R_\eps(\xi)\|\\
&\le\bigl(2|c_{2,\eps}-C_2|+2\delta_\eps\bigr)\|\xi\|.
\end{aligned}
$$
As $\eps\to0$, the right-hand side tends to zero by Lemma~\ref{lem:push}. 
This proves \eqref{eq:approx} for $k=2$.

The same calculation also explains the injectivity argument. If
$\overline\Gamma_{b,2}^{\,d}(\xi)=0$, 
then the left-hand side of the last estimate equals
$2C_2\|(i\times i)^*\xi\|$ for every $\eps$. 
Thus $(i\times i)^*\xi=0$ in reduced bounded cohomology. But
$(i\times i)^*\xi=T_1^q(v)$, and Proposition~\ref{prop:symmetric-products} 
says that $T_1^q$ is injective. Hence $v=0$ and $\xi=0$. 
Thus $\overline\Gamma_{b,2}^{\,q+2}$ is injective on
$\im\widetilde T_1^q$ for $q=2,3$. 

\emph{The general case.}
Let $k=m+1$ be arbitrary and $\xi\in\im\widetilde T_m^q$. 
Take $\rho_\eps$, $B_\eps$, and $U_{j,\eps}$ from Lemma~\ref{lem:push}, as fixed before the proposition. Set
$$
\mathcal E_\eps=\{(x_1,\ldots,x_k)\in M^k: x_\ell\in B_\eps\text{ for some }\ell\}.
$$
This set is invariant, and the union bound gives
$$
\mu^k(\mathcal E_\eps)\le\sum_{\ell=1}^k\mu^k(M^{\ell-1}\times B_\eps\times M^{k-\ell})=k\delta_\eps.
$$
Define
$$
(\mathcal R_\eps c)(w_0,\ldots,w_d)=\int_{\mathcal E_\eps}c\bigl
(\gamma_k(\rho_\eps(w_0),\mathbf x),\ldots,\gamma_k(\rho_\eps(w_d),\mathbf x)\bigr)\,d\mu^k(\mathbf x).
$$
The cochain-map argument above uses only invariance 
of $\mathcal E_\eps$ and the cocycle identity, so it applies without change. 
The norm of this cochain map is at most $k\delta_\eps$, 
and its induced map on reduced bounded cohomology satisfies
\begin{equation}\label{eq:Rbound}
 \|\mathcal R_\eps(\xi)\|\le k\delta_\eps\|\xi\|.
\end{equation}

Outside $\mathcal E_\eps$ we partition $M^k$ into products $C$ 
of the sets in Lemma~\ref{lem:push}. On each such product the associated homomorphism has the form
$w\mapsto u_{\mathbf x}\varphi_C(w)u_{\mathbf x}^{-1}$, where
$\varphi_C:\F^k\to G^k$ is independent of $\mathbf x\in C$. 
This product contributes $\mu^k(C)\varphi_C^*\xi$. 
If one coordinate lies in $V_{j,a,\eps}$, $V_{j,b,\eps}$, or $W$, 
the corresponding positive-degree factor restricts through an amenable group and is zero. 
These contributions therefore vanish exactly as for two coordinates.

It remains to consider
$$
C_\tau=U_{\tau(1),\eps}\times\cdots\times U_{\tau(k),\eps},\qquad \tau:\{1,\ldots,k\}\to\{1,\ldots,k\}.
$$
The associated homomorphism is $\varphi_{C_\tau}=(i^k)\circ D_\tau$, where
$$D_\tau(w_1,\ldots,w_k)=(w_{\tau(1)},\ldots,w_{\tau(k)}).$$
Thus its contribution is $\mu^k(C_\tau)D_\tau^*\bigl((i^k)^*\xi\bigr)$. 
For positive-degree classes $u_1,\ldots,u_k$ on $\F$, the definition of the external product and naturality give
\begin{equation}\label{eq:diagonal}
 D_\tau^*(u_1\times\cdots\times u_k)=D_\tau^*(p_1^*u_1\smile\cdots\smile p_k^*u_k)=
 p_{\tau(1)}^*u_1\smile\cdots\smile p_{\tau(k)}^*u_k.
\end{equation}
This is the general form of the diagonal calculation above: 
all factors assigned to one $U_{j,\eps}$ become an internal cup product on the $j$th free-group factor.

We claim that $D_\tau^*\bigl((i^k)^*\xi\bigr)=0$ whenever $\tau$ is not injective. 
Each summand of $(i^k)^*\xi$ contains $m$ copies of $\theta$ and one class $v\in \rHb^q(\F)$. 
Group the factors with the same value of $\tau(\ell)$ in \eqref{eq:diagonal}. 
Graded commutativity permits this regrouping without a sign, 
since all degrees are even except possibly for the single degree-three factor. 
A group of indices containing at least two elements contains a $\theta$,
because at most one factor differs from $\theta$. 
After reordering, its internal product is $\theta\smile\zeta$ 
for a positive-degree class $\zeta$, and hence is zero by Lemma~\ref{lem:brooks}. 
Every summand therefore vanishes, 
which proves the claim. As in the two-coordinate case, 
this argument uses cup-product vanishing on $\F$.

If $\tau$ is injective, it is a permutation. Moreover,
$$
\mu^k(C_\tau)=\lambda_{\tau(1),\eps}\cdots\lambda_{\tau(k),\eps}=c_{k,\eps},
$$
independently of the permutation. Summing these $k!$ contributions and adding the error term gives
\begin{equation}\label{eq:exactcell}
 \rho_\eps^*\overline\Gamma_{b,k}^{\,d}(\xi)=c_{k,\eps}\Sym_k\bigl((i^k)^*\xi\bigr)+\mathcal R_\eps(\xi)
 \quad\text{in }\rHb^d(\F^k).
\end{equation}

By construction $(i^k)^*\xi$ is symmetric, 
so $\Sym_k\bigl((i^k)^*\xi\bigr)=k!(i^k)^*\xi$. 
Subtract $k!C_k(i^k)^*\xi$ from \eqref{eq:exactcell}. 
Using \eqref{eq:Rbound} and that restriction is norm non-increasing, we obtain
$$
\begin{aligned}
 &\|\rho_\eps^*\overline\Gamma_{b,k}^{\,d}(\xi)-k!C_k(i^k)^*\xi\|\le k!|c_{k,\eps}-C_k|\|(i^k)^*\xi\|+\|\mathcal R_\eps(\xi)\|\\
 &\le\bigl(k!|c_{k,\eps}-C_k|+k\delta_\eps\bigr)\|\xi\|.
 \end{aligned}
$$
Finally, Lemma~\ref{lem:push} gives $\delta_\eps\to0$ and
$c_{k,\eps}=\prod_j\lambda_{j,\eps}\to\prod_j\lambda_j=C_k>0$. This proves the proposition.
\end{proof}

\begin{proof}[Proof of Theorem~\ref{thm:main}]
We use the restriction-and-limit argument of \cite[proofs of Theorems~A and~B]{BM}; see also \cite[Lemma~4.1(i)]{Nitsche}.
Choose $i,J_2,J_3$ from Lemma~\ref{lem:exact} or Lemma~\ref{lem:othersections},
according to the hypothesis being used, and put
$K_2=\ker(\overline c_G^2\circ J_2)$ as in \eqref{eq:commondata}.
Choose $\theta,z$ from Lemma~\ref{lem:brooks}.

We first define the maps asserted in part~\textup{(b)}.
For $m\ge0$ and $q\in\{2,3\}$, put $d=2m+q$ and $k=m+1$, and set
$$
\Phi_{m,q}(v):=\frac1{k!}\overline\Gamma_{b,k}^{\,d}\bigl(\widetilde T_m^q(v)\bigr).
$$
The domain is $K_2$ when $(m,q)=(0,2)$ and $\rHb^q(\F)$ in every other case.
In particular, the two maps in the initial degrees are
$$
\Phi_{0,2}=\bigl(\overline\Gamma_{b,1}^{\,2}\circ J_2\bigr)\big|_{K_2},\qquad
\Phi_{0,3}=\overline\Gamma_{b,1}^{\,3}\circ J_3.
$$
For $m\ge1$, the definition is equivalently
$$
\Phi_{m,q}(v)=\frac1{(m+1)!}\overline\Gamma_{b,m+1}^{\,2m+q}\Bigl(\operatorname{Sym}_{m+1}
\bigl(J_2(\theta)^{\times m}\times J_q(v)\bigr)\Bigr).
$$
By \eqref{eq:transfer-symmetrization}, this is also the transfer of
the ordered product $J_2(\theta)^{\times m}\times J_q(v)$.
Values of these maps lie in
$\rHb^{2m+q}(\cG)$; we will verify that their images are exact.
Using $i^*J_q=\id$ from \eqref{eq:commondata} and naturality of the external product, we obtain
$$
(i^k)^*\bigl(\widetilde T_m^q(v)\bigr)=\Sym_k(\theta^{\times m}\times v)=T_m^q(v).
$$
Since $T_m^q$ is injective, $\widetilde T_m^q$ is injective.
Moreover, restriction $(i^k)^*$ is injective on $\im\widetilde T_m^q$.

We first prove injectivity of the transfer on that image. Let
$\xi\in\im\widetilde T_m^q$ and suppose $\overline\Gamma_{b,k}^{\,d}(\xi)=0$.
Then $\rho_\eps^*\overline\Gamma_{b,k}^{\,d}(\xi)=0$ for every $\eps$. By Proposition~\ref{prop:approx},
$$
k!C_k\|(i^k)^*\xi\|=\|\rho_\eps^*\overline\Gamma_{b,k}^{\,d}(\xi)-k!C_k(i^k)^*\xi\|\longrightarrow0.
$$
The number on the left is independent of $\eps$, and $C_k>0$. 
Thus $(i^k)^*\xi$ has norm zero. Since $(i^k)^*\xi$ is a class in reduced bounded cohomology, 
it is zero. Injectivity of $(i^k)^*$ on $\im\widetilde T_m^q$ gives $\xi=0$. 
It follows that $\Phi_{m,q}$ is injective. 

Pullback by $\rho_\eps$ is norm non-increasing, so
$$
\|\overline\Gamma_{b,k}^{\,d}(\xi)\|
\ge \lim_{\eps\to0}\|\rho_\eps^*\overline\Gamma_{b,k}^{\,d}(\xi)\|=k!C_k\|(i^k)^*\xi\|.
$$ 
Apply this to $\xi=\widetilde T_m^q(v)$, divide by $k!$, 
and use the bounded left inverse $Q_{m,q}$ from \eqref{eq:explicitQ}. We obtain
$$
\|\Phi_{m,q}(v)\|\ge C_k\|T_m^q(v)\|\ge\frac{C_k}{\|Q_{m,q}\|}\|v\|.
$$

We next show that the images specified in the theorem lie in reduced exact bounded cohomology. 
If $m\ge1$, each summand of $\widetilde T_m^q(v)$ contains the factor $J_2(\theta)$, 
which has zero comparison image. By compatibility of comparison with external products, 
the comparison image of this summand is zero. 
Permutations are homomorphism pullbacks and preserve exactness; 
summing gives an ordinary primitive for the symmetric class. Thus
$$
\widetilde T_m^q(v)\in\rEH^{2m+q}(G^k)\quad(m\ge1,\ v\in \rHb^q(\F)).
$$
Its transfer belongs to reduced exact bounded cohomology by Lemma~\ref{lem:transfer}. 

For $m=0,q=3$, the section has $J_3\bigl(\rHb^3(\F)\bigr)\subset\rEH^3(G)$, 
so the same transfer argument proves exactness. For $m=0,q=2$, restrict to
$K_2=\ker(\overline c_G^2J_2)$. Under \textup{(i)} this has codimension 
at most one in the infinite-dimensional space $\rHb^2(\F)$ 
and is therefore infinite-dimensional; under \textup{(ii)} 
or \textup{(iii)} it is all of $\rHb^2(\F)$. This proves part~\textup{(b)}, 
with the domains and codomains stated in \eqref{eq:baseinj}--\eqref{eq:highinj}. 
The two degree patterns $2m+2$ and $2m+3$, $m\ge0$, cover every integer $d\ge2$. 
Part~\textup{(a)} now follows from injectivity and the infinite-dimensionality of $\rHb^2(\F)$ and $\rHb^3(\F)$.

Finally, for $\Diff_0(M,\mu)$ we use the volume-preserving isotopies
\eqref{eq:volpush}. These are also isotopies in $\Homeo_0(M,\mu)$.
For $\Ham(M,\omega)$ we use the Hamiltonians \eqref{eq:Ham}; their flows
are also isotopies in $\Symp_0(M,\omega)$. Each group preserves the
measure, so Lemma~\ref{lem:transfer} defines its transfer. The time-one
maps define homomorphisms for which Proposition~\ref{prop:approx}
gives the required pullback estimate. This completes the proof.
\end{proof}

\section{Proof of Theorem ~\ref{thm:disc}}
\label{sec:4}

\subsection{Disc case}
\label{sec:4-1}

We first prove the disc case of Theorem~\ref{thm:disc}; the sphere
proof follows it. The positive-genus surfaces are treated afterward in
Theorem~\ref{thm:surface}. In the disc proof, let $\mathcal T_{\mathbb D}$ denote any one
of its four transformation groups.
Multiplying the area form by a positive constant changes none of these groups,
so assume $\int_{\mathbb D}\omega=1$ and write $\mu$ for this area measure.
All configuration-space measures below are restrictions of product measures.

Write $B_3=\langle\sigma_1,\sigma_2\mid
\sigma_1\sigma_2\sigma_1=\sigma_2\sigma_1\sigma_2\rangle$ and let
$P_3\lhd B_3$ be the pure braid group. Put
$\Delta^2=(\sigma_1\sigma_2)^3$. The standard decomposition is
$$
P_3=\langle\sigma_1^2,\sigma_2^2\rangle\times\langle\Delta^2\rangle\cong\F\times\Z.
$$
Let $\jmath:\F\hookrightarrow B_3$ be the inclusion given by
$\jmath(a)=\sigma_1^2$ and $\jmath(b)=\sigma_2^2$.
These braid-group facts are recalled in \cite[Sections~2 and~4.1]{Kimura}.

\begin{lemma}\label{lem:braid-cohomology}
For $q=2,3$, restriction gives an isometric embedding
$$\jmath^*:\rHb^q(B_3)\longrightarrow\rHb^q(\F).$$
The domain $\rHb^q(B_3)$ is infinite-dimensional and equals $\rEH^q(B_3)$.
There are a fixed nonzero class $\alpha\in\rHb^2(B_3)$ and an
$\ell^1$-cycle $z\in Z_2^{\ell^1}(\F)$ with
\begin{equation}\label{eq:disc-normalization}
 \langle\jmath^*\alpha,z\rangle=1.
\end{equation}
\end{lemma}
\begin{proof}
We first prove infinite-dimensionality independently of the restriction
map. Let $p:B_3\to\operatorname{PSL}(2,\Z)$ denote the central quotient
homomorphism. Its kernel is $Z(B_3)=\langle\Delta^2\rangle\cong\Z$, and
$\operatorname{PSL}(2,\Z)\cong\Z/2*\Z/3$ is a non-elementary hyperbolic
group. Since the kernel is amenable, we obtain an
isometric isomorphism
$$
p^*:\rHb^q\bigl(\operatorname{PSL}(2,\Z)\bigr)\longrightarrow\rHb^q(B_3)\qquad(q\ge2);
$$
see \cite[Theorem~2.2]{Kimura} and \cite[Chapter~3]{Frigerio}.
The second bounded cohomology of this quotient is infinite-dimensional
by \cite{EpsteinFujiwara}; in degree two it coincides with reduced bounded
cohomology. In degree three, \cite[Corollary~3]{FPS} gives
$$\dim\rEH^3\bigl(\operatorname{PSL}(2,\Z)\bigr)=\infty.$$
Pullback preserves exactness by naturality of the comparison map.
The isometry of $p^*$ therefore implies that
$\rEH^3(B_3)$ is infinite-dimensional. In particular,
$\rHb^q(B_3)$ is infinite-dimensional for $q=2,3$.
See also \cite[proof of Corollary~1.2]{Kimura} for the degree-three
argument.

For exactness, recall Arnol'd's computation of the ordinary
cohomology of braid groups \cite{Arnold1970}. In particular, $H^r(B_3;\R)=0\thinspace (r>1)$.
For $B_3$, this also follows from the one-relator computation
in \cite[Chapter~II, Section~4, Example~3]{Brown}.
Consequently, the comparison map $c_{B_3}^q$ is zero for
every $q\ge2$, hence
$$EH_b^q(B_3)=H_b^q(B_3),\qquad\rEH^q(B_3)=\rHb^q(B_3)\qquad(q\ge2).$$

We now prove the assertion about restriction.
Restriction from $B_3$ to its finite-index subgroup $P_3$ is isometric onto
its image: averaging over the cosets is a norm non-increasing left inverse
to restriction. The reduced bounded cohomology of $P_3=\F\times\Z$ is identified
isometrically with that of $\F$, with the inverse given by restriction to
the free factor. Composing proves the assertion about $\jmath^*$; see
\cite[Theorems~2.1--2.2]{Kimura} and \cite[Chapter~3]{Frigerio}.
Restriction is not asserted to be onto the entire space $\rHb^q(\F)$.
Its image is infinite-dimensional by the first part of the proof.

Finally, choose $0\ne\alpha\in\rHb^2(B_3)$. Isometry makes
$\jmath^*\alpha$ nonzero. The degree-two duality used in
Lemma~\ref{lem:brooks} gives an $\ell^1$-cycle with nonzero evaluation.
Rescale it to obtain \eqref{eq:disc-normalization}.
\end{proof}

Let $X_3(\mathbb D)$ denote the space of ordered configurations of three points in $\Int\mathbb D$:
$$
X_3(\mathbb D)=\{(x_1,x_2,x_3)\in(\Int\mathbb D)^3:x_i\ne x_j\text{ for }i\ne j\}.
$$
It has measure one. A surface transformation acts on it diagonally, by
$f(x_1,x_2,x_3)=(f(x_1),f(x_2),f(x_3))$.
In the product $X_3(\mathbb D)^k$, points belonging to different triples are
not required to be distinct. In particular, we do not identify this space
with $X_{3k}(\mathbb D)$.

\begin{lemma}\label{lem:disctransfer}
For every $k\ge1$ and $d\ge0$, the braid cocycle for the diagonal 
action on $X_3(\mathbb D)^k$ induces the product Gambaudo--Ghys map
$$
\overline\Gamma_{b,\mathbb D,k}^{\,d}:\rHb^d(B_3^k)\longrightarrow\rHb^d(\mathcal T_{\mathbb D}),
\qquad \|\overline\Gamma_{b,\mathbb D,k}^{\,d}\|\le1.
$$
For $k=1$ we write $\overline\Gamma_{b,\mathbb D}^{\,d}$ for this map.
The product maps satisfy the permutation identity
\eqref{eq:transfer-symmetrization} with $X=X_3(\mathbb D)$.
For each of the smooth groups in Theorem~\ref{thm:disc}, these maps preserve
reduced exact bounded cohomology. If $d_1,\ldots,d_k\ge0$ and
$u_j\in\rHb^{d_j}(B_3)$, then
\begin{equation}\label{eq:disc-product-transfer}
 \overline\Gamma_{b,\mathbb D,k}^{\,d_1+\cdots+d_k}(u_1\times\cdots\times u_k)
 =\overline\Gamma_{b,\mathbb D}^{\,d_1}(u_1)\smile\cdots\smile\overline\Gamma_{b,\mathbb D}^{\,d_k}(u_k).
\end{equation}
\end{lemma}

\begin{proof}
We use the construction of \cite[Section~2]{BKdisc}.
Fix a base configuration
$\mathbf z=(z_1,z_2,z_3)\in X_3(\mathbb D)$.
For almost every $\mathbf x=(x_1,x_2,x_3)\in X_3(\mathbb D)$,
the straight-line path
$$
\alpha_{\mathbf z,\mathbf x}(t)=\bigl((1-t)z_1+tx_1,\,(1-t)z_2+tx_2,\,
(1-t)z_3+tx_3\bigr),\qquad 0\le t\le1,
$$
lies in $X_3(\mathbb D)$ and joins $\mathbf z$ to $\mathbf x$.
On the exceptional set of measure zero, choose reference
paths measurably. Let
$\alpha_{\mathbf x,\mathbf z}=\alpha_{\mathbf z,\mathbf x}^{-1}$
denote the reverse path.

For $f\in\mathcal T_{\mathbb D}$, choose an isotopy
$\{f_t\}_{t\in[0,1]}$ from the identity to $f$ and write
$$f_t(\mathbf x)=(f_t(x_1),f_t(x_2),f_t(x_3)).$$
Define
$$
\gamma_{\mathbb D}(f,\mathbf x)=\left[\alpha_{\mathbf z,\mathbf x}
*\{f_t(\mathbf x)\}_{t\in[0,1]}*\alpha_{f(\mathbf x),\mathbf z}\right]^{-1}
\in\pi_1(X_3(\mathbb D),\mathbf z)=P_3\subset B_3.
$$

The homotopy class is independent of the chosen isotopy,
since the group of homeomorphisms of the disc fixing its
boundary is contractible. This defines a measurable
cocycle, with values in $B_3$, satisfying
$$
\gamma_{\mathbb D}(fh,\mathbf x)=\gamma_{\mathbb D}(f,h(\mathbf x))\gamma_{\mathbb D}(h,\mathbf x)
$$
for every $f,h\in\mathcal T_{\mathbb D}$ and almost every
$\mathbf x\in X_3(\mathbb D)$.

We now define the product transfer on cochains. A point of
$X_3(\mathbb D)^k$ is a tuple
$\mathbf X=(\mathbf x^{(1)},\ldots,\mathbf x^{(k)})$, with each
$\mathbf x^{(j)}$ an ordered triple. The action is
$$
f\mathbf X=(f\mathbf x^{(1)},\ldots,f\mathbf x^{(k)}),\qquad
f\mathbf x^{(j)}=(f(x^{(j)}_1),f(x^{(j)}_2),f(x^{(j)}_3)).
$$
This is an action of the group $\mathcal T_{\mathbb D}$. It preserves the probability measure
$$\nu_k=\bigl(\mu^3|_{X_3(\mathbb D)}\bigr)^{\otimes k}.$$
Using these paths in each factor, define the product cocycle by
$$
\begin{aligned}
&\gamma_{\mathbb D,k}:\mathcal T_{\mathbb D}\times X_3(\mathbb D)^k\longrightarrow P_3^k\subset B_3^k,\\
&\gamma_{\mathbb D,k}(f,\mathbf X)=\bigl(\gamma_{\mathbb D}(f,\mathbf x^{(1)}),\ldots,
\gamma_{\mathbb D}(f,\mathbf x^{(k)})\bigr).
\end{aligned}
$$
The preceding cocycle identity in each coordinate gives
$$
\gamma_{\mathbb D,k}(fh,\mathbf X)=\gamma_{\mathbb D,k}(f,h\mathbf X)\gamma_{\mathbb D,k}(h,\mathbf X).
$$

Let $c:(B_3^k)^{d+1}\to\R$ be a bounded cochain invariant under right multiplication. Define
$$
\begin{aligned}
&(\Gamma'_{b,\mathbb D,k}c)(f_0,\ldots,f_d)=\int_{X_3(\mathbb D)^k}
c\bigl(\gamma_{\mathbb D,k}(f_0,\mathbf X),\ldots,\gamma_{\mathbb D,k}(f_d,\mathbf X)\bigr)\,d\nu_k(\mathbf X).
\end{aligned}
$$
In particular, when $k=1$ this is
$$
(\Gamma'_{b,\mathbb D}c)(f_0,\ldots,f_d)=\int_{X_3(\mathbb D)}
 c\bigl(\gamma_{\mathbb D}(f_0,\mathbf x),\ldots,\gamma_{\mathbb D}(f_d,\mathbf x)\bigr)\,d\mu^3(\mathbf x).
$$
The product integral is absolutely convergent. Since
$\nu_k(X_3(\mathbb D)^k)=1$,
$$\|\Gamma'_{b,\mathbb D,k}c\|_\infty\le\|c\|_\infty.$$
We check that it is a right-homogeneous cochain. For
$h\in\mathcal T_{\mathbb D}$ the cocycle identity gives
$$
\begin{aligned}
 &(\Gamma'_{b,\mathbb D,k}c)(f_0h,\ldots,f_dh)=\\
 &\int_{X_3(\mathbb D)^k}
 c\bigl(\gamma_{\mathbb D,k}(f_0,h\mathbf X)\gamma_{\mathbb D,k}(h,\mathbf X),\ldots,
\gamma_{\mathbb D,k}(f_d,h\mathbf X)\gamma_{\mathbb D,k}(h,\mathbf X)\bigr)\,d\nu_k(\mathbf X)=\\
 &\int_{X_3(\mathbb D)^k}c\bigl(\gamma_{\mathbb D,k}(f_0,h\mathbf X),\ldots,
\gamma_{\mathbb D,k}(f_d,h\mathbf X)\bigr)\,d\nu_k(\mathbf X)=
 (\Gamma'_{b,\mathbb D,k}c)(f_0,\ldots,f_d).
 \end{aligned}
$$
The second equality uses right invariance of $c$, with the same multiplier
$\gamma_{\mathbb D,k}(h,\mathbf X)\in B_3^k$ in all $d+1$ arguments. The last
uses the measure-preserving change of variable $\mathbf Y=h\mathbf X$.

We also check the differential. The homogeneous differential is the
alternating sum obtained by deleting one argument, so
$$
\begin{aligned}
&(\delta\Gamma'_{b,\mathbb D,k}c)(f_0,\ldots,f_{d+1})=\\
&\sum_{\ell=0}^{d+1}(-1)^\ell\int_{X_3(\mathbb D)^k}c\bigl(\gamma_{\mathbb D,k}(f_0,\mathbf X),\ldots,
\widehat{\gamma_{\mathbb D,k}(f_\ell,\mathbf X)},\ldots,
\gamma_{\mathbb D,k}(f_{d+1},\mathbf X)\bigr)\,d\nu_k(\mathbf X)=\\
&\int_{X_3(\mathbb D)^k}(\delta c)\bigl(\gamma_{\mathbb D,k}(f_0,\mathbf X),\ldots,
\gamma_{\mathbb D,k}(f_{d+1},\mathbf X)\bigr)\,d\nu_k(\mathbf X)\\
&=(\Gamma'_{b,\mathbb D,k}\delta c)(f_0,\ldots,f_{d+1}).
\end{aligned}
$$
The hat means that the indicated argument is omitted. Thus
$\delta\Gamma'_{b,\mathbb D,k}=\Gamma'_{b,\mathbb D,k}\delta$. We obtain a
map on bounded cohomology, and its norm bound sends zero-seminorm classes
to zero-seminorm classes. Its induced map on reduced bounded cohomology is
$\overline\Gamma_{b,\mathbb D,k}^{\,d}$. The construction 
applies to area-preserving homeomorphisms as well as smooth transformations.

For the smooth transformation groups, the finite-image property
follows from \cite[Lemma~2.1]{BMentropy}: with the reference
paths chosen there, for every fixed $f$ there is a finite set
$S_f\subset P_3$ such that
$$
\gamma_{\mathbb D}(f,\mathbf x)\in S_f\qquad\text{for almost every }\mathbf x\in X_3(\mathbb D).
$$
We may use these reference paths, since the induced bounded
transfer is independent of their choice. Consequently,
$\gamma_{\mathbb D,k}(f,\mathbf X)\in S_f^k$ for almost every
$\mathbf X\in X_3(\mathbb D)^k$.

For each fixed tuple $(f_0,\ldots,f_d)$, the integrand in the
transfer formula therefore has finite image almost everywhere,
even when the cochain is unbounded. Since $\nu_k$ is finite,
the same integral defines a map $\Gamma'_{\mathbb D,k}$ on
ordinary right-homogeneous cochains. The preceding cocycle
calculation shows that this is a cochain map, and it agrees
with $\Gamma'_{b,\mathbb D,k}$ on bounded cochains.

In particular, if a bounded cocycle $c$ has an ordinary
primitive $A$, then
$$
\Gamma'_{b,\mathbb D,k}c=\Gamma'_{\mathbb D,k}(\delta A)=\delta\bigl(\Gamma'_{\mathbb D,k}A\bigr).
$$
Thus the transfer preserves exact bounded cohomology and,
by its norm bound, reduced exact bounded cohomology.

We finish by proving \eqref{eq:disc-product-transfer} on cochains, first for
two factors. In the right-homogeneous convention, the cup product of a
$p$-cochain $c$ and a $q$-cochain $e$ on the same group is
$$
(c\smile e)(g_0,\ldots,g_{p+q})=c(g_0,\ldots,g_p)e(g_p,\ldots,g_{p+q}).
$$
This is \eqref{eq:cupcochain} under the homogeneous identification in
Lemma~\ref{lem:transfer}. If $c$ and $e$ are cochains on $B_3$, their
external product is the cochain on $B_3\times B_3$ given by
$$
(c\times e)\bigl((g_0,h_0),\ldots,(g_{p+q},h_{p+q})\bigr)=c(g_0,\ldots,g_p)e(h_p,\ldots,h_{p+q}).
$$
Let $c,e$ be bounded; they need not be cocycles for the following identity.
For $(\mathbf x,\mathbf y)\in X_3(\mathbb D)^2$ we have
$$
\gamma_{\mathbb D,2}(f,(\mathbf x,\mathbf y))=\bigl(\gamma_{\mathbb D}(f,\mathbf x),\gamma_{\mathbb D}(f,\mathbf y)\bigr).
$$
For this computation, fix $\mathbf f=(f_0,\ldots,f_{p+q})$ and write the
two integrands as
$$
\begin{aligned}
&c_{\mathbf f}(\mathbf x)=c\bigl(\gamma_{\mathbb D}(f_0,\mathbf x),\ldots,
\gamma_{\mathbb D}(f_p,\mathbf x)\bigr),\\
&e_{\mathbf f}(\mathbf y)=e\bigl(\gamma_{\mathbb D}(f_p,\mathbf y),\ldots,
\gamma_{\mathbb D}(f_{p+q},\mathbf y)\bigr).
\end{aligned}
$$
Substitution into the product transfer gives
$$
\begin{aligned}
&(\Gamma'_{b,\mathbb D,2}(c\times e))(f_0,\ldots,f_{p+q})=\int_{X_3(\mathbb D)^2}
c_{\mathbf f}(\mathbf x)e_{\mathbf f}(\mathbf y)\,d\nu_2(\mathbf x,\mathbf y)\\
&=\left[\int_{X_3(\mathbb D)}c_{\mathbf f}(\mathbf x)\,d\mu^3(\mathbf x)\right]
\left[\int_{X_3(\mathbb D)}e_{\mathbf f}(\mathbf y)\,d\mu^3(\mathbf y)\right]\\
&=(\Gamma'_{b,\mathbb D}c)(f_0,\ldots,f_p)(\Gamma'_{b,\mathbb D}e)(f_p,\ldots,f_{p+q})\\
&=\bigl((\Gamma'_{b,\mathbb D}c)\smile(\Gamma'_{b,\mathbb D}e)\bigr)(f_0,\ldots,f_{p+q}).
\end{aligned}
$$
Thus we have the actual cochain identity
$$
\Gamma'_{b,\mathbb D,2}(c\times e)=(\Gamma'_{b,\mathbb D}c)\smile(\Gamma'_{b,\mathbb D}e).
$$

For completeness, let $c_j$ be a bounded right-homogeneous $d_j$-cochain on
$B_3$, put $D_0=0$ and $D_j=d_1+\cdots+d_j$, and write $d=D_k$.
For $\mathbf h_\ell=(h_\ell^{(1)},\ldots,h_\ell^{(k)})\in B_3^k$,
our external-product convention gives
$$
(c_1\times\cdots\times c_k)(\mathbf h_0,\ldots,\mathbf h_d)=\prod_{j=1}^k
c_j\bigl(h_{D_{j-1}}^{(j)},\ldots,h_{D_j}^{(j)}\bigr).
$$
For the fixed tuple $\mathbf f=(f_0,\ldots,f_d)$, put
$$
c_{j,\mathbf f}(\mathbf x)=
c_j\bigl(\gamma_{\mathbb D}(f_{D_{j-1}},\mathbf x),\ldots,\gamma_{\mathbb D}(f_{D_j},\mathbf x)\bigr).
$$
The same substitution and Fubini's theorem give
$$
\begin{aligned}
&(\Gamma'_{b,\mathbb D,k}(c_1\times\cdots\times c_k))(f_0,\ldots,f_d)
=\int_{X_3(\mathbb D)^k}\prod_{j=1}^k c_{j,\mathbf f}(\mathbf x^{(j)})\,d\nu_k(\mathbf X)\\
&=\prod_{j=1}^k\left[\int_{X_3(\mathbb D)}c_{j,\mathbf f}(\mathbf x)\,d\mu^3(\mathbf x)\right]
=\prod_{j=1}^k(\Gamma'_{b,\mathbb D}c_j)(f_{D_{j-1}},\ldots,f_{D_j})\\
&=\bigl((\Gamma'_{b,\mathbb D}c_1)\smile\cdots\smile(\Gamma'_{b,\mathbb D}c_k)\bigr)(f_0,\ldots,f_d).
\end{aligned}
$$
All integrands are bounded in absolute value by
$\prod_j\|c_j\|_\infty$. Taking cocycles representing $u_j$ 
proves the product identity on bounded
cohomology. The transfer and product norm bounds, including
\eqref{eq:productnorm}, show that zero-seminorm classes give zero-seminorm
terms on both sides, so the identity passes to reduced bounded cohomology
and gives \eqref{eq:disc-product-transfer}. 
\end{proof}

We next construct the simultaneous rotations used to compute the pullback of this transfer.
The number of copies $k$ is fixed in each construction.
Choose $k$ pairwise disjoint open discs $E_j\Subset\Int\mathbb D$.
Inside each $E_j$ choose two overlapping round discs $N_{j,1},N_{j,2}$ in
area coordinates. The pairs are congruent, and their closures lie in $E_j$.
Each of the three sets
$$
N_{j,1}\setminus\overline{N_{j,2}},\qquad N_{j,1}\cap N_{j,2},\qquad N_{j,2}\setminus\overline{N_{j,1}}
$$
has positive area. For any fixed $k$, this is possible in one area-coordinate
chart by taking the radii sufficiently small. Figure~\ref{fig:disc-pairs}
illustrates the case $k=2$.

\begin{figure}[htbp]
\centering
\begin{adjustbox}{max width=\linewidth}
\includegraphics{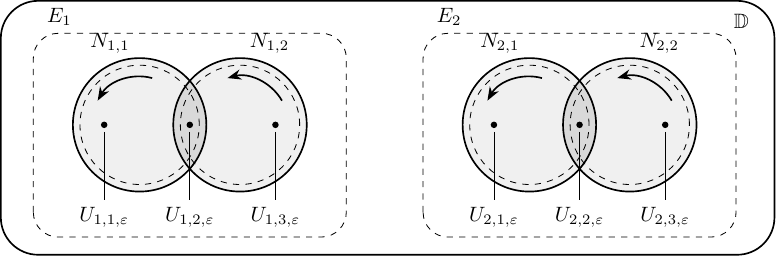}
\end{adjustbox}
\caption{Two disjoint copies of the disc rotation construction. The dashed
inner circles bound the regions rotated through one complete turn; the thin
annuli between the circles are transition regions. A triple with one point
in each of the three indicated subsets in a single $E_j$ gives the free braid
homomorphism. A triple cannot have two points in each of two different $E_j$.}
\label{fig:disc-pairs}
\end{figure}

For $s=1,2$, let $N_{j,s}^{\eps}$ be the concentric disc of radius $R-\eps$
inside $N_{j,s}$, for $\eps>0$ small. Choose a smooth radial function
$\chi_\eps$, equal to one on $[0,R-\eps]$, equal to zero near $R$, and with
$0\le\chi_\eps\le1$. The isotopy
$$(r,\varphi)\longmapsto(r,\varphi+2\pi t\chi_\eps(r)),\qquad 0\le t\le1,
$$
extended by the identity outside $N_{j,s}$, is Hamiltonian. 
Let $f_{j,s,\eps}$ be its time-one map. It fixes
$N_{j,s}^{\eps}$ pointwise, but its isotopy makes each point in that disc
travel once around the center.

Define $\rho_{j,\eps}:\F\to\Ham(\mathbb D,\omega)$ by
$\rho_{j,\eps}(a)=f_{j,1,\eps},\qquad\rho_{j,\eps}(b)=f_{j,2,\eps}$.
Images belonging to different $j$ have disjoint supports and commute. Hence
\begin{equation}\label{eq:disc-rho}
\rho_\eps:\F^k\longrightarrow\Ham(\mathbb D,\omega),\qquad\rho_\eps(w_1,\ldots,w_k)=\prod_{j=1}^k\rho_{j,\eps}(w_j)
\end{equation}
is a homomorphism, also viewed as taking values in any $\mathcal T_{\mathbb D}$.
Put
$$
B_\eps=\bigcup_{j,s}(N_{j,s}\setminus N_{j,s}^{\eps}),\qquad b_\eps=\mu(B_\eps)\longrightarrow0.
$$
Each time-one map fixes $\mathbb D\setminus B_\eps$ pointwise, so $B_\eps$ is
invariant under the generated group. Define
\begin{align*}
 &U_{j,1,\eps}=(N_{j,1}\setminus\overline{N_{j,2}})\setminus B_\eps,\\
 &U_{j,2,\eps}=(N_{j,1}\cap N_{j,2})\setminus B_\eps,\\
 &U_{j,3,\eps}=(N_{j,2}\setminus\overline{N_{j,1}})\setminus B_\eps.
 \end{align*}
Their areas are independent of $j$. Set
\begin{equation}\label{eq:disc-lambda}
 \lambda_\eps=6\prod_{s=1}^3\mu(U_{j,s,\eps}),
 \qquad
 \lambda=\lim_{\eps\to0}\lambda_\eps>0.
\end{equation}
Let $p_j:\F^k\to\F$ denote projection onto the $j$th factor.

\begin{lemma}\label{lem:disc-approx}
For every $d\ge2$ and $u\in\rHb^d(B_3)$,
\begin{equation}\label{eq:disc-approx}
 \left\|\rho_\eps^*\overline\Gamma_{b,\mathbb D}^{\,d}(u)
       -\lambda\sum_{j=1}^k p_j^*\jmath^*u\right\|
 \le\bigl(3b_\eps+k|\lambda_\eps-\lambda|\bigr)\|u\|.
\end{equation}
In particular, the expression on the left tends to zero.
\end{lemma}
\begin{proof}
We give the complete decomposition of a single three-point configuration;
the same decomposition applies simultaneously to all $k$ copies. This is the
local computation of \cite[Lemma~4.2 and Section~4.3]{Kimura}, in the disjoint
setting of \cite[proof of Theorem~4.3]{Nitsche}.

Let $\mathcal E_\eps\subset X_3(\mathbb D)$ consist of triples with at least
one coordinate in $B_\eps$. It is invariant under the diagonal action of the time-one maps, and
$$\mu^3(\mathcal E_\eps)=1-(1-b_\eps)^3\le3b_\eps.$$
For a bounded cocycle $c$ representing $u$, integration over
$\mathcal E_\eps$ is a bounded cochain map: invariance gives the
measure-preserving change of variables needed for right homogeneity, and the
differential commutes with integration. Its value has norm at most
$3b_\eps\|c\|_\infty$.

Every point of a triple outside $\mathcal E_\eps$ is fixed by each time-one map.
Consequently $w\mapsto\gamma_{\mathbb D}(\rho_\eps(w),\mathbf x)$ is a group
homomorphism. A rotation supported in an $E_j$ containing at most one point
of the triple gives the trivial braid: that point traces a loop in the
simply connected disc $E_j$, which contains neither of the other points.
Since the $E_j$ are disjoint and a triple has only three points, at most one
$E_j$ contains two or more points. Thus the associated homomorphism is either
trivial or factors through one projection $p_j$. This also deals with
configurations spread among different copies; they cannot produce two
independent nonabelian braid contributions.

Fix the possibly nontrivial copy $j$. The braid made by one complete
rotation of $N_{j,s}$ is trivial if it contains zero or one marked point, a
full twist of the enclosed pair if it contains two, and the central full
twist $\Delta^2$ if it contains all three. If one of the two discs contains
zero or one point, one generator is trivial. If it contains all three, its
generator is central. In both cases the two generators have abelian image.
If both discs contain exactly two points and those pairs are equal, the
corresponding twists are equal up to the same conjugacy; both occur inside
$E_j$ away from the third point. This again gives cyclic image.
The only remaining case is two different pairs of marked points. They share
one point, and the triple then has exactly one point in each of
$U_{j,1,\eps},U_{j,2,\eps},U_{j,3,\eps}$.

Order those three regions from the first exclusive part through the
intersection to the second exclusive part. In that ordering the two
rotations make the braids $\sigma_1^2,\sigma_2^2$. For an arbitrary labeling
and arbitrary closing paths, the simultaneous associated homomorphism is
$$
w\longmapsto\beta_{\mathbf x}\,\jmath(w_j)\,\beta_{\mathbf x}^{-1},\qquad \beta_{\mathbf x}\in B_3.
$$
The same conjugator works on both generators, and hence on every word: an
orientation-preserving identification of the two overlapping marked discs
with the model identifies both twists simultaneously. This is precisely the
common-conjugator assertion of \cite[Lemma~4.2]{Kimura}. The conjugator does
not depend on $w$.

The six possible assignments of labels give six product subsets, each of
measure $\prod_s\mu(U_{j,s,\eps})$. Because $u$ is a class on $B_3$, inner
conjugation by $\beta_{\mathbf x}$ acts trivially on it. Thisis not
true for an arbitrary class on $P_3$. The contribution of these six subsets
is therefore $\lambda_\eps p_j^*\jmath^*u$.
All other subsets have amenable image and their positive-degree
bounded-cohomology contribution is zero. Partitioning by the countably many
associated homomorphisms gives the same conclusion for their union: finite
partial sums are zero in reduced bounded cohomology, and the remaining
integral is bounded by its measure times $\|c\|_\infty$.

Adding the $k$ contributions and the exceptional contribution proves
$$
\left\|\rho_\eps^*\overline\Gamma_{b,\mathbb D}^{\,d}(u)-\lambda_\eps\sum_{j=1}^kp_j^*\jmath^*u\right\|\le3b_\eps\|u\|.
$$
Here we take the infimum over bounded representatives of $u$.
Since each $p_j^*$ and $\jmath^*$ has norm at most one, replacing
$\lambda_\eps$ by $\lambda$ gives \eqref{eq:disc-approx}.
\end{proof}

\begin{remark*}[The fixed Brooks class for the disc]\label{rem:disc-brooks-choice}
For the disc proof, we now choose the fixed class $\alpha$ so that
$$\jmath^*\alpha=[\delta\phi_{ab}]+[\delta\phi_{ba}],$$
where $\phi_w$ is the Brooks quasimorphism counting occurrences of $w$
minus occurrences of $w^{-1}$ in reduced words. For this choice,
\cite[Theorem~A\textup{(b)}]{AB} and bilinearity imply
$$
\jmath^*\alpha\smile\jmath^*\beta=0\qquad\bigl(\beta\in\rHb^q(B_3),\ q\ge1\bigr).
$$
This will make every repeated-coordinate term vanish in the disc
calculation, exactly as for the fixed class in Theorem~\ref{thm:main}.
We verify both nonvanishing and extension to $B_3$ before making this
choice in the proof.

We verify that this choice extends to $B_3$. Put
$e=[\delta\phi_{ab}]$ and $\eta=e+[\delta\phi_{ba}]$.
In $P_3/\langle\Delta^2\rangle=\F$, conjugation by
$\sigma_1\sigma_2\sigma_1$ and $\sigma_1\sigma_2$ induces, respectively,
$$
S(a)=b,\quad S(b)=a,\qquad T(a)=b,\quad T(b)=b^{-1}a^{-1}.
$$
These formulas follow from the braid relation and the central quotient;
see the presentation preceding Lemma~\ref{lem:braid-cohomology}.
The two conjugating braids generate $B_3$. Clearly
$S^*e=[\delta\phi_{ba}]$ and $S^*\eta=\eta$.
Let $\varepsilon_b:\F\to\R$ be the exponent-sum homomorphism with
$\varepsilon_b(a)=0$ and $\varepsilon_b(b)=1$. Direct counting gives
$$
\phi_{ab}(T(g))=\phi_{ab}(g)-\varepsilon_b(g)\qquad(g\in\F).
$$
Indeed, the substitution is
$a\mapsto b$, $a^{-1}\mapsto b^{-1}$,
$b\mapsto b^{-1}a^{-1}$, $b^{-1}\mapsto ab$.
The cancellations at the adjacent pairs $ab$ and $b^{-1}a^{-1}$ remove,
respectively, one negative and one positive occurrence counted by
$\phi_{ab}$; no other cancellation occurs for the 
reduced word to which the substitution is applied.
Before these cancellations the signed count is
$-\varepsilon_b(g)$, which proves the displayed identity.
Taking coboundaries gives $T^*e=e$. Checking $a$ and $b$ also gives
$$S\circ T\circ S=\operatorname{Inn}(a^{-1})\circ T^{-1}.$$
Since inner automorphisms act trivially on bounded cohomology, this
identity implies $T^*S^*e=S^*e$. Hence $T^*\eta=\eta$, so $\eta$ is
invariant under the full braid-group action on the central quotient.
Let $Q$ be the homogenization of $\phi_{ab}+\phi_{ba}$. Since
$Q(a)=Q(b)=0$ and $Q(ab)=2$, the class $\eta=[\delta Q]$ is nonzero in bounded cohomology, 
and hence in reduced bounded cohomology.

We first replace $Q$ by an invariant representative of the
same bounded-cohomology class. Let
$\varepsilon:\F\to\R$ be the homomorphism satisfying
$\varepsilon(a)=\varepsilon(b)=1$, and put $Q_0=Q-\frac23\varepsilon$.
Since $T^*[\delta Q]=[\delta Q]$ and $Q$ is homogeneous,
$Q\circ T-Q$ is a homomorphism. Evaluating on $a,b$ gives
$$
Q\circ T-Q=-2\varepsilon_b,\qquad\varepsilon\circ T-\varepsilon=-3\varepsilon_b.
$$
Hence $Q_0\circ T=Q_0$. Both $Q$ and $\varepsilon$ are
invariant under $S$, so $Q_0\circ S=Q_0$ as well.
Thus $Q_0$ is invariant under the induced $B_3$-action,
and $\delta Q_0=\delta Q$.

Let $p:P_3\to\F$ be the projection.
The homogeneous quasimorphism $Q_0\circ p$ is invariant
under conjugation by $B_3$. Since $P_3$ is normal of
index 6, it extends to a homogeneous quasimorphism
$\overline Q:B_3\to\R$. Explicitly,
$$
\overline Q(\beta)=\frac16 Q_0\bigl(p(\beta^6)\bigr),\qquad \beta\in B_3.
$$
The formula is defined because $\beta^6\in P_3$.
For $h\in P_3$, homogeneity gives
$\overline Q(h)=Q_0(p(h))$.
Define $\alpha=[\delta\overline Q]\in\rHb^2(B_3)$.
Since $p\circ\jmath=\id$, we obtain
$$
\jmath^*\alpha=[\delta Q_0]=[\delta Q]=\eta\ne0.
$$
Choose $z$ as in Lemma~\ref{lem:braid-cohomology}, so that
$\langle\jmath^*\alpha,z\rangle=1$.
\end{remark*}

\begin{proof}[Proof of Theorem~\ref{thm:disc} for the disc]
Choose $\alpha,z$ by the preceding Brooks-class construction, and put
$\theta=\jmath^*\alpha$. Thus $\langle\theta,z\rangle=1$ and
$\theta\smile u=0$ for every positive-degree bounded class $u$ on
$\F$. In this proof, $T_m^q$ denotes the map of
Proposition~\ref{prop:symmetric-products} for this fixed $\theta,z$.
Let $m\ge0$, $q\in\{2,3\}$, $k=m+1$, and $d=2m+q$.
For $\beta\in\rHb^q(B_3)$ define
$$
\xi_\beta=\Sym_k(\alpha^{\times m}\times\beta)\in\rHb^d(B_3^k).
$$
This class is symmetric. Naturality gives
$(\jmath^k)^*\xi_\beta=T_m^q(\jmath^*\beta)$,
so $\beta\mapsto\xi_\beta$ is injective by Proposition~\ref{prop:symmetric-products} and
Lemma~\ref{lem:braid-cohomology}. Define
\begin{equation}\label{eq:disc-symmetric-map}
\begin{aligned}
&\Psi_{m,q}:\rHb^q(B_3)\longrightarrow\rHb^d(\mathcal T_{\mathbb D}),\\
&\Psi_{m,q}(\beta) =\frac1{k!}\overline\Gamma_{b,\mathbb D,k}^{\,d}(\xi_\beta).
\end{aligned}
\end{equation}
By \eqref{eq:transfer-symmetrization} and
\eqref{eq:disc-product-transfer}, this is equivalently
$$
\Psi_{m,q}(\beta)=\bigl(\overline\Gamma_{b,\mathbb D}^{\,2}(\alpha)\bigr)^{\smile m}
\smile\overline\Gamma_{b,\mathbb D}^{\,q}(\beta).
$$
Thus the map is the same cup product of transferred classes, now written
as a transfer of a symmetric external product. 

\emph{Two factors: $m=1$, $k=2$.}
Fix $q=2$ or $3$ and write $b=\jmath^*\beta$.
Use the two disjoint rotation constructions preceding
Lemma~\ref{lem:disc-approx}. They give one homomorphism
$\rho_\eps:\F^2\to\Ham(\mathbb D,\omega)$, independent of $\beta$.
Apply that lemma first to $\alpha$ in degree two and then to $\beta$
in degree $q$. The constant $\lambda>0$ is defined in
\eqref{eq:disc-lambda}. With the same $\rho_\eps$ and $\lambda$ in
both applications, we obtain
$$
\rho_\eps^*\overline\Gamma_{b,\mathbb D}^{\,2}(\alpha)
\longrightarrow\lambda(p_1^*\theta+p_2^*\theta),\qquad
\rho_\eps^*\overline\Gamma_{b,\mathbb D}^{\,q}(\beta)
\longrightarrow\lambda(p_1^*b+p_2^*b).
$$
These are norm limits. The product estimate
\eqref{eq:productnorm} allows us to multiply them. More explicitly,
put $u_\eps=\rho_\eps^*\overline\Gamma_{b,\mathbb D}^{\,2}(\alpha)$,
$w_\eps=\rho_\eps^*\overline\Gamma_{b,\mathbb D}^{\,q}(\beta)$,
and denote their displayed limits by $u_0,w_0$. Then
$$
\|u_\eps\smile w_\eps-u_0\smile w_0\|\le\|u_\eps-u_0\|\,\|w_\eps\|
+\|u_0\|\,\|w_\eps-w_0\|\longrightarrow0.
$$
The norms $\|w_\eps\|$ are bounded by $\|\beta\|$.
Expanding the limiting product gives
\begin{equation}\label{eq:disc-two}
\rho_\eps^*\Psi_{1,q}(\beta)\longrightarrow\lambda^2\bigl[p_1^*(\theta\smile b)+\theta\times b
+b\times\theta+p_2^*(\theta\smile b)\bigr].
\end{equation}
The first and fourth terms vanish by the choice of $\theta$.
The two mixed terms have the same positive sign, since
$(-1)^{2q}=1$. Hence \eqref{eq:disc-two} reduces to
\begin{equation}\label{eq:disc-two-symmetric}
 \rho_\eps^*\Psi_{1,q}(\beta)
   \longrightarrow\lambda^2(\theta\times b+b\times\theta)
   =\lambda^2T_1^q(b).
\end{equation}
This is the same symmetric-product limit as in Theorem~\ref{thm:main}.
If $\Psi_{1,q}(\beta)=0$, then \eqref{eq:disc-two-symmetric} gives
$T_1^q(b)=0$. Proposition~\ref{prop:symmetric-products} implies $b=0$,
and injectivity of $\jmath^*$ gives $\beta=0$. 

\emph{An arbitrary number of factors.}
For $m\ge2$, use $k=m+1$ disjoint copies and apply
Lemma~\ref{lem:disc-approx} to $\alpha$ and $\beta$, with the same
homomorphism. Replacing one factor at a time in the cup product, using
the preceding estimate, gives
$$
\rho_\eps^*\Psi_{m,q}(\beta)\longrightarrow\lambda^k\left(\sum_{j=1}^kp_j^*\theta\right)^{\smile m}
\smile\left(\sum_{j=1}^kp_j^*b\right).
$$
There are $k$ labelled factors of positive degree. In a nonbijective
assignment to the $k$ coordinates, a repeated coordinate contains a
fixed $\theta$ and another positive-degree class. Their internal cup
product is zero. Only bijective assignments remain, and their sum is
$\Sym_k(\theta^{\times m}\times b)=T_m^q(b)$. Therefore
\begin{equation}\label{eq:disc-symmetric-limit}
\rho_\eps^*\Psi_{m,q}(\beta)\longrightarrow\lambda^k T_m^q(\jmath^*\beta).
\end{equation}
The same formula holds for $m=1$ by \eqref{eq:disc-two-symmetric} and
for $m=0$ by Lemma~\ref{lem:disc-approx}, with $T_0^q=\id$.
Since $\lambda>0$, \eqref{eq:disc-symmetric-limit},
Proposition~\ref{prop:symmetric-products}, and injectivity of
$\jmath^*$ prove that every $\Psi_{m,q}$ is injective. In particular,
a nonzero class gives an image of positive reduced norm. For example,
if $Q_{m,q}$ is the bounded left inverse from
\eqref{eq:explicitQ}, then
$$
\|\Psi_{m,q}(\beta)\|\ge\lambda^k\|T_m^q(\jmath^*\beta)\|\ge\frac{\lambda^k}{\|Q_{m,q}\|}\,\|\beta\|.
$$
We used that $\jmath^*$ is isometric; the constants may depend on $k$.

Both $\rHb^2(B_3)$ and $\rHb^3(B_3)$ are infinite-dimensional by
Lemma~\ref{lem:braid-cohomology}, which proves the reduced bounded-cohomology
assertion for the disc. For smooth transformations, $\alpha$ and
$\beta$ are reduced exact by that lemma, and
Lemma~\ref{lem:disctransfer} preserves exactness. Each summand of
$\xi_\beta$ is exact, so its normalized transfer in
\eqref{eq:disc-symmetric-map} is reduced exact. For homeomorphisms
we only use bounded cochains on the finite-measure configuration space
and do not assert an ordinary transfer. All testing homomorphisms take
values in $\Ham(\mathbb D,\omega)$,
so the same proof applies to each of the four transformation groups.
\end{proof}

\subsection{Sphere case}
\label{sec:4-2}

We now prove the sphere case of Theorem~\ref{thm:disc}. Write
$\mathcal T_{S^2}$ for any of the four area-preserving transformation groups
in that theorem, and normalize the area form so that $\mu(S^2)=1$.
The domain of the transfer is the reduced bounded cohomology of the full
mapping class group of the sphere with four marked points, not that of its
pure subgroup.
We first explain this choice and its relation to the three-strand disc braid group.

Fix a disc in $S^2$ and a fourth point outside it. Adding the constant
fourth strand to a three-strand pure disc braid defines a homomorphism
$$\kappa:P_3(\mathbb D)\longrightarrow P_4(S^2).$$
For $r=3,4$, write $\Delta_r^2$ for the full twist on $r$ strands, in the
disc or sphere as appropriate. Thus
$\Delta_3^2=(\sigma_1\sigma_2)^3$, and $\Delta_4^2$ has order two in
$P_4(S^2)$. Denote by $\delta_1,\delta_2$ the elementary spherical braids
supported in the chosen three-punctured disc. In particular,
$\kappa(\sigma_s^2)=\delta_s^2$ for $s=1,2$.
We recall the following standard braid and mapping class group
identifications; see \cite{FVB} and
\cite[Theorem~2.4 and Section~5.2]{Kimura}.
The homomorphism $\kappa$ is onto and
\begin{equation}\label{eq:sphere-braid-quotient}
\begin{aligned}
&\ker\kappa=\langle(\Delta_3^2)^2\rangle,
\qquad \kappa(\Delta_3^2)=\Delta_4^2,\\
&P_4(S^2)=\langle\delta_1^2,\delta_2^2\rangle
\times\langle\Delta_4^2\rangle\cong\F\times\Z/2\Z.
\end{aligned}
\end{equation}
Let $\operatorname{Mod}(S^2,4)$ denote the orientation-preserving
mapping class group of the four-punctured sphere. We identify
$$
\operatorname{Mod}(S^2,4)=B_4(S^2)/\langle\Delta_4^2\rangle.
$$
The images of $\delta_1^2,\delta_2^2$ define an embedding
$\jmath_{S^2}:\F\hookrightarrow\operatorname{Mod}(S^2,4)$
whose image is the pure mapping class group and has index $24$.

\begin{lemma}\label{lem:sphere-cohomology}
For $q=2,3$, restriction is an isometric embedding
\begin{equation}\label{eq:sphere-restriction}
 \jmath_{S^2}^*:\rHb^q\bigl(\operatorname{Mod}(S^2,4)\bigr)
 \longrightarrow\rHb^q(\F).
\end{equation}
Its domain is infinite-dimensional and equals
$\rEH^q\bigl(\operatorname{Mod}(S^2,4)\bigr)$.
There are a fixed nonzero class
$\alpha\in\rHb^2\bigl(\operatorname{Mod}(S^2,4)\bigr)$
and a cycle $z\in Z_2^{\ell^1}(\F)$ such that
\begin{equation}\label{eq:sphere-normalization}
\langle\jmath_{S^2}^*\alpha,z\rangle=1.
\end{equation}
\end{lemma}

\begin{proof}
Since the pure subgroup has finite index, restriction is
isometric in bounded cohomology; see
\cite[Theorem~2.1]{Kimura}. It therefore induces the isometric
embedding \eqref{eq:sphere-restriction} on reduced bounded
cohomology.

The group $\operatorname{Mod}(S^2,4)$ is non-elementary
hyperbolic, since it contains $\F$ with finite index.
Its reduced bounded cohomology is infinite-dimensional
in degrees two and three by \cite{EpsteinFujiwara} and
\cite[Corollary~3]{FPS}, respectively; see also
\cite[Section~5.4]{Kimura} for the degree-three conclusion.

Ordinary restriction to a finite-index subgroup is injective.
Since $H^q(\F)=0$ for $q\ge2$, it follows that
$$H^q\bigl(\operatorname{Mod}(S^2,4))=0\qquad(q\ge2).$$
Thus the comparison map is zero in these degrees,
and reduced bounded cohomology equals reduced exact bounded
cohomology.

Choose $0\ne\alpha\in\rHb^2\bigl(\operatorname{Mod}(S^2,4)\bigr)$.
Its restriction is nonzero by isometry. The degree-two
duality recalled in Lemma~\ref{lem:brooks} gives an
$\ell^1$-cycle with nonzero pairing with
$\jmath_{S^2}^*\alpha$; rescaling that cycle gives
\eqref{eq:sphere-normalization}.
\end{proof}

In particular, using $\operatorname{Mod}(S^2,4)$ instead of $B_4(S^2)$
does not discard bounded-cohomology classes: pullback by the quotient with
finite central kernel is isometric in bounded and reduced bounded
cohomology. It is also an isomorphism on ordinary real cohomology.
It does not identify the domain of the transfer with all of
$\rHb^q(\F)$; only the image in \eqref{eq:sphere-restriction} is used.

Let $X_4(S^2)$ denote the ordered configuration space
$$
X_4(S^2)=\{(x_1,x_2,x_3,x_4)\in(S^2)^4:x_i\ne x_j\text{ for }i\ne j\}
$$
with the restriction of the product measure. This measure has total mass one, and $\mathcal T_{S^2}$ acts diagonally.
We use the standard configuration-space cocycle with values in
$\operatorname{PMod}(S^2,4)\subset\operatorname{Mod}(S^2,4)$.

\begin{lemma}\label{lem:spheretransfer}
Integration over $X_4(S^2)$ defines a bounded linear map
$$
\overline\Gamma_{b,S^2}^{\,d}:\rHb^d(\operatorname{Mod}(S^2,4))\longrightarrow
\rHb^d(\mathcal T_{S^2}),\qquad\|\overline\Gamma_{b,S^2}^{\,d}\|\le1.
$$
For smooth transformations this map preserves reduced exact bounded
cohomology. Integration over $X_4(S^2)^k$ defines the corresponding product
transfer, denoted by $\overline\Gamma_{b,S^2,k}^{\,d}$, and the
transfer of an external product equals the cup product of the individual
transferred classes. The product maps satisfy
\eqref{eq:transfer-symmetrization} with $X=X_4(S^2)$.
\end{lemma}

\begin{proof}
We explain the two points that differ from the disc: isotopy independence
and the ordinary transfer. Choose a base configuration and measurable
closing paths. An isotopy from the identity to $f$ gives a braid in $P_4(S^2)$.
If two isotopies have the same endpoint, the difference is obtained by
evaluating a loop of surface homeomorphisms on the configuration.
Such evaluation loops are central in the configuration fundamental group:
for a loop $\ell$ of configurations, the square
$(s,t)\mapsto f_t(\ell(s))$ gives a homotopy between the two orders of
concatenation. By \eqref{eq:sphere-braid-quotient}, this central ambiguity is
contained in $\langle\Delta_4^2\rangle$.
Projection to $\operatorname{PMod}(S^2,4)$ removes it.
The resulting measurable cocycle, denoted by $\gamma_{S^2}$, satisfies
$$
\gamma_{S^2}(fg,\mathbf x)=\gamma_{S^2}(f,g\mathbf x)\gamma_{S^2}(g,\mathbf x)
$$
in our right-homogeneous convention. This is the sphere construction in
\cite[Section~5, before Section~5.1]{Kimura}; the centrality argument also
applies to homeomorphism isotopies.

For a bounded homogeneous cochain $c$ on $\operatorname{Mod}(S^2,4)$, set
\begin{equation}\label{eq:sphere-transfer}
 (\Gamma'_{b,S^2}c)(f_0,\ldots,f_d)
 =\int_{X_4(S^2)}c\bigl(\gamma_{S^2}(f_0,\mathbf x),\ldots,\gamma_{S^2}(f_d,\mathbf x)\bigr)\,d\mu^4(\mathbf x).
\end{equation}
The same cocycle identity and change of variables as in
Lemma~\ref{lem:disctransfer} show that this is a cochain map with norm at
most one. It induces the asserted reduced bounded-cohomology map for all
four transformation groups. 

For the smooth groups, the reference paths can be chosen so
that, for each fixed $f$, the cocycle $\gamma_{S^2}(f,\cdot)$
has finite image on a full-measure subset; see
\cite[Lemma~2.1]{BMentropy} and \cite[Section~5]{Kimura}.
The induced bounded transfer is unchanged by this choice.
Thus, for each tuple of transformations, the transfer
integrand has finite image almost everywhere even for an
unbounded ordinary cochain. Hence, the same integral 
defines an ordinary cochain map, which agrees with the
bounded transfer on bounded cochains and commutes with
comparison. In particular, an ordinary primitive of a bounded
cocycle transfers to an ordinary primitive of its image.
Hence the transfer preserves exact bounded cohomology and,
by its norm bound, reduced exact bounded cohomology.
No ordinary transfer is asserted here for area-preserving homeomorphisms.

Finally, the diagonal action on $X_4(S^2)^k$ uses $k$ separate four-point
configurations. The external-product cochain is a product of functions, each depending on
one configuration. Fubini's theorem therefore gives
\begin{equation}\label{eq:sphere-product-transfer}
 \overline\Gamma_{b,S^2,k}^{\,d_1+\cdots+d_k}(u_1\times\cdots\times u_k)
 =\overline\Gamma_{b,S^2}^{\,d_1}(u_1)\smile\cdots\smile\overline\Gamma_{b,S^2}^{\,d_k}(u_k).
\end{equation}
This is the same cochain computation as
\eqref{eq:disc-product-transfer}. Permuting the $k$ four-point
configurations also gives \eqref{eq:transfer-symmetrization}.
\end{proof}

Fix $k\ge1$. In an area-coordinate chart of $S^2$, choose $k$ disjoint
open discs $E_j$ and congruent pairs of overlapping round discs
$N_{j,1},N_{j,2}\Subset E_j$, as in the disc construction.
Choose them small enough that their union does not cover the sphere.
Use exactly the radial Hamiltonian isotopies defined before
\eqref{eq:disc-rho}, now extended by the identity on $S^2$.
For their time-one maps put
\begin{equation}\label{eq:sphere-rho}
 \begin{aligned}
 &\rho_{j,\eps}(a)=f_{j,1,\eps},\qquad\rho_{j,\eps}(b)=f_{j,2,\eps},\\
 &\rho_\eps:\F^k\longrightarrow\Ham(S^2,\omega),\qquad\rho_\eps(w_1,\ldots,w_k)=\prod_{j=1}^k\rho_{j,\eps}(w_j).
 \end{aligned}
\end{equation}
The factors commute because their supports are disjoint. Let $B_\eps$ be
the union of the transition annuli and put
$b_\eps=\mu(B_\eps)\to0$. Every time-one map fixes $S^2\setminus B_\eps$
pointwise, and $B_\eps$ is invariant under all the time-one maps.

For each $j$, use the four subsets
$$
\begin{aligned}
 &U_{j,1,\eps}=(N_{j,1}\setminus\overline{N_{j,2}})\setminus B_\eps,
 \quad U_{j,2,\eps}=(N_{j,1}\cap N_{j,2})\setminus B_\eps,\\
 &U_{j,3,\eps}=(N_{j,2}\setminus\overline{N_{j,1}})\setminus B_\eps,
 \quad U_{j,4,\eps}=(S^2\setminus(N_{j,1}\cup N_{j,2}))\setminus B_\eps.
 \end{aligned}
$$
Up to boundary sets of measure zero, they partition $S^2\setminus B_\eps$.
The two distributions used below are illustrated in
Figure~\ref{fig:sphere-pairs}; the points shown avoid the transition annuli.

\begin{figure}[htbp]
\centering
\begin{adjustbox}{max width=\linewidth}
\includegraphics{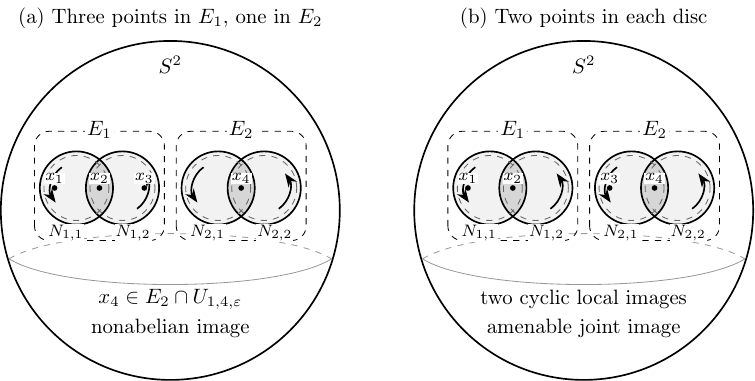}
\end{adjustbox}
\caption{Two copies of the sphere rotation construction, viewed in a chart
with the remaining sphere understood. Solid circles bound the rotating
discs; the strips between solid and dashed circles indicate transition
annuli. In \textup{(a)}, the first three points occupy the three distinct
parts of the first pair, and the fourth lies outside that pair. Thus the
quadruple has type $(1,1,1,1)$ for the first pair, even when its fourth
point lies in the second construction. In \textup{(b)}, each pair contains
two points. Its local image is cyclic, and the two images commute;
the joint image is amenable. Neither picture represents coincident points.
The cyclic-image assertions are proved in Lemma~\ref{lem:sphere-approx}.}
\label{fig:sphere-pairs}
\end{figure}

The first three areas are independent of $j$. So is the fourth area, since
$\mu(U_{j,4,\eps})=1-b_\eps-\sum_{s=1}^3\mu(U_{j,s,\eps})$.
Define
\begin{equation}\label{eq:sphere-lambda}
 \lambda_{k,\eps}=4!\prod_{s=1}^4\mu(U_{j,s,\eps}),\qquad
 \lambda_k=\lim_{\eps\to0}\lambda_{k,\eps}>0.
\end{equation}
Positivity follows because all three parts of each overlap have positive
area and $S^2\setminus(N_{j,1}\cup N_{j,2})$ has positive area. The notation
allows these constants to depend on $k$, since the discs may be chosen
smaller when $k$ increases. Put $p_j:\F^k\to\F$ for the $j$th projection.

\begin{lemma}\label{lem:sphere-approx}
With the choices above, for every $d\ge2$ and
$u\in\rHb^d(\operatorname{Mod}(S^2,4))$,
\begin{equation}\label{eq:sphere-approx}
\left\|\rho_\eps^*\overline\Gamma_{b,S^2}^{\,d}(u)-\lambda_k\sum_{j=1}^k p_j^*\jmath_{S^2}^*u\right\|
\le\bigl(4b_\eps+k|\lambda_{k,\eps}-\lambda_k|\bigr)\|u\|.
\end{equation}
In particular, the left-hand side tends to zero.
\end{lemma}

\begin{proof}
Kimura's computation for a single pair of rotating discs is \cite[Lemma~5.2]{Kimura}.
We give the simultaneous computation, including configurations meeting
several of the disjoint discs $E_j$.
Let $\mathcal E_\eps\subset X_4(S^2)$ consist of quadruples with at least
one point in $B_\eps$. It is invariant under the diagonal action of the time-one maps,
and $\mu^4(\mathcal E_\eps)=1-(1-b_\eps)^4\le4b_\eps$.
For a bounded cocycle $c$ representing $u$, the integral over this set is
a cochain map and has norm at most $4b_\eps\|c\|_\infty$, exactly as in
Lemma~\ref{lem:disc-approx}.

Fix a quadruple $\mathbf x$ outside $\mathcal E_\eps$. Since every time-one map
fixes $\mathbf x$, the map
$w\mapsto\gamma_{S^2}(\rho_\eps(w),\mathbf x)$ is a homomorphism.
We first analyze its restriction to one factor $j$.
A complete rotation of a disc with $r$ marked points represents the
Dehn twist about its boundary in $\operatorname{Mod}(S^2,4)$.
It is trivial if $r=0,1,3,4$, since one side of the boundary contains at
most one marked point. Only $r=2$ can give a nontrivial twist.
These are statements in the central quotient: a full twist on three or
four strands can be nontrivial in $P_4(S^2)$.

Let $n_s$ be the number of points of $\mathbf x$ in $U_{j,s,\eps}$.
The two rotation discs contain $n_1+n_2$ and $n_2+n_3$ points.
If either number is not two, at least one generator is trivial and the
local image is cyclic. If both numbers equal two, then
$$
(n_1,n_2,n_3,n_4)=(2,0,2,0),\quad(1,1,1,1),\quad\text{or }(0,2,0,2).
$$
In the first case the overlap contains no marked point; in the last case
each exclusive part contains no marked point. The two round boundary
curves meet in two points, and in either case a puncture-free bigon can be
removed. They become disjoint essential curves on the four-punctured
sphere and hence are isotopic. The two positive Dehn twists are consequently equal.
Their image is cyclic. This verifies directly the two exceptional cyclic
types in \cite[proof of Lemma~5.2]{Kimura}.

For type $(1,1,1,1)$, the first twist encloses the points in regions one
and two and the second encloses those in regions two and three. The two
boundary curves, their four complementary regions, and their marked
points admit an orientation-preserving identification with the standard
model of the sphere with four marked points. The two twists are thereby identified
simultaneously with $\jmath_{S^2}(a)$ and $\jmath_{S^2}(b)$.
An arbitrary labeling or closing-path choice changes this by a common
conjugation in the full group $\operatorname{Mod}(S^2,4)$.
Thus the local homomorphism is conjugate to $\jmath_{S^2}$, as also shown
in \cite[Lemma~5.2]{Kimura}.

Now restore all $k$ factors. A quadruple of type $(1,1,1,1)$ for the $j$th
pair has three points in $E_j$ and only one outside $E_j$. Every other
$E_\ell$ therefore contains at most one point. Rotations supported there
give the trivial mapping class. The full associated homomorphism is
consequently
\begin{equation}\label{eq:sphere-tracing}
 w\longmapsto\eta_{\mathbf x}\,\jmath_{S^2}(w_j)\,\eta_{\mathbf x}^{-1}
 \quad\text{for some }\eta_{\mathbf x}\in\operatorname{Mod}(S^2,4).
\end{equation}
In particular, the sets of these quadruples for different $j$ are disjoint.
If no pair has type $(1,1,1,1)$, every local image is cyclic. The local
images commute, since the surface homomorphisms 
commute and the cocycle defines a homomorphism at $\mathbf x$. 
Their joint image is therefore abelian. This includes the distribution of two points in each of two
different $E_j$, see Figure~\ref{fig:sphere-pairs}(b).

For each $j$, the quadruples of type $(1,1,1,1)$ form the union of $4!$
labelled product sets, with total measure $\lambda_{k,\eps}$.
Because $u$ is a class on the full mapping class group, conjugation acts
trivially. Integrating its bounded
cochain homotopy, whose bound is independent of the conjugator
\cite[Lemma~A.2]{Loh}, makes the contribution of this set equal to
$\lambda_{k,\eps}p_j^*\jmath_{S^2}^*u$.

All remaining quadruples outside $\mathcal E_\eps$ give homomorphisms with amenable image, 
so their positive-degree bounded-cohomology contribution is zero.
To justify this also for the union, partition by the tuple of images of
the generators of $\F^k$. There are countably many such tuples and the
parts are measurable. Each part has one homomorphism with amenable image.
Finite partial sums give the zero class; the norm of the remaining integral
is at most its measure times $\|c\|_\infty$ and tends to zero. Hence the
whole contribution is zero in reduced bounded cohomology.
We have proved
$$
\left\|\rho_\eps^*\overline\Gamma_{b,S^2}^{\,d}(u)-\lambda_{k,\eps}\sum_jp_j^*\jmath_{S^2}^*u\right\|
\le4b_\eps\|u\|.
$$
Each restriction has norm at most one; replacing
$\lambda_{k,\eps}$ by $\lambda_k$ gives \eqref{eq:sphere-approx}.
\end{proof}

\begin{remark*}[The fixed Brooks class for the sphere]
\label{rem:sphere-brooks-choice}
We choose the fixed class in the sphere proof so that its
restriction to $\F$ is the same Brooks class as in the disc case.

Let $Q_0:\F\to\R$ be the invariant homogeneous quasimorphism
constructed for the disc, so that
$$[\delta Q_0]=[\delta\phi_{ab}]+[\delta\phi_{ba}]\ne0.$$
Put $\Gamma=B_3/\langle\Delta_3^2\rangle,\thinspace\Delta_3^2=(\sigma_1\sigma_2)^3$,
and write $\overline\sigma_1,\overline\sigma_2$ for the
images of the standard braid generators in $\Gamma$.
We identify $\F$ with the subgroup generated by
$a=\overline\sigma_1^{\,2}$ and $b=\overline\sigma_2^{\,2}$.
The invariant quasimorphism $Q_0$ extends to a homogeneous
quasimorphism
$$
\overline Q_3:\Gamma\longrightarrow\R,\qquad\overline Q_3|_{\F}=Q_0,
$$
by the same finite-index extension argument used for the disc.

Let $\delta_1,\delta_2,\delta_3$ be the standard generators
of $B_4(S^2)$. Define $r:B_4(S^2)\longrightarrow\Gamma$,
$$
r(\delta_1)=\overline\sigma_1,\qquad r(\delta_2)=\overline\sigma_2,\qquad
r(\delta_3)=\overline\sigma_1.
$$
The braid relations are preserved. The additional spherical
relation $\delta_1\delta_2\delta_3^2\delta_2\delta_1=1$ is
preserved as well, since
$$
r(\delta_1\delta_2\delta_3^2\delta_2\delta_1)=\overline\sigma_1\overline\sigma_2
\overline\sigma_1^{\,2}\overline\sigma_2\overline\sigma_1=
(\overline\sigma_1\overline\sigma_2\overline\sigma_1)^2=1.
$$
Thus $r$ is a well-defined homomorphism; see
\cite[Section~5.2]{Kimura} for the spherical presentation.
Consequently, $\overline Q_4=\overline Q_3\circ r$
is a homogeneous quasimorphism on $B_4(S^2)$.

Since $\Delta_4^2$ is central of order two, $\overline Q_4$
descends to a homogeneous quasimorphism
$Q_{S^2}:\operatorname{Mod}(S^2,4)\to\R$. Set
$$
\alpha_{S^2}=[\delta Q_{S^2}]\in\rEH^2\bigl(\operatorname{Mod}(S^2,4)\bigr).
$$
The inclusion $\jmath_{S^2}$ sends $a,b$ to the mapping
classes represented by $\delta_1^2,\delta_2^2$.
Their images under $r$ are $a,b$, respectively. Hence
$$
 \theta:=\jmath_{S^2}^*\alpha_{S^2}
 =[\delta Q_0]
 =[\delta\phi_{ab}]+[\delta\phi_{ba}]\ne0.
$$
By \cite[Theorem~A\textup{(b)}]{AB} and bilinearity,
$$
 \theta\smile u=0
 \qquad
 \bigl(u\in\rHb^s(\F),\ s\ge1\bigr).
$$
Choose $z\in Z_2^{\ell^1}(\F)$ with
$\langle\theta,z\rangle=1$, using the duality
recalled in Lemma~\ref{lem:brooks}.

With this choice of $\alpha_{S^2}$, every repeated-coordinate
term in the spherical product calculation vanishes, just as
in the disc case. Thus the projection $\Pi_k$ is unnecessary
in that proof.
\end{remark*}

\begin{proof}[Proof of Theorem~\ref{thm:disc} for the sphere]
Choose $\alpha=\alpha_{S^2}$ and $z$ as in the preceding
remark, and put $\theta=\jmath_{S^2}^*\alpha$.
Then $\langle\theta,z\rangle=1$ and
$$\theta\smile u=0\qquad\bigl(u\in\rHb^s(\F),\ s\ge1\bigr).$$
Use $T_m^q$ from Proposition~\ref{prop:symmetric-products}
with this $\theta,z$. 
Let $m\ge0$, $q=2,3$, $k=m+1$, and $d=2m+q$. For
$\beta\in\rHb^q(\operatorname{Mod}(S^2,4))$, set
$\xi_\beta=\Sym_k(\alpha^{\times m}\times\beta)$.
The class $\xi_\beta$ is symmetric, and its restriction under
$(\jmath_{S^2}^k)^*$ is $T_m^q(\jmath_{S^2}^*\beta)$. Define
\begin{equation}\label{eq:sphere-symmetric-map}
\Psi^{S^2}_{m,q}(\beta)=\frac1{k!}\overline\Gamma_{b,S^2,k}^{\,d}(\xi_\beta).
\end{equation}
Permutation invariance \eqref{eq:transfer-symmetrization} and
\eqref{eq:sphere-product-transfer} identify this map with
$$
\Psi^{S^2}_{m,q}(\beta)=\bigl(\overline\Gamma_{b,S^2}^{\,2}(\alpha)\bigr)^{\smile m}
\smile\overline\Gamma_{b,S^2}^{\,q}(\beta).
$$
Thus we use exactly the normalized symmetric construction of the disc.
Write $b=\jmath_{S^2}^*\beta$ for the remainder of this proof.

\emph{Two factors: $m=1$, $k=2$.}
Choose the two rotation pairs defining
$\rho_\eps:\F^2\to\Ham(S^2,\omega)$ and the constant $\lambda_2>0$
in \eqref{eq:sphere-lambda}. Apply Lemma~\ref{lem:sphere-approx}
separately to $\alpha$ and $\beta$, with the same homomorphism. It gives
$$
\rho_\eps^*\overline\Gamma_{b,S^2}^{\,2}(\alpha)\longrightarrow\lambda_2(p_1^*\theta+p_2^*\theta),\quad
\rho_\eps^*\overline\Gamma_{b,S^2}^{\,q}(\beta)\longrightarrow\lambda_2(p_1^*b+p_2^*b).
$$
The product estimate used in the disc proof justifies multiplication of
these norm limits. Thus
\begin{equation}\label{eq:sphere-two}
\rho_\eps^*\Psi^{S^2}_{1,q}(\beta)\longrightarrow
\lambda_2^2\bigl[p_1^*(\theta\smile b)+\theta\times b+b\times\theta+p_2^*(\theta\smile b)\bigr].
\end{equation}

The first and fourth terms in \eqref{eq:sphere-two}
vanish because $\theta\smile b=0$.
The two mixed terms have the same positive sign, since
$(-1)^{2q}=1$. Therefore
\begin{equation}\label{eq:sphere-two-symmetric}
 \rho_\eps^*\Psi^{S^2}_{1,q}(\beta)\longrightarrow\lambda_2^2(\theta\times b+b\times\theta)
 =\lambda_2^2T_1^q(b).
\end{equation}
If the transferred class is zero, this identity gives
$T_1^q(b)=0$. Proposition~\ref{prop:symmetric-products}
gives $b=0$, and injectivity of $\jmath_{S^2}^*$ gives $\beta=0$.

\emph{An arbitrary number of factors.}
For $k=m+1$, Lemma~\ref{lem:sphere-approx} and the cup-product norm
estimate give
$$
\rho_\eps^*\Psi^{S^2}_{m,q}(\beta)\longrightarrow\lambda_k^k\left(\sum_{j=1}^kp_j^*\theta\right)^{\smile m}
\smile\left(\sum_{j=1}^kp_j^*b\right).
$$

The constants may depend on $k$.
In the expansion there are $k=m+1$ labelled factors:
$m$ copies of $\theta$ and one copy of $b$.
If an assignment repeats a coordinate, that coordinate
contains a factor $\theta$ and another positive-degree
factor. Their internal cup product vanishes.
Thus only bijective assignments remain, and their sum is
$\Sym_k(\theta^{\times m}\times b)=T_m^q(b)$.
Consequently,
\begin{equation}\label{eq:sphere-symmetric-limit}
\rho_\eps^*\Psi^{S^2}_{m,q}(\beta)\longrightarrow\lambda_k^kT_m^q(b).
\end{equation}

If $\Psi^{S^2}_{m,q}(\beta)=0$, the norm limit in
\eqref{eq:sphere-symmetric-limit} gives $T_m^q(b)=0$.
Since $\lambda_k>0$, Proposition~\ref{prop:symmetric-products}
and injectivity of $\jmath_{S^2}^*$ imply $\beta=0$.
Hence every $\Psi^{S^2}_{m,q}$ is injective.
The domains in degrees two and three are infinite-dimensional
by Lemma~\ref{lem:sphere-cohomology}, so the degree patterns
$2m+2$ and $2m+3$ give infinite-dimensional reduced bounded
cohomology in every degree at least two.

For the smooth groups, $\alpha$ and $\beta$ are reduced exact by
Lemma~\ref{lem:sphere-cohomology}. Permutation pullbacks preserve
exactness, so every $\xi_\beta$ is reduced exact. Its normalized
transfer in \eqref{eq:sphere-symmetric-map} is reduced exact by
Lemma~\ref{lem:spheretransfer}. For homeomorphisms we use only bounded
cochains and assert the reduced bounded conclusion. The testing maps
\eqref{eq:sphere-rho} take values in $\Ham(S^2,\omega)$, so the same
calculation applies to all four groups in the theorem.
\end{proof}

\section{Proof of Theorem ~\ref{thm:surface}}
\label{sec:5}

\medskip
We first establish the local
three-point restriction and transfer statements used for every positive
genus. We then prove the theorem for $g\ge2$, and finally treat the
two-point construction needed for the torus.

Let $S=\Sigma_g$, where $g\ge1$, with normalized area measure $\mu$.
Let $\mathcal T_S$ denote any of the four area-preserving transformation
groups defined at the beginning of the paper.
Recall the group $K(S,3)$ defined before Theorem~\ref{thm:surface}.
The Birman exact sequence identifies it with
$$
 K(S,3)\cong
 \begin{cases}
 B_3(S),&g\ge2,\\
 B_3(T^2)/Z(B_3(T^2)),&g=1;
 \end{cases}
$$
see \cite[Theorem~2.4]{Kimura}. Choose an embedded disc containing the
three base points. Extension by the identity gives an inclusion
$\iota:B_3\hookrightarrow K(S,3)$; see \cite[Corollary~4.2(iii)]{ParisRolfsen}
Recall the inclusion $\jmath:\F\hookrightarrow B_3$ defined before
Lemma~\ref{lem:braid-cohomology}. Orient and mark every copy of the
three-punctured disc in the same way.

\begin{lemma}\label{lem:surface-bbf}
There is an infinite-dimensional subspace
$E_S\subset\rEH^2(K(S,3))$ on which $\jmath^*\iota^*$ is injective.
One can choose $\alpha\in E_S$ and $z\in Z_2^{\ell^1}(\F)$ such that
\begin{equation}\label{eq:surface-bbf-normalization}
 \eta=\jmath^*\iota^*\alpha\ne0,
 \qquad \langle\eta,z\rangle=1.
\end{equation}
\end{lemma}
\begin{proof}
Bestvina--Bromberg--Fujiwara \cite[Proposition~4.9]{BBF} implies
that restriction from the 3 puncured mapping class group 
$\operatorname{Mod}(S,3)$ to its disc subgroup
has infinite-dimensional image on homogeneous quasimorphisms modulo homomorphisms
and bounded functions, which is $EH_b^2$. This application to disc braid
subgroups also appears in \cite[Proposition~2.6 and Corollary~2.7]{BrandenburskyBi}
for ambient genus greater than one. For the marked torus we apply
\cite[Proposition~4.9]{BBF} directly to the same three-punctured disc.

Restriction factors through $K(S,3)$. Hence the image of
$$
 EH_b^2(K(S,3))\xrightarrow{\ \iota^*\ }EH_b^2(B_3)
$$
is infinite-dimensional. The degree-two seminorm is a norm, by
\cite[Theorem~6.5]{Frigerio}, so this conclusion is unchanged on passing
to reduced exact bounded cohomology. Lemma~\ref{lem:braid-cohomology}
makes $\jmath^*$ isometric onto its image. Consequently the composite
$\jmath^*\iota^*$ still has infinite-dimensional image. Choose linearly
independent classes in that image, choose one preimage for each, and let
$E_S$ be their algebraic span. The composite is injective on $E_S$.

Choose $0\ne\alpha\in E_S$. Its restriction $\eta$ is nonzero, and the
degree-two duality in Lemma~\ref{lem:brooks} supplies an $\ell^1$-cycle
with nonzero evaluation on $\eta$. Rescaling this cycle gives
\eqref{eq:surface-bbf-normalization}.
\end{proof}

Let $X_3(S)$ denote the ordered configuration space, with measure
$\nu=\mu^3|_{X_3(S)}$, and let $\nu_k=\nu^{\otimes k}$ on $X_3(S)^k$.
Both measures have mass one. Choose measurable closing paths from a base
configuration. Tracing an isotopy defines the endpoint cocycle
$$
 \gamma_{S,3}:\mathcal T_S\times X_3(S)\longrightarrow K(S,3).
$$
Its values lie in the image of the pure braid group. The central quotient
makes it independent of the chosen isotopy: the evaluation of a loop of
unmarked isotopies is central and is in the kernel of the Birman map.
Equivalently, the cocycle is obtained directly from that map applied to
the tracing braid; see \cite[Theorem~2.4 and Section~3.1]{Kimura}.
It satisfies $\gamma_{S,3}(fg,\mathbf x)=\gamma_{S,3}(f,g\mathbf x)\gamma_{S,3}(g,\mathbf x)$.
For $\mathbf X=(\mathbf x^{(1)},\ldots,\mathbf x^{(k)})$, define
$$
\gamma_{S,3,k}(f,\mathbf X)=\bigl(\gamma_{S,3}(f,\mathbf x^{(1)}),\ldots,
\gamma_{S,3}(f,\mathbf x^{(k)})\bigr).
$$
As in Lemma~\ref{lem:disctransfer}, different triples are independent.

\begin{lemma}\label{lem:surface-config-transfer}
The diagonal action on $X_3(S)^k$ defines a norm non-increasing map
$$
\Gamma_{S,3,k}^d:\rHb^d(K(S,3)^k)\longrightarrow\rHb^d(\mathcal T_S).
$$
Write $\Gamma_{S,3}^d=\Gamma_{S,3,1}^d$. These maps satisfy the
permutation identity \eqref{eq:transfer-symmetrization} with
$X=X_3(S)$. They also satisfy
\begin{equation}\label{eq:surface-config-product}
\Gamma_{S,3,k}^{d_1+\cdots+d_k}(u_1\times\cdots\times u_k)
=\Gamma_{S,3}^{d_1}(u_1)\smile\cdots\smile\Gamma_{S,3}^{d_k}(u_k).
\end{equation}
For the smooth transformation groups they preserve reduced exact bounded
cohomology. The bounded maps are also defined for
$\Homeo_0(S,\mu)$.
\end{lemma}
\begin{proof}
For a tuple $\mathbf f=(f_0,\ldots,f_d)$, write
$\boldsymbol\gamma_k(\mathbf f,\mathbf X)
=(\gamma_{S,3,k}(f_i,\mathbf X))_{i=0}^d$.
For a bounded homogeneous cochain $c$ on $K(S,3)^k$, define
\begin{equation}\label{eq:surface-config-cochain}
 (\Gamma'_{S,3,k}c)(\mathbf f)
   =\int_{X_3(S)^k}c\bigl(\boldsymbol\gamma_k(\mathbf f,\mathbf X)\bigr)
       \,d\nu_k(\mathbf X).
\end{equation}
The integrand is measurable and bounded. In the right-homogeneous
convention of Lemma~\ref{lem:transfer}, the cocycle identity and the
measure-preserving substitution $\mathbf Y=h\mathbf X$ prove right
invariance, exactly as in Lemma~\ref{lem:disctransfer}, and
$\delta\Gamma'_{S,3,k}=\Gamma'_{S,3,k}\delta$.
The estimate $\|\Gamma'_{S,3,k}c\|_\infty\le\|c\|_\infty$ gives the
map on reduced bounded cohomology. 

We recall the product calculation with its independent variables.
Let $c,e$ have degrees $p,q$ and fix
$\mathbf f=(f_0,\ldots,f_{p+q})$. Set
$$
c_{\mathbf f}(\mathbf x)
=c\bigl((\gamma_{S,3}(f_i,\mathbf x))_{i=0}^p\bigr),\\
e_{\mathbf f}(\mathbf y)=e\bigl((\gamma_{S,3}(f_i,\mathbf y))_{i=p}^{p+q}\bigr).
$$
Substitution in \eqref{eq:surface-config-cochain} gives the cochain identity
$$
\begin{aligned}
&(\Gamma'_{S,3,2}(c\times e))(\mathbf f)
=\int_{X_3(S)^2}c_{\mathbf f}(\mathbf x)e_{\mathbf f}(\mathbf y)\,d\nu(\mathbf x)d\nu(\mathbf y)=\\
&\left(\int_{X_3(S)}c_{\mathbf f}\,d\nu\right)\left(\int_{X_3(S)}e_{\mathbf f}\,d\nu\right)=
\bigl((\Gamma'_{S,3}c)\smile(\Gamma'_{S,3}e)\bigr)(\mathbf f).
\end{aligned}
$$
Fubini applies because both functions are bounded and the measure is
finite. Permuting the independent configuration variables gives
\eqref{eq:transfer-symmetrization}, so normalized symmetrization does
not change any transferred class. Repeating the product calculation proves
\eqref{eq:surface-config-product} on reduced bounded cohomology.

For smooth transformations, use the section of the mapping-class-valued
configuration cocycle in \cite[Lemma~2.1]{BMentropy}; see also
\cite[Section~3.2]{Kimura}. For each fixed transformation that cocycle
has finite image outside a null set. On the identity component its
values lie in $K(S,3)$, so the same is true as a $K(S,3)$-valued
cocycle, and coordinatewise on the product. Hence, formula
\eqref{eq:surface-config-cochain} is a finite sum even for an ordinary,
unbounded cochain. It defines an ordinary cochain map with the same
formula on bounded cochains. Thus, if $c=\delta B$ with $c$ bounded,
$$
\Gamma'_{S,3,k}c=\delta\bigl(\Gamma'_{S,3,k}B\bigr).
$$
This proves preservation of exactness, and the norm estimate proves the
reduced assertion. No ordinary configuration-space transfer is asserted
here for arbitrary homeomorphisms.
\end{proof}

We use the disc rotations of Lemma~\ref{lem:disc-approx} in pairwise
disjoint embedded discs $D_1,\ldots,D_k\subset S$. Each $D_j$ contains
two overlapping rotation discs with closures in $\operatorname{int}(D_j)$.
The time-one maps of their Hamiltonian rotations define
$\rho_{j,\eps}:\F\to\Ham(S,\omega)$, extended by the identity outside
$D_j$. Their commuting images define
$$
\rho_\eps(w_1,\ldots,w_k)=\rho_{1,\eps}(w_1)\cdots\rho_{k,\eps}(w_k).
$$
Let $B_\eps$ be the union of all transition collars and put
$b_\eps=\mu(B_\eps)$. Let $U_{j,1,\eps},U_{j,2,\eps},U_{j,3,\eps}$
be, in order, the first exclusive region, the overlap, and the second
exclusive region, with the collars removed. The construction can be
chosen so that
$$
\lambda_{j,\eps}=6\prod_{s=1}^3\mu(U_{j,s,\eps})\longrightarrow\lambda_j>0,\qquad b_\eps\longrightarrow0.
$$
Let $p_j:\F^k\to\F$ denote the coordinate projections.

\begin{lemma}\label{lem:surface-local-approx}
For $d\ge2$ and $u\in\rHb^d(K(S,3))$, the above homomorphisms satisfy
\begin{equation}\label{eq:surface-local-approx}
\left\|\rho_\eps^*\Gamma_{S,3}^d(u)-\sum_{j=1}^k\lambda_j p_j^*\jmath^*\iota^*u\right\|
 \le\left(3b_\eps+\sum_{j=1}^k|\lambda_{j,\eps}-\lambda_j|\right)\|u\|.
\end{equation}
In particular, the expression on the left tends to zero. 
\end{lemma}
\begin{proof}
We check that the configuration decomposition used for the disc remains
valid after inclusion into $S$ and, for the torus, after the central
quotient. The set of triples having a point in $B_\eps$ is invariant
under the time-one maps, with measure at most $3b_\eps$.
Its integral is a cochain map of norm at most $3b_\eps$, by the same
cocycle and change-of-variables argument as in
Lemma~\ref{lem:disc-approx}.

Outside this set every marked point is fixed by every time-one map.
Tracing therefore gives a homomorphism on $\F^k$. A supporting disc
containing no more than one of the selected points gives the trivial
braid: a one-point trajectory contracts inside that disc, without
meeting either of the other selected points. A triple has two or more
points in at most one of the pairwise disjoint discs. Hence at most one
factor $\rho_{j,\eps}$ can contribute nontrivially.

Within that factor the complete rotations give exactly the local braids
listed in Lemma~\ref{lem:disc-approx}. A disc enclosing zero or one point
gives the identity. Enclosing all three gives the full twist, which is
central in the \emph{local} $B_3$; it need not be central in $K(S,3)$.
Enclosing two points gives the twist about that pair. If the two
rotations enclose the same pair, the two curves are isotopic in the
local marked disc. If one generator is trivial or the local full twist,
or both twist about the same pair, the image is abelian. Its image
under inclusion into $K(S,3)$ is still amenable.
The only remaining case has one point in each of the three regions
$U_{j,1,\eps},U_{j,2,\eps},U_{j,3,\eps}$.

In their given order, this case reads $\sigma_1^2,\sigma_2^2$ on the
two generators. A different closing path or label assignment conjugates
both generators by the same element of $K(S,3)$. To identify different
supporting discs, move the marked model disc by an ambient isotopy from
the identity and use its traced braid as the connecting path. Thus these
local inclusions are conjugate to the fixed $\iota$ in $K(S,3)$. 
Consequently the homomorphism
on this region is conjugate to $\iota\jmath p_j$.
There are six label assignments, with total measure
$\lambda_{j,\eps}$. The uniformly bounded inner-automorphism homotopy
\cite[Lemma~A.2]{Loh} removes the measurable conjugators. The contribution
is therefore $\lambda_{j,\eps}p_j^*\jmath^*\iota^*u$.

All other regions, except of possibly the 
region that contain triples having a point in $B_\epsilon$, 
give zero in positive-degree bounded cohomology
because their image is amenable. Combining the regions and 
integrating gives the stated estimate with
$\lambda_{j,\eps}$ in place of $\lambda_j$ and error $3b_\eps\|u\|$.
Pullbacks have norm at most one, so replacing those constants by their
limits proves \eqref{eq:surface-local-approx}.
\end{proof}

\subsection{Hyperbolic surfaces}
\label{sec:5-1}

\medskip
We first prove Theorem~\ref{thm:surface} for $g\ge2$. The essential
neighborhood is fixed once and for all, while the number of additional
supporting discs is unrestricted. The next degree-two splitting lemma
will also be used in the torus proof.

\begin{lemma}\label{lem:degree-two-product-splitting}
Let $p_j:\F^k\to\F$ and $e_j:\F\to\F^k$ denote the coordinate
projections and inclusions. For every $\xi\in\rHb^2(\F^k)$,
\begin{equation}\label{eq:degree-two-product-splitting}
 \xi=\sum_{j=1}^k p_j^*e_j^*\xi.
\end{equation}
\end{lemma}
\begin{proof}

Note that the comparison map
$$
c_{\F^k}^{\,d}:\Hb^d(\F^k)\longrightarrow H^d(\F^k;\R)
$$
vanishes for every $k\ge1$ and $d\ge1$.
Indeed, the ordinary K\"unneth formula shows that
$H_d(\F^k;\R)$ is generated by images of fundamental
classes of $d$-tori; see \cite[Section~3.B]{Hatcher}.
Every bounded class restricts trivially to their
fundamental groups, since these groups are amenable.
Naturality of comparison and the universal coefficient
theorem therefore imply the assertion.
For the same argument in degree two, see
\cite[Remark~5.6 and proof of Theorem~5.15]{FrigerioSisto2023}.

It follows that every degree-two bounded class is represented by
$\delta q$, where $q$ is a homogeneous quasimorphism on $\F^k$, and hence 
$q(xy)=q(x)+q(y)$ whenever $xy=yx$. Consequently,
$$q(g_1,\ldots,g_k)=\sum_{j=1}^k q(e_j(g_j)).$$
Taking coboundaries gives \eqref{eq:degree-two-product-splitting} already
in bounded cohomology and hence in reduced bounded cohomology.
\end{proof}

Let $S=\Sigma_g$, $g\ge2$, and choose a compact one-holed torus
$Q\subset S$ with nonempty complementary interior. Choose a basepoint
in $Q$ and simple meridian and longitude loops giving
$i:\F\hookrightarrow\pi_1(S)$. In a standard surface presentation,
write these as $a_1,b_1$ and complete them to
$$
\pi_1(S)=\left\langle a_1,b_1,\ldots,a_g,b_g\ \middle|\ \prod_{j=1}^g[a_j,b_j]=1\right\rangle.
$$
Define $r:\pi_1(S)\to\F=\langle a,b\rangle$ by
$$
\begin{aligned}
&r(a_1)=a,\quad r(b_1)=b,\quad r(a_2)=b,\quad r(b_2)=a,\\
&r(a_j)=r(b_j)=1\qquad(3\le j\le g).
\end{aligned}
$$
The relator maps to $[a,b][b,a]=1$, so $r$ is a homomorphism and
$r\circ i=\id$. The surface group is center free, since $g\ge2$. 
Thus it is the group $G$ in
Lemma~\ref{lem:transfer}.
For $q=2,3$ define the one-point transfer classes
\begin{equation}\label{eq:mixed-theta}
\Theta_q(v)=\overline\Gamma_{b,1}^{\,q}(r^*v),\qquad v\in\rHb^q(\F).
\end{equation}
Since $H^q(\F;\R)=0$, the map $r^*$ has exact image. The comparison
identity in Lemma~\ref{lem:transfer} gives
$$
\Theta_q(v)\in\rEH^q(\mathcal T_S),\qquad\|\Theta_q(v)\|\le\|v\|.
$$
Here the one-point transfer is defined on the compact surface. Its
finite-image proof uses continuity of isotopies and compactness, not
smoothness or the dimension assumption in Theorem~\ref{thm:main}.
Thus this exactness statement includes $\Homeo_0(S,\mu)$.

\begin{lemma}\label{lem:essential-one-point}
There are homomorphisms
$s_\delta:\F\to\Ham(S,\omega)$ supported in $\operatorname{int}(Q)$
and a number $\kappa>0$, independent of $v$ and $q\in\{2,3\}$, such
that
\begin{equation}\label{eq:essential-one-point}
 s_\delta^*\Theta_q(v)\longrightarrow\kappa v
 \quad\text{in }\rHb^q(\F)\quad(\delta\longrightarrow0).
\end{equation}
\end{lemma}
\begin{proof}
Choose thin annuli $N_a,N_b\subset\operatorname{int}(Q)$ about the two
simple generator loops, with intersection a disc. On each annulus use
coordinates $(\vartheta,p)$ with $\omega=d\vartheta\wedge dp$, where
$\vartheta\in S^1$, and use the compactly supported Hamiltonian
$H_\delta=-p\chi_\delta(p)$. Here $\chi_\delta$ equals one on intervals
exhausting the annulus and is zero near its boundary. Its vector field is
$(\chi_\delta+p\chi'_\delta)\partial_\vartheta$ under the convention
$\iota_{X_H}\omega=-dH$. Its flow makes one complete circuit on the
region where $\chi_\delta=1$. It is the annular Hamiltonian construction of
Lemma~\ref{lem:push} in dimension two. Send the two free
generators to the two time-one maps to obtain $s_\delta$.

Outside the union $C_\delta$ of the transition collars every point is
fixed by both endpoints. The set $C_\delta$ is invariant under the
endpoints: a measure-preserving bijection fixing its complement
pointwise preserves it. Its measure tends to zero. On the fixed overlap
$U_\delta$, the tracing homomorphism is conjugate to $i$. On each
exclusive annular region its image is cyclic, and outside both annuli
it is trivial. This is the one-point calculation of
\cite[Lemma~3.1]{BM}; the simple loops chosen above satisfy its surface
hypothesis. Thus we obtain
$$
\|s_\delta^*\Theta_q(v)-\mu(U_\delta)v\|\le\mu(C_\delta)\|r^*v\|.
$$
The cyclic and trivial images have zero positive-degree bounded
cohomology. Also $i^*r^*=\id$ and both pullbacks are norm
non-increasing, so $\|r^*v\|=\|v\|$. Set
$\kappa=\mu(N_a\cap N_b)>0$. Since $\mu(U_\delta)\to\kappa$, the last
estimate proves \eqref{eq:essential-one-point}.
\end{proof}

Choose $m$ pairwise disjoint closed discs
$D_1,\ldots,D_m\subset S\setminus Q$ and use the local rotations of
Lemma~\ref{lem:surface-local-approx} inside them. The supports commute
with each other and with $s_\delta$. Hence
\begin{equation}\label{eq:mixed-testing-map}
 \begin{aligned}
 &\rho_{\eps,\delta}:\F^{m+1}\longrightarrow\Ham(S,\omega),\\
 &\rho_{\eps,\delta}(g_1,\ldots,g_m,h)
   =\rho_{1,\eps}(g_1)\cdots\rho_{m,\eps}(g_m)s_\delta(h).
 \end{aligned}
\end{equation}
Let $p_j:\F^{m+1}\to\F$ again denote the projections.

\begin{lemma}\label{lem:one-point-factorization}
For the homomorphism \eqref{eq:mixed-testing-map} and every
$v\in\rHb^q(\F)$, $q=2,3$, we have the following identity
\begin{equation}\label{eq:one-point-factorization}
\rho_{\eps,\delta}^*\Theta_q(v)=p_{m+1}^*s_\delta^*\Theta_q(v).
\end{equation}
\end{lemma}
\begin{proof}
Choose the one-point closing path for $x\in D_j$ by first going from the global basepoint to a
fixed point of $D_j$ and then taking a path inside $D_j$. Outside the
finitely many discs use the bounded-length closing paths from
Lemma~\ref{lem:transfer}. The additional local paths may also be chosen
with bounded length. Thus these choices retain the finite-image property.

If an isotopy is supported in $D_j$, its closed trajectory for a point
of $D_j$ is conjugate to a loop entirely inside $D_j$. This loop is
contractible. Therefore its one-point cocycle has value $1$ everywhere. This holds
also for products of isotopies supported in the pairwise disjoint discs,
including points in their transition collars. The isotopy defining
$s_\delta(h)$ is stationary on every $D_j$.

Write $\gamma_{S,1}$ for the one-point cocycle and
$e=\rho_{1,\eps}(g_1)\cdots\rho_{m,\eps}(g_m)$.
Using the concatenated isotopies, the cocycle identity gives, for every
point where the closing paths are defined,
$$
\gamma_{S,1}(e s_\delta(h),x)=\gamma_{S,1}(e,s_\delta(h)x)\gamma_{S,1}(s_\delta(h),x)=\gamma_{S,1}(s_\delta(h),x).
$$
The one-point cocycle is independent of isotopy because the center
of the surface group is trivial. Let $\mathbf w_i=(g_{1,i},\ldots,g_{m,i},h_i)\in\F^{m+1}$.
For a bounded homogeneous cochain $c$ on $\pi_1(S)$, put
$$
c_{\mathbf h}(x)=c\bigl((\gamma_{S,1}(s_\delta(h_i),x))_{i=0}^q\bigr).
$$
The preceding cocycle identity and the transfer formula give
$$
\bigl(\rho_{\eps,\delta}^*\Gamma'_{S,1}c\bigr)
(\mathbf w_0,\ldots,\mathbf w_q)=\int_S c_{\mathbf h}(x)\,d\mu(x)
 =\bigl(p_{m+1}^*s_\delta^*\Gamma'_{S,1}c\bigr)(\mathbf w_0,\ldots,\mathbf w_q).
$$
Here $\Gamma'_{S,1}$ is the one-point cochain map of
Lemma~\ref{lem:transfer} for this choice of closing paths.
Take $c$ to be the pullback by $r$ of a cocycle representing $v$.
This proves \eqref{eq:one-point-factorization} on cochains for the
chosen section, and therefore on reduced bounded cohomology.
Independence of closing paths gives the same cohomology identity for
any section.
\end{proof}

\phantomsection\label{proof:surface}
\begin{proof}[Proof of Theorem~\ref{thm:surface} for $g\ge2$]
Let $S=\Sigma_g$ with $g\ge2$, and fix the one-point maps
\eqref{eq:mixed-theta}. Choose $\alpha,\eta,z$ from
Lemma~\ref{lem:surface-bbf} and put
$A=\Gamma_{S,3}^2(\alpha)\in\rHb^2(\mathcal T_S)$.
The variable class will come from $\rHb^q(\F)$ through the one-point
map, not from degree-three cohomology of a braid group.
For $m\ge0$ and $q\in\{2,3\}$, define $\Phi_{m,q}(v)$ as the
transfer over $X_3(S)^m\times S$ of the symmetric class
$$
\frac1{m!}\bigl(\Sym_m(\alpha^{\times m})\bigr)\times r^*v\in\rHb^{2m+q}(K(S,3)^m\times\pi_1(S)).
$$
Here $\Sym_m$ permutes only the first $m$ group coordinates and fixes
the last one. Since all $m$ fixed factors have degree two and equal
$\alpha$, this class is $\alpha^{\times m}\times r^*v$. The
factor $1/m!$ is the same normalized-symmetrization convention as
\eqref{eq:transfer-symmetrization}, applied to the identical factors.
The cochain calculation below shows
\begin{equation}\label{eq:mixed-phi}
 \begin{aligned}
 &\Phi_{m,q}:\rHb^q(\F)\longrightarrow\rHb^{2m+q}(\mathcal T_S),\\
 &\Phi_{m,q}(v)=A^{\smile m}\smile\Theta_q(v).
 \end{aligned}
\end{equation}
For $m=0$ this is $\Theta_q$. We will prove injectivity, then verify
that every image is reduced exact, including for homeomorphisms.

\emph{The mixed product transfer on cochains.}
The symmetric construction just defined has the same
transfer as the ordered product $\alpha^{\times m}\times r^*v$.
Unlike the other product transfers, its last integration factor is $S$,
not $X_3(S)$. Thus we do not permute it with a configuration factor.
The diagonal action has the product cocycle whose first $m$ entries
are $\gamma_{S,3}$ and whose last entry is $\gamma_{S,1}$, with values
in $K(S,3)^m\times\pi_1(S)$. Its integration is a bounded cochain map
by the same cocycle and identities as \eqref{eq:surface-config-cochain}.
For an explicit verification of the product identity, choose a bounded
$2$-cocycle $c$ representing $\alpha$ and a bounded $q$-cocycle $e$
representing $r^*v$. For fixed
$\mathbf f=(f_0,\ldots,f_{2m+q})$ set
$$
C_{j,\mathbf f}(\mathbf x)=c\bigl((\gamma_{S,3}(f_i,\mathbf x))_{i=2j-2}^{2j}\bigr),\qquad
E_{\mathbf f}(y)=e\bigl((\gamma_{S,1}(f_i,y))_{i=2m}^{2m+q}\bigr).
$$
Writing $\mathbf X=(\mathbf x^{(1)},\ldots,\mathbf x^{(m)})$, Fubini gives
$$
\begin{aligned}
&\int_{X_3(S)^m\times S}\left(\prod_{j=1}^m C_{j,\mathbf f}(\mathbf x^{(j)})\right)
E_{\mathbf f}(y)\,d\nu_m(\mathbf X)d\mu(y)\\
&=\left(\prod_{j=1}^m\int_{X_3(S)}C_{j,\mathbf f}\,d\nu\right)\left(\int_S E_{\mathbf f}\,d\mu\right).
\end{aligned}
$$
The last line is the evaluation of the cup product of the transferred
cochains, with the successive degree intervals displayed in the
definitions of $C_{j,\mathbf f}$ and $E_{\mathbf f}$. All integrands are
bounded. Permuting the first $m$ configuration variables in this
integral leaves it unchanged, exactly as in
\eqref{eq:transfer-symmetrization}. Hence its value also equals the
transfer of the normalized symmetric class used in the definition of
$\Phi_{m,q}$. No permutation involving the last variable is used.
In particular, the mixed bounded transfer is defined for homeomorphisms
without integrating an unbounded primitive.

\emph{Two coordinates: one disc and one essential neighborhood.}
For $m=1$ choose a supporting disc outside $Q$ and use
$\rho_{\eps,\delta}$ from \eqref{eq:mixed-testing-map}. Write
$$
\eta_\eps=\rho_{1,\eps}^*A,\qquad E_\delta=s_\delta^*A,\qquad v_\delta=s_\delta^*\Theta_q(v).
$$
Lemmas~\ref{lem:degree-two-product-splitting}
and~\ref{lem:one-point-factorization} give
$$
\rho_{\eps,\delta}^*A=p_1^*\eta_\eps+p_2^*E_\delta,\qquad
\rho_{\eps,\delta}^*\Theta_q(v)=p_2^*v_\delta.
$$
These are exact identities in reduced bounded cohomology. Their product is
\begin{equation}\label{eq:mixed-two-expansion}
\rho_{\eps,\delta}^*\Phi_{1,q}(v)=\eta_\eps\times v_\delta+1\times(E_\delta\smile v_\delta).
\end{equation}
To remove the second summand, let
$\tau_1:\F^2\to\F^2$ be given by $\tau_1(g,h)=(1,h)$, and put $\Pi_1=\id-\tau_1^*$.
The pullback $\tau_1^*$ annihilates $\eta_\eps\times v_\delta$ and fixes
$1\times(E_\delta\smile v_\delta)$. Therefore
\begin{equation}\label{eq:mixed-two-limit}
\Pi_1\rho_{\eps,\delta}^*\Phi_{1,q}(v)=\eta_\eps\times v_\delta
\longrightarrow\lambda_1\kappa\,\eta\times v.
\end{equation}

The limit follows from $\eta_\eps\to\lambda_1\eta$ and
$v_\delta\to\kappa v$ in Lemmas~\ref{lem:surface-local-approx}
and~\ref{lem:essential-one-point}, since the external product is
bounded and bilinear. If $\Phi_{1,q}(v)=0$,
\eqref{eq:mixed-two-limit} gives $\eta\times v=0$. 
By Lemma~\ref{lem:crossdual} and
$\langle\eta,z\rangle=1$, external product with $\eta$ is injective,
so $v=0$. This proves injectivity for $q=2,3$, without any condition
on the class $E_\delta$.

\emph{An arbitrary number of disc factors.}
For $m\ge2$ choose $m$ pairwise disjoint supporting discs in
$S\setminus Q$. Such a finite collection exists for every $m$;
its areas and the positive constants below may depend on $m$.
Put $\eta_{j,\eps}=\rho_{j,\eps}^*A$ and retain $E_\delta,v_\delta$
as above. The same two lemmas give
\begin{equation}\label{eq:mixed-general-expansion}
\rho_{\eps,\delta}^*\Phi_{m,q}(v)=\left(\sum_{j=1}^m p_j^*\eta_{j,\eps}
+p_{m+1}^*E_\delta\right)^{\smile m}\smile p_{m+1}^*v_\delta.
\end{equation}

\emph{Removing terms with an unused coordinate.}
\phantomsection\label{par:coordinate-projections}
Fix $k\ge1$. For $1\le j\le k$, let $\tau_j:\F^k\to\F^k$ replace
coordinate $j$ by the identity and leave every other coordinate unchanged.
For $0\le r\le k$, set
$$
\Pi_r=(\id-\tau_1^*)\cdots(\id-\tau_r^*),\qquad \Pi_0=\id.
$$
The product $\F^k$ is always specified; $\Pi_r$ acts only
on its first $r$ coordinates. The pullbacks commute and are idempotent,
so $\Pi_r$ is a bounded projection with $\|\Pi_r\|\le2^r$.
A positive-degree class pulled back by the trivial homomorphism is zero.
Therefore $\tau_j^*$ annihilates an external product whose $j$th factor
has positive degree and fixes one whose $j$th factor has degree zero.
It follows that $\Pi_r$ fixes a product having positive degree in all
of its first $r$ coordinates, and annihilates a product with degree zero
in any of those coordinates.

In particular, let $k=m+1$, let $\theta,z$ satisfy the hypotheses of
Proposition~\ref{prop:symmetric-products}, let $v\in\rHb^q(\F)$ with
$q=2,3$, and let $c_1,\ldots,c_k$ be real numbers. Then
$$
\Pi_k\left[\left(\sum_{j=1}^k c_jp_j^*\theta\right)^{\smile m}
\smile\left(\sum_{j=1}^k c_jp_j^*v\right)\right]=\left(\prod_{j=1}^k c_j\right)T_m^q(v).
$$
Here $p_j:\F^k\to\F$ is projection. To verify the identity, expand the
product and regard its $m$ fixed factors as separately labelled. There
are $k$ factors of positive degree to assign to $k$ coordinates. The
projection removes any term with an unused coordinate, so only the
bijective assignments remain. Each has coefficient $\prod_jc_j$, and
their sum is precisely $\Sym_k(\theta^{\times m}\times v)$. No
cup-product vanishing is assumed in this calculation. 
We use $\Pi_m$ acting on the first $m$
coordinates of $\F^{m+1}$. The last coordinate, carrying $v_\delta$,
is kept fixed. The bound is $\|\Pi_m\|\le2^m$. 

Expand the power in \eqref{eq:mixed-general-expansion}. There are
exactly $m$ degree-two factors to distribute. A term using $E_\delta$
or assigning two factors to the same disc coordinate leaves another
disc coordinate unused, and is therefore sent to zero by $\Pi_m$.
The remaining terms assign one factor to each of the first $m$
coordinates. There are $m!$ such assignments. They agree, with no
sign change because all assigned factors have even degree. Hence
\begin{equation}\label{eq:mixed-recovery}
\Pi_m\rho_{\eps,\delta}^*\Phi_{m,q}(v)=m!\,\eta_{1,\eps}\times\cdots\times\eta_{m,\eps}\times v_\delta.
\end{equation}
For $m=2$, \eqref{eq:mixed-recovery} is
$2\,\eta_{1,\eps}\times\eta_{2,\eps}\times v_\delta$;
$\Pi_2$ here acts on the first two coordinates of $\F^3$.
The approximation lemmas and boundedness of the external product show
that the right side converges in norm to
\begin{equation}\label{eq:mixed-limit}
 m!\,\kappa\left(\prod_{j=1}^m\lambda_j\right)\eta^{\times m}\times v.
\end{equation}
If $\Phi_{m,q}(v)=0$, \eqref{eq:mixed-recovery} and the positive
coefficient in \eqref{eq:mixed-limit} give
$\eta^{\times m}\times v=0$ in reduced bounded cohomology.
Lemma~\ref{lem:crossdual} again gives $v=0$.

For the lower norm estimate, there is a 
$\ell^1$-cycle $z^{\times m}$, and
$\langle\eta^{\times m},z^{\times m}\rangle=1$ by
Lemma~\ref{lem:crossdual}. Its product identity and norm bound give
$$
\|\eta^{\times m}\times v\|\ge \frac{\|v\|}{\binom{2m+q}{2m}\|z^{\times m}\|_1}.
$$
Since $\|\Pi_m\|\le2^m$ and pullback is norm non-increasing,
\eqref{eq:mixed-recovery}--\eqref{eq:mixed-limit} yield
$$
\|\Phi_{m,q}(v)\|\ge \frac{m!\,\kappa\prod_{j=1}^m\lambda_j}
{2^m\binom{2m+q}{2m}\|z^{\times m}\|_1}\,\|v\|\qquad(m\ge1).
$$
The denominator is positive by the pairing identity.
For $m=0$, the lower bound is $\kappa\|v\|$.
The upper bound is $\|\Phi_{m,q}(v)\|\le\|A\|^m\|v\|$.

\emph{Exactness, including for homeomorphisms.}
By \eqref{eq:mixed-theta} and Lemma~\ref{lem:transfer}, the reduced
class $\Theta_q(v)$ is the image of an exact bounded-cohomology class.
Choose a bounded cocycle $d_v$ representing that exact class and an
ordinary cochain $B_v$ on $\mathcal T_S$ with $d_v=\delta B_v$. Choose a bounded $2$-cocycle
$c_A$ representing $A$. The cochain $B_v$ may be unbounded, but the
Leibniz identity for ordinary cochains gives
\begin{equation}\label{eq:mixed-exact-primitive}
c_A^{\smile m}\smile d_v=\delta\bigl(c_A^{\smile m}\smile B_v\bigr).
\end{equation}
The cocycle on the left is bounded and represents
$\Phi_{m,q}(v)$, while the expression on the right gives its ordinary
primitive directly on the transformation group. Formula
\eqref{eq:mixed-exact-primitive} proves that the class is reduced exact
without assuming that $A$ is exact. In particular, for homeomorphisms
we never integrate an unbounded cochain over $X_3(S)$.

Both $\rHb^2(\F)$ and $\rHb^3(\F)$ are infinite-dimensional, and
$2m+2,2m+3$, $m\ge0$, exhaust all degrees at least two. The injective
maps in \eqref{eq:mixed-phi} give both conclusions of the theorem.
Every testing homomorphism takes values in $\Ham(S,\omega)$, and all
cocycle calculations apply to each of the four stated transformation
groups. This proves Theorem~\ref{thm:surface} for $g\ge2$.
\end{proof}

\subsection{Torus case}
\label{sec:5-2}
\medskip
We finish the proof of Theorem~\ref{thm:surface} with the torus.
The quotient of its fundamental group by its center is trivial, so the
one-point factor used above cannot be used here. We replace it by
Kimura's two-point transfer. The first task is to keep the annular
construction disjoint from a fixed disc; the second is to check that
adding local disc isotopies does not change the two-point factor in
positive-degree bounded cohomology.

Put $T=T^2$ and use $K(T,2)$ as defined before
Theorem~\ref{thm:surface}. The Birman map and the two-strand braid
identifications give
$$
K(T,2)=B_2(T)/Z(B_2(T)),\qquad P_2(T)/Z(P_2(T))\cong\F.
$$
Let $\jmath_T:\F\hookrightarrow K(T,2)$ denote the inclusion of this
pure quotient. Its index is two. With a base configuration $(z_1,z_2)$,
choose the generators $a,b$ by moving $z_1$ once along a meridian or
longitude in $T\setminus\{z_2\}$, while keeping $z_2$ fixed. Moving the
second point instead gives braids $a_2,b_2$; the simultaneous motions
$aa_2,bb_2$ generate the center of $P_2(T)$. In particular, the free
group here is the fundamental group of a punctured torus, 
see \cite[Section~5.3]{Kimura}.

\begin{lemma}\label{lem:torus-two-cohomology}
For $q=2,3$, the space $\rHb^q(K(T,2))$ is infinite-dimensional and
equals $\rEH^q(K(T,2))$. Restriction
$$\jmath_T^*:\rHb^q(K(T,2))\longrightarrow\rHb^q(\F)$$
is an isometric embedding.
\end{lemma}
\begin{proof}
As recalled in \cite[proof of Corollary~1.4, Section~5.4]{Kimura},
$$K(T,2)\cong\Z/2*\Z/2*\Z/2.$$
This group is non-elementary hyperbolic. Its second bounded cohomology
is infinite-dimensional by \cite{EpsteinFujiwara}, and its third reduced
exact bounded cohomology is infinite-dimensional by
\cite[Corollary~3]{FPS}. The degree-two seminorm is a norm, as recalled
in Lemma~\ref{lem:brooks}. This proves infinite-dimensionality before
using restriction to the free subgroup. The subgroup $\jmath_T(\F)$ is normal of index two in
$K(T,2)$. Hence restriction is injective and isometric
in bounded cohomology by \cite[Theorem~2.1]{Kimura}.
Thus $\|\jmath_T^*u\|=\|u\|$, and passing to the quotients
by zero-seminorm classes gives the asserted isometric
embedding on reduced bounded cohomology.

Ordinary restriction to the subgroup of index two is injective over
$\R$: composition with ordinary transfer is multiplication by two.
Since $H^q(\F;\R)=0$ for $q\ge2$, it follows that
$H^q(K(T,2);\R)=0$ in these degrees. Thus the comparison map is zero,
and all the bounded classes in the statement are exact, also after
reduction.
\end{proof}

Let $X_2(T)$ denote the ordered configuration space of two points, with
measure $\nu_2=\mu^2|_{X_2(T)}$. The cocycle, followed by the Birman map, is
$$
\gamma_{T,2}:\mathcal T_T\times X_2(T)\longrightarrow K(T,2).
$$
It is independent of isotopy after taking this quotient. Its values lie
in the pure subgroup $\jmath_T(\F)$; allowing the full group $K(T,2)$
in the domain of cohomology accounts for the exchange of the two labels.
For a bounded homogeneous cochain $c$ and
$\mathbf f=(f_0,\ldots,f_q)$, define
$$
(\Gamma'_{T,2}c)(\mathbf f)=\int_{X_2(T)}c\bigl((\gamma_{T,2}(f_i,\mathbf x))_{i=0}^q\bigr)\,d\nu_2(\mathbf x).
$$
The cocycle identity, invariance of $\nu_2$, and the differential
show that this is a bounded cochain map, exactly as in
Lemma~\ref{lem:surface-config-transfer}. Its norm is at most one.
Let
\begin{equation}\label{eq:torus-two-transfer}
\Theta_{T,q}:\rHb^q(K(T,2))\longrightarrow\rHb^q(\mathcal T_T)
\end{equation}
denote the induced map. For smooth transformations it preserves
reduced exact bounded cohomology. Indeed, the finite-image result
\cite[Lemma~2.1]{BMentropy} applies also to two marked points. For that
section the same formula is a finite sum for each transformation tuple,
even for an ordinary cochain. Hence it transfers ordinary primitives,
as in Lemma~\ref{lem:surface-config-transfer}. No ordinary two-point
transfer on all cochains is asserted for arbitrary homeomorphisms.
Lemma~\ref{lem:torus-homeo-exact} below gives ordinary primitives for
a specified family of degree-two classes. The bounded transfer is
independent of the closing paths by the uniformly bounded
change-of-section homotopy.

\begin{lemma}[Localized two-point approximation]\label{lem:torus-two-approx}
There are a closed disc $D_0\subset T$, homomorphisms
$s_\delta:\F\to\Ham(T,\omega)$ whose defining isotopies are supported
in $T\setminus D_0$, and a constant $\kappa_T>0$ such that
\begin{equation}\label{eq:torus-two-approx}
\|s_\delta^*\Theta_{T,q}(u)-\kappa_T\jmath_T^*u\|\le e_\delta\|u\|,\qquad e_\delta\longrightarrow0,
\end{equation}
for $q=2,3$ and $u\in\rHb^q(K(T,2))$. The choices of $D_0$, the
homomorphisms, and the constants are independent of $u$ and $q$.
\end{lemma}
\begin{proof}
We use the two-point calculation of \cite[Lemma~5.3]{Kimura}. 
Choose a small closed disc $U$ about $z_1$, and choose meridian
and longitude paths passing through $U$. Outside $U$ their two bands
are disjoint. Choose a closed disc $D_0$ about $z_2$ disjoint from a
fixed small neighborhood of $U$ and these paths.

Adjoining one thin band to $U$ gives a meridional annulus, and adjoining
the other gives a longitudinal annulus. Shrink the widths of the bands
and the boundary collars to zero. This gives nested annuli
$W_{a,\delta}\Subset V_{a,\delta}$ and
$W_{b,\delta}\Subset V_{b,\delta}$, all avoiding $D_0$, and disjoint
open discs $U_{1,\delta},U_{2,\delta}$ with the following properties:
$U_{1,\delta}$ contains $z_1$ and lies in both inner annuli;
$U_{2,\delta}$ contains $D_0$ and is disjoint from the two outer
annuli; and
$$
\begin{aligned}
&\mu(U_{1,\delta})\longrightarrow a_0:=\mu(U)\in(0,1),\qquad\mu(U_{2,\delta})\longrightarrow1-a_0,\\
&b_\delta:=\mu\bigl(T\setminus(U_{1,\delta}\cup U_{2,\delta})\bigr)\longrightarrow0.
\end{aligned}
$$
For clarity, $U_{1,\delta}$ is a slightly smaller disc in $U$.
The complement of the two outer annuli is a disc because their union
is a regular neighborhood of a meridian and longitude meeting in one
disc. Take $U_{2,\delta}$ slightly inside this complementary disc,
retaining $D_0$. The area left outside the two good discs consists of
the shrinking bands and collars, see Figure~\ref{fig:torus-two-approx}. 
This also explains why the unused
disc $D_0$ can be fixed while $\delta$ tends to zero.

\begin{figure}[htbp]
\centering
\begin{adjustbox}{max width=\linewidth}
\includegraphics{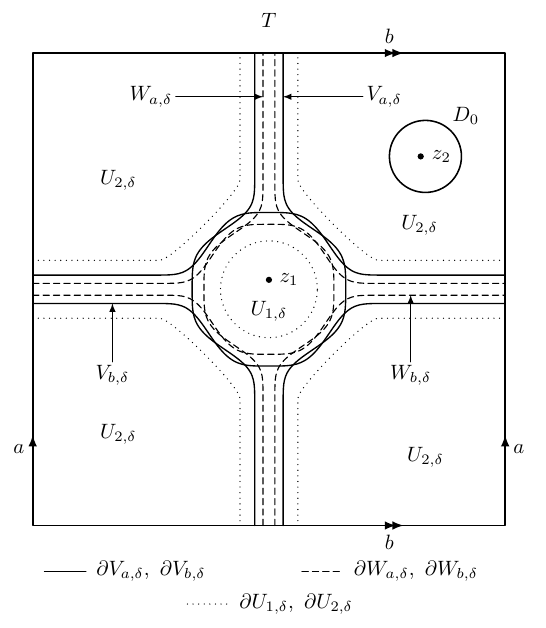}
\end{adjustbox}
\caption{Annuli for the localized two-point construction on $T$.
Solid and dashed curves bound the
outer and inner annuli, respectively; dotted curves bound the
two discs $U_{1,\delta}$ and $U_{2,\delta}$.
The four portions labelled $U_{2,\delta}$ form one disc containing
$D_0$ and $z_2$. The disc $U_{1,\delta}$ contains $z_1$ and lies
in both inner annuli, while $U_{2,\delta}$ is disjoint from both
outer annuli.}
\label{fig:torus-two-approx}
\end{figure}

On each outer annulus choose area coordinates $(\vartheta,p)$,
$\vartheta\in S^1$, in which its inner annulus is a subannulus and
$\omega=d\vartheta\wedge dp$. The area-preserving annulus
parametrization permits these coordinates; equivalently, straighten
the two inner boundary curves first and use area as the transverse
coordinate. Take
$$H_\delta(\vartheta,p)=-p\chi_\delta(p),$$
where $\chi_\delta=1$ on the inner annulus and vanishes near the outer
boundary. With $\iota_{X_H}\omega=-dH$, the vector field is
$(\chi_\delta+p\chi_\delta')\partial_\vartheta$.
Its flow makes one complete positive circuit on the inner annulus
and is stationary off the outer annulus. The transition collars,
including the compensating motion required for a Hamiltonian, are
retained in the exceptional set. Extending $H_\delta$ by zero gives a
Hamiltonian on $T$. Send $a,b$ to the two time-one maps to define
$s_\delta$. Their chosen isotopies are stationary on $D_0$.

For a good two-point configuration, tracing the endpoints 
defines a homomorphism from $\F$. If both points lie in $U_{2,\delta}$,
this homomorphism is trivial. If both lie in $U_{1,\delta}$, one
complete meridional rotation moves both points together and represents
a conjugate of $aa_2$; the longitudinal rotation similarly represents
a conjugate of $bb_2$. These braids are central and hence trivial in
$K(T,2)$. With one point in each disc, the two rotations give,
up to one common conjugation, the generators of $\jmath_T(\F)$.
This is exactly the three-type computation in
\cite[proof of Lemma~5.3]{Kimura}. The simultaneous conjugator is
obtained by identifying the marked meridian and longitude with the
chosen model; it works on both generators and therefore on every word.
Interchanging the two labels conjugates this homomorphism in the full group $K(T,2)$.

The two mixed good regions have total measure
$$
\kappa_{T,\delta}=2\mu(U_{1,\delta})\mu(U_{2,\delta})\longrightarrow\kappa_T:=2a_0(1-a_0)>0.
$$
Their cohomology contribution is
$\kappa_{T,\delta}\jmath_T^*u$. The other good regions contribute
zero. The exceptional configurations have at least one point outside
the two good discs, and measure at most $2b_\delta$.
That set is invariant under the endpoints, since its complement is
fixed pointwise. Its integral is a cochain map of norm at most
$2b_\delta$. Taking the infimum over bounded cocycles representing
$u$ gives
$$
\|s_\delta^*\Theta_{T,q}(u)-\kappa_{T,\delta}\jmath_T^*u\|\le2b_\delta\|u\|.
$$
Thus \eqref{eq:torus-two-approx} holds with
$e_\delta=2b_\delta+|\kappa_{T,\delta}-\kappa_T|$.
\end{proof}

Choose $m$ pairwise disjoint closed discs
$D_1,\ldots,D_m\subset\operatorname{int}(D_0)$. Use the three-point
local rotations of Lemma~\ref{lem:surface-local-approx} in these discs.
Their supports are disjoint from those of $s_\delta$ for every
$\delta$. They give the homomorphism
\begin{equation}\label{eq:torus-mixed-testing-map}
\begin{aligned}
&\rho_{\eps,\delta}:\F^{m+1}\longrightarrow\Ham(T,\omega),\\
&\rho_{\eps,\delta}(g_1,\ldots,g_m,h)=\rho_{1,\eps}(g_1)\cdots\rho_{m,\eps}(g_m)s_\delta(h).
\end{aligned}
\end{equation}
Let $p_j$ denote the coordinate projections of $\F^{m+1}$.

\begin{lemma}[Two-point factorization]\label{lem:torus-two-factorization}
Let $q\in\{2,3\}$ and $u\in\rHb^q(K(T,2))$.
For the homomorphism in \eqref{eq:torus-mixed-testing-map}, we have
\begin{equation}\label{eq:torus-two-factorization}
\rho_{\eps,\delta}^*\Theta_{T,q}(u)=p_{m+1}^*s_\delta^*\Theta_{T,q}(u).
\end{equation}
\end{lemma}
\begin{proof}
Unlike one-point tracing, two-point tracing of an isotopy supported
in a disc need not be trivial. We prove the identity by showing that
this additional contribution is a bounded coboundary.

Discard the null sets with a point on a disc boundary. Let
$Y_j=X_2(\operatorname{int}(D_j))$ for $1\le j\le m$, and let
$Y_0=X_2(T)\setminus\bigcup_{j=1}^mY_j$. These sets are invariant
under the combined action and under $s_\delta$.
Choose closing paths adapted to this partition. On $Y_j$, first follow
a fixed path to a base configuration in $D_j$, then a measurable path
inside $X_2(D_j)$. On $Y_0$, at most one selected point lies in each
disc. Choose its local part of the closing path inside that disc; any
other selected point is kept outside it. If neither point is in a supporting disc, no adaptation is needed.
Otherwise, first place each disc point at a fixed reference point in
its disc and choose the other point, if any, outside those discs.
Then move the disc points along paths in their discs. The last paths
lie in contractible fibres. 

Write $\gamma_2=\gamma_{T,2}$ and
$e=\rho_{1,\eps}(g_1)\cdots\rho_{m,\eps}(g_m)$.
For $\mathbf x\in Y_0$, each local trajectory closes to a contractible
loop in a disc containing no other selected point. Hence
$\gamma_2(e,\mathbf x)=1$. Since $s_\delta$ is stationary on all
the discs, the cocycle identity gives
$$
\gamma_2(e s_\delta(h),\mathbf x)=\gamma_2(s_\delta(h),\mathbf x)\qquad(\mathbf x\in Y_0).
$$
This is an exact cocycle identity on $Y_0$, including trajectories in
the transition collars.

On $Y_j$, only the isotopy supported in $D_j$ moves the two points.
For the chosen paths its cocycle values lie in a fixed conjugate
$C_j$ of the image of $P_2(D_j)\cong\Z$ in $K(T,2)$.
Thus $C_j$ is amenable. Choose a normalized bounded homogeneous
$q$-cocycle $c$ representing $u$. Normalization does not affect its
class, by the standard bounded normalization homotopy. Amenable
vanishing supplies a bounded homogeneous cochain $b_j$ on $C_j$ with
$c|_{C_j}=\delta b_j$; see \cite[Corollary~3.7]{Frigerio}.

For $\mathbf w_i=(g_{1,i},\ldots,g_{m,i},h_i)$, abbreviate
$\gamma_\rho(\mathbf w_i,\mathbf x)=\gamma_2(\rho_{\eps,\delta}(\mathbf w_i),\mathbf x)$.
Define the bounded cochain
$$
\mathcal B_j(\mathbf w_0,\ldots,\mathbf w_{q-1})=
\int_{Y_j}b_j\bigl((\gamma_\rho(\mathbf w_i,\mathbf x))_{i=0}^{q-1}\bigr)\,d\nu_2(\mathbf x).
$$
Its norm is at most $\nu_2(Y_j)\|b_j\|_\infty$. Invariance of
$Y_j$ and the cocycle identity prove homogeneity. Write $I_jc$ for
the same integral with $c$. We obtain
$$
(\delta\mathcal B_j)(\mathbf w_0,\ldots,\mathbf w_q)=\int_{Y_j}(\delta b_j)
\bigl((\gamma_\rho(\mathbf w_i,\mathbf x))_{i=0}^q\bigr)\,d\nu_2(\mathbf x)
=(I_jc)(\mathbf w_0,\ldots,\mathbf w_q).
$$
Combining this with the identity on $Y_0$ proves
$$
\rho_{\eps,\delta}^*\Gamma'_{T,2}c-p_{m+1}^*s_\delta^*\Gamma'_{T,2}c
=\delta\left(\sum_{j=1}^m\mathcal B_j\right).
$$
This is a bounded coboundary. It proves
\eqref{eq:torus-two-factorization} first in bounded cohomology and then
after reduction. Independence of closing paths gives the same class
identity for every section. Ordinary exactness is proved using
a finite-image section for smooth transformations.
\end{proof}

\begin{lemma}[Exact two-point classes for torus homeomorphisms]
\label{lem:torus-homeo-exact}
There is an infinite-dimensional linear subspace
$\mathcal E_T\subset\rHb^2(K(T,2))$ of codimension at most one such that
$$
\Theta_{T,2}(\mathcal E_T)\subset\rEH^2\bigl(\Homeo_0(T,\mu)\bigr).
$$
\end{lemma}
\begin{proof}
Put $K=K(T,2)$ and $\mathcal H=\Homeo_0(T,\mu)$ for this proof.
By Lemma~\ref{lem:torus-two-cohomology}, every class in $\rHb^2(K)$
is exact. It is represented by the coboundary of a homogeneous
quasimorphism $\varphi:K\to\R$. This homogeneous quasimorphism is
unique since $K\cong\mathbb{Z}/2*\mathbb{Z}/2*\mathbb{Z}/2$, 
and hence $\operatorname{Hom}(K,\R)=0$.
There is no further identification on passing to reduced cohomology,
because the degree-two seminorm is a norm.

Let $t\in K$ be the image of a generator of the local subgroup
$B_2(D)$ for a disc $D$ containing the two base points. Define
$$
\mathcal E_T=\{[\delta\varphi]\in\rHb^2(K):
\varphi\text{ is homogeneous and }\varphi(t)=0\}.
$$
Since $\rEH^2(K(T,2))=\rHb^2(K(T,2))$ is
infinite-dimensional and the condition $\varphi(t)=0$
imposes at most one linear equation, $\mathcal E_T$ is
infinite-dimensional.

Lift $\varphi$ to $B_2(T)$ and restrict it to $P_2(T)$.
It vanishes on the center and on the local pure subgroup generated by
$t^2$. Choose a full-measure two-cell whose interior contains $D$.
Its two-strand pure braid group has the same cyclic generator.
Consequently the lift satisfies the hypotheses of
\cite[Theorem~1.1]{BKMS}. In the identification
$P_2(T)\cong\F\times\Z^2$, this local pure generator is the commutator
$[a,b]$ in the free factor; see \cite[Lemma~2.6]{BKS}.

Use the closing paths of \cite[Section~2]{BKMS} and write
$\gamma=\gamma_{T,2}$. Theorem~1.1 and its proof in that paper show that
$$
Q_\varphi(f)=\int_{X_2(T)}\varphi(\gamma(f,\mathbf x))\,d\nu_2(\mathbf x)\qquad(f\in\mathcal H)
$$
is well defined and is a quasimorphism. 

We check that its coboundary represents the bounded transfer already
defined in \eqref{eq:torus-two-transfer}, not merely a class agreeing
with it on the diffeomorphism subgroup. In the right-homogeneous
cochain convention, put
$$
\begin{aligned}
&F_\varphi(k_0,k_1)=\varphi(k_0k_1^{-1})\qquad(k_0,k_1\in K),\\
&C_\varphi(f_0,f_1)=Q_\varphi(f_0f_1^{-1})\qquad(f_0,f_1\in\mathcal H).
 \end{aligned}
$$
These are ordinary homogeneous $1$-cochains and need not be bounded.
The cocycle identity gives
$$
\gamma(f_0,\mathbf x)\gamma(f_1,\mathbf x)^{-1}=\gamma(f_0f_1^{-1},f_1\mathbf x).
$$
Since $f_1$ preserves $\nu_2$, we obtain
$$
 \begin{aligned}
&\int_{X_2(T)}
F_\varphi(\gamma(f_0,\mathbf x),\gamma(f_1,\mathbf x))\,d\nu_2(\mathbf x)
=\int_{X_2(T)}\varphi(\gamma(f_0f_1^{-1},f_1\mathbf x))\,d\nu_2(\mathbf x)\\
&=Q_\varphi(f_0f_1^{-1})=C_\varphi(f_0,f_1).
 \end{aligned}
$$
All three expressions are finite by the preceding integrability statement.
The cochain $\delta F_\varphi$ is bounded and represents
$[\delta\varphi]$. We have
\begin{equation}\label{eq:torus-homeo-primitive}
\begin{aligned}
&(\delta C_\varphi)(f_0,f_1,f_2)
=\int_{X_2(T)}(\delta F_\varphi)(\gamma(f_0,\mathbf x),\gamma(f_1,\mathbf x),
\gamma(f_2,\mathbf x))\,d\nu_2(\mathbf x)\\
&=(\Gamma'_{T,2}\delta F_\varphi)(f_0,f_1,f_2).
\end{aligned}
\end{equation}
Thus \eqref{eq:torus-homeo-primitive} supplies an ordinary primitive
of the transferred bounded cocycle. It proves the asserted exactness
before reduction and hence after reduction. 
\end{proof}

\phantomsection\label{proof:torus}
\begin{proof}[Proof of Theorem~\ref{thm:surface} for the torus]
Choose $\alpha,\eta,z$ from Lemma~\ref{lem:surface-bbf} for $S=T$,
and put $A=\Gamma_{T,3}^2(\alpha)$. Use the two-point maps
\eqref{eq:torus-two-transfer}. For $m\ge0$, $q=2,3$, and
$u\in\rHb^q(K(T,2))$, define
\begin{equation}\label{eq:torus-mixed-map}
\begin{aligned}
&\Phi^{T}_{m,q}:\rHb^q(K(T,2))\longrightarrow\rHb^{2m+q}(\mathcal T_T),\\
&\Phi^{T}_{m,q}(u)=A^{\smile m}\smile\Theta_{T,q}(u).
\end{aligned}
\end{equation}
The variable degree-three classes now come from the full two-strand quotient. 

\emph{The mixed transfer.}
Write $X=X_3(T)$, $Y=X_2(T)$, and use their probability measures
$\nu_3,\nu_2$. The product space $X^m\times Y$ has the diagonal
action and the cocycle with values in $K(T,3)^m\times K(T,2)$
obtained by using $\gamma_{T,3}$ in its first $m$ entries and
$\gamma_{T,2}$ in its last entry. Integration defines a
cochain map, by the cocycle and differential identities in
Lemma~\ref{lem:surface-config-transfer}. We transfer the class
$$
\frac1{m!}\bigl(\Sym_m(\alpha^{\times m})\bigr)\times u.
$$
Only the $m$ identical three-point factors are permuted. The two-point
factor is distinguished. The normalized class equals
$\alpha^{\times m}\times u$, since all the fixed factors have degree two.
For $m=0$, omit these factors and their symmetrization.

For the cochain calculation, choose bounded cocycles $c,e$ representing
$\alpha,u$, and fix $\mathbf f=(f_0,\ldots,f_{2m+q})$. Put
$$
\begin{aligned}
&C_{j,\mathbf f}(\mathbf x)=c\bigl((\gamma_{T,3}(f_i,\mathbf x))_{i=2j-2}^{2j}\bigr),\quad
E_{\mathbf f}(\mathbf y)=e\bigl((\gamma_{T,2}(f_i,\mathbf y))_{i=2m}^{2m+q}\bigr),\\
&C_{\mathbf f}(\mathbf X)=\prod_{j=1}^m C_{j,\mathbf f}(\mathbf x^{(j)}),
\quad\mathbf X=(\mathbf x^{(1)},\ldots,\mathbf x^{(m)}).
\end{aligned}
$$
Fubini gives
$$
\int_{X^m\times Y} C_{\mathbf f}(\mathbf X)E_{\mathbf f}(\mathbf y)\,d\nu_3^{\otimes m}(\mathbf X)d\nu_2(\mathbf y)
=\left(\prod_{j=1}^m\int_X C_{j,\mathbf f}\,d\nu_3\right)\left(\int_Y E_{\mathbf f}\,d\nu_2\right).
$$
The last line is the evaluation of the cup product in
\eqref{eq:torus-mixed-map}. All functions are bounded, 
so this identity also applies to homeomorphisms.

\emph{Two factors: one disc and the annular construction.}
For $m=1$ use \eqref{eq:torus-mixed-testing-map}, and put
$$
\eta_\eps=\rho_{1,\eps}^*A,\qquad E_\delta=s_\delta^*A,\qquad
u_\delta=s_\delta^*\Theta_{T,q}(u).
$$
Lemma~\ref{lem:degree-two-product-splitting} and
\eqref{eq:torus-two-factorization} give the exact class identities
$$
\rho_{\eps,\delta}^*A=p_1^*\eta_\eps+p_2^*E_\delta,\qquad
\rho_{\eps,\delta}^*\Theta_{T,q}(u)=p_2^*u_\delta.
$$
Their product is
$$
\rho_{\eps,\delta}^*\Phi^T_{1,q}(u)=\eta_\eps\times u_\delta
+1\times(E_\delta\smile u_\delta).
$$
Use $\tau_1(g,h)=(1,h)$ and $\Pi_1=\id-\tau_1^*$.
This fixes the first external product and annihilates the second. Therefore
\begin{equation}\label{eq:torus-mixed-two-limit}
\Pi_1\rho_{\eps,\delta}^*\Phi^T_{1,q}(u)=\eta_\eps\times u_\delta
 \hspace{0pt}\longrightarrow\lambda_1\kappa_T\,\eta\times\jmath_T^*u.
\end{equation}
Here \eqref{eq:surface-local-approx} gives
$\eta_\eps\to\lambda_1\eta$, and \eqref{eq:torus-two-approx} gives
$u_\delta\to\kappa_T\jmath_T^*u$. 
External product with $\eta$ is injective by
Lemma~\ref{lem:crossdual}, since $\langle\eta,z\rangle=1$.
Restriction $\jmath_T^*$ is injective by
Lemma~\ref{lem:torus-two-cohomology}. Thus a zero class on the left
of \eqref{eq:torus-mixed-two-limit} forces $u=0$, for both $q=2$ and
$q=3$. No vanishing of $E_\delta$ or of its internal cup products is
assumed.

\emph{An arbitrary number of disc factors.}
For any $m\ge2$ place $m$ disjoint discs in $\operatorname{int}(D_0)$
and put $\eta_{j,\eps}=\rho_{j,\eps}^*A$. The same two lemmas give
$$
\rho_{\eps,\delta}^*\Phi^T_{m,q}(u)=\left(\sum_{j=1}^m p_j^*\eta_{j,\eps}
+p_{m+1}^*E_\delta\right)^{\smile m}\smile p_{m+1}^*u_\delta.
$$
Apply $\Pi_m=\prod_{j=1}^m(\id-\tau_j^*)$ from the
coordinate-projection construction, only on the first $m$ coordinates.
Every term with an unused disc coordinate is annihilated. There are
exactly $m$ degree-two factors to distribute, so a surviving term uses
each disc coordinate once and uses no factor $E_\delta$. The $m!$
surviving assignments have the same positive sign. Hence
\begin{equation}\label{eq:torus-mixed-recovery}
\Pi_m\rho_{\eps,\delta}^*\Phi^T_{m,q}(u)
=m!\,\eta_{1,\eps}\times\cdots\times\eta_{m,\eps}\times u_\delta.
\end{equation}
The limit of
\eqref{eq:torus-mixed-recovery} in reduced norm is
\begin{equation}\label{eq:torus-mixed-limit}
 m!\,\kappa_T\left(\prod_{j=1}^m\lambda_j\right)\eta^{\times m}\times\jmath_T^*u.
\end{equation}
Its coefficient is strictly positive. Lemma~\ref{lem:crossdual} makes
external product with $\eta^{\times m}$ injective, and
$\jmath_T^*$ is injective. Thus
\eqref{eq:torus-mixed-recovery}--\eqref{eq:torus-mixed-limit} prove
injectivity of every $\Phi^T_{m,q}$. For $m=0$, use
\eqref{eq:torus-two-approx} directly.

The norm argument is the same as for $g\ge2$, and we obtain 
nonvanishing after reduction. 

\emph{Exactness and completion of the proof.}
For each smooth transformation group, the two-point transfer preserves
exactness, and Lemma~\ref{lem:torus-two-cohomology} makes its domain
exact. Choose a bounded cocycle $d_u$ representing an exact lift of
$\Theta_{T,q}(u)$ and an ordinary cochain $B_u$ with $d_u=\delta B_u$.
If $c_A$ is a bounded cocycle representing $A$, the ordinary Leibniz
identity gives
$$
c_A^{\smile m}\smile d_u=\delta\bigl(c_A^{\smile m}\smile B_u\bigr).
$$
This proves that the image of \eqref{eq:torus-mixed-map} is reduced
exact for the smooth groups.

For $\mathcal T_T=\Homeo_0(T,\mu)$, first take $q=2$ and
$u=[\delta\varphi]\in\mathcal E_T$ from
Lemma~\ref{lem:torus-homeo-exact}. Equation~\eqref{eq:torus-homeo-primitive}
gives the ordinary homogeneous cochain $C_\varphi$ with
$d_u=\delta C_\varphi$ bounded and representing $\Theta_{T,2}(u)$.
The same computation is now valid on the homeomorphism group:
\begin{equation}\label{eq:torus-even-exact-primitive}
c_A^{\smile m}\smile d_u=\delta\bigl(c_A^{\smile m}\smile C_\varphi\bigr)\qquad(m\ge0).
\end{equation}
For $m=0$ the factor $c_A^{\smile0}$ is the degree-zero unit.
The primitive in \eqref{eq:torus-even-exact-primitive} is an ordinary
cochain on the transformation group. It need not be bounded, and no
unbounded three-point cochain is integrated. Thus every restriction
$\Phi^T_{m,2}|_{\mathcal E_T}$ has reduced exact image.
Its injectivity and lower norm estimate have already been proved
on the entire domain. Since $\mathcal E_T$ is infinite-dimensional,
this gives infinite-dimensional reduced exact bounded cohomology
in every even degree for torus homeomorphisms. We have not proved
exactness of the degree-three two-point factor class on this group.

Both spaces $\rHb^2(K(T,2))$ and $\rHb^3(K(T,2))$ are
infinite-dimensional by Lemma~\ref{lem:torus-two-cohomology}. The
degrees $2m+2$ and $2m+3$, $m\ge0$, exhaust all integers at least two.
Thus the maps \eqref{eq:torus-mixed-map} give reduced bounded
infinite-dimensionality in every degree, reduced exact
infinite-dimensionality in every degree for the smooth groups,
and the additional even-degree reduced exact conclusion for
homeomorphisms. All testing homomorphisms are Hamiltonian, and
the same cochain computations apply directly to each of the four
groups. Together with the preceding proof for $g\ge2$, this proves
Theorem~\ref{thm:surface}.
\end{proof}

\section{Ordinary cohomology, final comments and remarks}
\label{sec:6}

\subsection{Ordinary cohomology}
\label{sec:6-1}
For completeness, we record the following consequence of the
classical mass-flow construction
\cite[Proposition~5.1 and Section~6]{Fathi}.
Combined with the ordinary cohomology of $\R$ regarded as
a discrete group, it gives infinite-dimensionality in every
positive degree under the hypothesis below.
In the orientable setting, this yields a strengthening of
the ordinary-cohomology statement in
\cite[Theorem~4.5]{Nitsche}.

\begin{theorem}[Ordinary cohomology of volume-preserving groups]
\label{thm:ordinary-cohomology}
Let $M$ be a compact connected oriented smooth manifold of
dimension $n\ge2$, possibly with boundary, and let $\mu$ be
the probability measure induced by a smooth positive volume
form. Set $G=\pi_1(M)/Z(\pi_1(M))$,
and suppose that
\begin{equation}\label{eq:ordinary-character}
 H^1(G;\R)\ne0.
\end{equation}
Let $\cG$ denote either $\Homeo_0(M,\mu)\qquad\text{or}\qquad\Diff_0(M,\mu)$.

Let $\R^\delta$ denote the additive group of real numbers
with the discrete topology. There are group homomorphisms
$$
F:\cG\longrightarrow\R^\delta,\qquad s:\R^\delta\longrightarrow\cG
$$
such that
\begin{equation}\label{eq:ordinary-splitting}
 F\circ s=\id_{\R^\delta}.
\end{equation}
The section $s$ can be chosen to be a smooth
volume-preserving flow, supported in $\Int M$ when the
boundary is nonempty.

Consequently, pullback gives an embedding
$$
F^*:H^*(\R^\delta;\R)\longrightarrow H^*(\cG;\R),
$$
and
$$
\dim H^d(\cG;\R)=\infty\qquad\text{for every }d\ge1.
$$
All group cohomology in this statement is ordinary
cohomology of discrete groups with trivial real coefficients.
\end{theorem}

\begin{proof} We first construct a real-valued homomorphism on
$\Homeo_0(M,\mu)$ and a section whose image consists of
smooth volume-preserving diffeomorphisms.

Since $G$ is finitely generated, condition
\eqref{eq:ordinary-character} gives a nonzero homomorphism
$G\to\Z$. Composing with the quotient map, obtain
$$
\chi:\pi_1(M)\longrightarrow\Z,\qquad\chi|_{Z(\pi_1(M))}=0.
$$
Choose a smooth map $u:M\to S^1$ representing $\chi$.
For an isotopy $\{f_t\}_{t\in[0,1]}$ from the identity to
$f\in\Homeo_0(M,\mu)$, let $\Delta_f(x)$ be the endpoint
of the real lift, starting at zero, of the path
$$t\longmapsto u(f_t(x))-u(x)\in S^1.$$
The function $\Delta_f:M\to\R$ is continuous.

Changing the isotopy changes this endpoint by the value
of $\chi$ on the evaluation of a loop of isotopies.
Such an evaluation lies in the center of $\pi_1(M)$;
see \cite[Proposition~2.1]{BM}. Hence $\Delta_f$ is
independent of the chosen isotopy. Concatenation gives
$$\Delta_{fg}(x)=\Delta_f(gx)+\Delta_g(x).$$
The scalar mass-flow construction
\cite[Section~5]{Fathi} therefore defines
$$F(f)=\int_M\Delta_f(x)\,d\mu(x).$$
The integral is finite by compactness. Since $g$ preserves
$\mu$, the preceding identity gives
$$F(fg)=F(f)+F(g).$$
Thus $F:\Homeo_0(M,\mu)\to\R$ is a group homomorphism.

We next construct a section of $F$.
Write $d\mu=\Omega$, where $\Omega$ is a smooth positive
volume form with $\int_M\Omega=1$, and put
$\alpha=u^*(d\vartheta)$,
where $d\vartheta$ is the standard closed one-form on
$ S^1$ with integral one. The class $[\alpha]$ is
nonzero. By Poincare--Lefschetz duality, choose a smooth
closed $(n-1)$-form $\beta$, compactly supported in
$\Int M$ when $\partial M\ne\varnothing$, such that
$$\int_M\alpha\wedge\beta=1.$$

Define the smooth vector field $X$ by $\iota_X\Omega=\beta$.
Since $d\beta=0$, Cartan's formula gives
$\mathcal L_X\Omega=0$. Thus the flow $\{P_t\}_{t\in\R}$
of $X$ preserves $\Omega$. It exists for all real times:
$M$ is compact, and $X$ vanishes near the boundary when
the boundary is nonempty.

For this flow, we have
$$
\Delta_{P_t}(x)=\int_0^t\alpha(X)(P_\tau(x))\,d\tau.
$$
Using preservation of $\mu$ and
$\alpha(X)\Omega=\alpha\wedge\iota_X\Omega$, we obtain
$$
F(P_t)=\int_M\int_0^t\alpha(X)(P_\tau(x))\,d\tau\,d\mu(x)
=t\int_M\alpha(X)\Omega=t\int_M\alpha\wedge\beta=t.
$$
Consequently, $s(t)=P_t$ is a group homomorphism with
$F\circ s=\id$. Its image lies in $\Diff_0(M,\mu)$.
The same section therefore works for the restriction
of $F$ to $\Diff_0(M,\mu)$.

For either choice of $\cG$, functoriality gives
$$s^*\circ F^*=(F\circ s)^*=\id\qquad\text{on }H^*(\R^\delta;\R).$$
Thus $F^*$ is an embedding. Now $H^d(\R^\delta;\R)$ is
infinite-dimensional for every $d\ge1$ and the proof follows.
\end{proof}

For a closed negatively curved manifold of dimension at
least two, the fundamental group has trivial center.
Hence $H^1(G;\R)\cong H^1(M;\R)$. If $b_1(M)>0$, 
Theorem~\ref{thm:ordinary-cohomology} therefore implies.

\begin{corollary}[Negatively curved manifolds]
\label{cor:negative-curvature-ordinary}
Let $M$ be a closed connected oriented smooth manifold of
dimension $n\ge2$ admitting a Riemannian metric of strictly
negative sectional curvature. If $n>2$ suppose that
$$b_1(M)=\dim H^1(M;\R)>0.$$
Then, we have
$$
\dim \cG^d\bigl(\Homeo_0(M,\mu);\R\bigr)=\infty,\quad
\dim \cG^d\bigl(\Diff_0(M,\mu);\R\bigr)=\infty\quad(d\ge1).
$$
\end{corollary}

\subsection{Final comments and remarks}
\label{sec:6-2}

The following proposition is a consequence of
\cite[Corollary~4.25]{MO} and
\cite[Theorem~6.14(c)]{DGO}, together with the standard
fixed-point argument for finite groups acting on trees.

\begin{proposition}\label{prop:AtoH}
Every group satisfying \textup{(i)} is acylindrically
hyperbolic and has no nontrivial finite normal subgroup.
In particular, it satisfies \textup{(iii)}.
The hyperbolically embedded free subgroup need not coincide
with the specified $\F$-vertex in \textup{(i)}.
\end{proposition}

\begin{proof}
Write $G=\F*_A B$. Since $A$ is an amenable subgroup of
$\F$, it is trivial or infinite cyclic. In particular,
it is finitely generated and has infinite index in $\F$.
By \cite[Corollary~4.25]{MO}, $G$ is either infinite cyclic
or acylindrically hyperbolic. The first possibility is
excluded because $G$ contains $\F$.

We show that $G$ has no nontrivial finite normal subgroup.
If the inclusion $A\to B$ is onto, then $G\cong\F$ and
this is immediate. Otherwise both edge inclusions are
proper, and the Bass--Serre action on the associated
tree $T$ is minimal; see \cite[Lemma~4.14]{MO}.
Let $N\lhd G$ be finite. Its fixed-point set $T^N$ is a
nonempty subtree; see \cite[Chapter~I, Section~6]{Serre}.
Normality makes $T^N$ invariant under $G$, so minimality
gives $T^N=T$. Hence $N$ lies in every edge stabilizer,
and in particular in $A$. Since $A$ is torsion-free,
$N=1$.

By \cite[Theorem~1.2]{Osin} and
\cite[Theorem~6.14(c)]{DGO}, an acylindrically hyperbolic
group with no nontrivial finite normal subgroup contains
a hyperbolically embedded subgroup isomorphic to $\F$.
Thus \textup{(iii)} holds. This conclusion does not assert
that the specified $\F$-vertex is hyperbolically embedded.
\end{proof}

\begin{remark}\label{rem:wholeV2}
For $n\ge3$, Proposition~\ref{prop:AtoH} shows that 
every group covered by Theorem~\ref{thm:main} 
satisfies \textup{(ii)} or \textup{(iii)}. Thus, 
allowing a different free embedding and algebraic 
rather than isometric lifts, the full space $\rHb^2(\F)$ 
can be injected into $\rEH^2(\cG)$ in every case. 
The possible codimension-one loss in the direct \textup{(i)} 
construction concerns that specific BBFIPP section and vertex; 
it is not a loss in the all-degree conclusion or an obstruction to another exact section.
\end{remark}

\begin{corollary}\label{cor:AH}
The all-degree conclusion of Theorem~\ref{thm:main} 
for $n\ge3$ also holds whenever $G=\pi/Z(\pi)$ is acylindrically hyperbolic, 
even if $G$ has a nontrivial finite normal subgroup and does 
not satisfy the condition \textup{(iii)}. In particular,
the conclusion holds for any compact orientable smooth $M$ of 
dimension at least three for which $G$ is a non-elementary 
torsion-free word-hyperbolic group.
\end{corollary}

\begin{proof}
By \cite[Theorem~1.2]{Osin} and \cite[Theorem~6.14(c)]{DGO}, 
there is a hyperbolically embedded subgroup $H=\F\times K$ with $K$ finite. 
Let $p:H\to\F$ be the projection, $s:\F\to H$ the direct-factor inclusion, 
and $j:H\hookrightarrow G$. Put $i=j\circ s$; this is an injective homomorphism.

The reduction through $\F\times K$ is the one used 
in \cite[Theorem~5.4 and Corollary~5.5]{FPS} and 
\cite[proof of Theorem~B]{BM}. Fix $q\in\{2,3\}$ 
and $u\in H_b^q(\F)$. Since $H^q(\F;\R)=0$, the 
class $u$ is exact. Naturality of comparison gives
$$c_H^q(p^*u)=p^*c_{\F}^q(u)=0.$$
The exact extension theorem \cite[Proposition~5.1]{FPS}, 
applied to $j:H\hookrightarrow_hG$, therefore gives $v\in EH_b^q(G)$ with $j^*v=p^*u$. 
Restricting further gives
$$i^*v=s^*j^*v=s^*p^*u=(p\circ s)^*u=u.$$
Given a class in $\rHb^q(\F)$, 
choose a representative $u\in H_b^q(\F)$ and take the image of $v$ 
in $\rEH^q(G)$. This proves that $i^*:\rEH^q(G)\to \rHb^q(\F)$ is onto. 
Choose algebraic linear right inverses as in Lemma~\ref{lem:othersections}.

Apply Lemma~\ref{lem:brooks} and Propositions~\ref{prop:symmetric-products} 
and~\ref{prop:approx} to this inclusion and the two right inverses. 
The proof of Theorem~\ref{thm:main} for $n\ge3$ then applies without change. 
Thus all four transformation-group conclusions follow.
\end{proof}

\begin{remark}[Comparison with Nitsche]\label{prop:nitsche}
For compact Riemannian manifolds of dimension at least three,
\cite[Theorem~4.2]{Nitsche} proves
$$
 \dim\rEH^d\bigl(\Homeo_{\mathrm{vol},0}(M)\bigr)=\infty
 \quad\text{for }d=3\text{ and even }d\ge2,
$$
assuming $\pi_1(M)\twoheadrightarrow\F$. This is exactly condition
\textup{(ii)}. On the compact oriented manifolds considered
here, Theorem~\ref{thm:main} also gives the degrees $5,7,9,\ldots$ and
allows condition \textup{(iii)}, with the four groups and support
conventions in its statement.

Strictness of the enlargement of the manifold hypothesis follows
from the third construction in the proof of
Proposition~\ref{prop:relations}. It gives closed and compact
manifolds with boundary satisfying \textup{(iii)} but not \textup{(ii)}. 
That example gives closed and boundary
manifolds satisfying \textup{(iii)} but not \textup{(ii)}. They are
therefore covered by Theorem~\ref{thm:main} and not by the hypothesis of
Nitsche's Theorem~4.2, even in its even-degree range. This comparison uses
the stated theorem and the example; it requires no additional
nonvanishing argument. The earlier degreewise transfer results in
\cite[Theorems~A and~B]{BM} remain the one-point antecedents of our
construction, so the comparison is not a claim of novelty for every
individual consequence.
\end{remark}

\begin{proposition}\label{prop:relations}
Among the three displayed conditions, the only universal implication
between distinct conditions is \textup{(i)}$\Rightarrow$\textup{(iii)}.
In particular, \textup{(ii)} and \textup{(iii)} are independent, and
\textup{(i)} and \textup{(ii)} are independent, but the three conditions
are not pairwise independent.
\end{proposition}

\begin{proof}
Proposition~\ref{prop:AtoH} gives
\textup{(i)}$\Rightarrow$\textup{(iii)}.
We establish all the other nonimplications by the following three
constructions.

\emph{Condition \textup{(i)} does not imply \textup{(ii)}.}
Let $Q$ be a closed oriented hyperbolic integral homology
three-sphere, and put $P=\pi_1(Q)$. Hence $P$ is infinite, 
torsion-free, and perfect. Fix $1\ne t\in P$ and set
$G_+=P*\langle b\rangle$.
Eliminating the generator $a$ gives
\begin{equation}\label{eq:primitiveamalgam}
 G_+\cong F(a,b)*_{\langle a\rangle=\langle t\rangle}P.
\end{equation}
Both edge maps are injective because $t$ has infinite order.
Thus $G_+$ satisfies \textup{(i)}, and hence \textup{(iii)} by
Proposition~\ref{prop:AtoH}. The free-product normal form gives
$Z(G_+)=1$; see \cite[Chapter~I, Section~4]{Serre}.
On the other hand,
$$(G_+)_{\mathrm{ab}}\cong\Z.$$
An epimorphism $G_+\twoheadrightarrow\F$ would induce an epimorphism
$\Z\twoheadrightarrow\Z^2$, which is impossible.
Hence $G_+$ does not satisfy \textup{(ii)}.

The closed oriented three-manifold $M_+=Q\#(S^1\times S^2)$
has fundamental group $G_+$ by van Kampen's theorem.
Thus its center quotient satisfies \textup{(i)} and \textup{(iii)},
but not \textup{(ii)}.
The manifold $M_+\times S^1$ has fundamental group $G_+\times\Z$
and center quotient $G_+$, while $M_+\times S^{n-3}$ has fundamental
group $G_+$ for every $n\ge5$.
These give closed manifold realizations in every dimension at least
three.

\emph{Condition \textup{(ii)} implies neither \textup{(i)} nor
\textup{(iii)}.}
Consider the closed symplectic four-manifold
$$
M_-=\Sigma_2\times\Sigma_2,\qquad G_-=\pi_1(\Sigma_2)\times\pi_1(\Sigma_2),
$$
with the product of area forms.
The surface groups have trivial center, so $G_-$ is the center
quotient of $\pi_1(M_-)$; see
\cite[Corollary~7.2(a)]{Osin}.
The presentation
$$
\pi_1(\Sigma_2)=\langle a_1,b_1,a_2,b_2\mid[a_1,b_1][a_2,b_2]=1\rangle
$$
gives an epimorphism onto $\F=\langle a,b\rangle$ by sending
$a_1$ to $a$, $a_2$ to $b$, and $b_1,b_2$ to the identity;
see \cite[Section~1.2]{Hatcher}.
Composing with projection onto the first surface-group factor
proves \textup{(ii)} for $G_-$.

A product of two infinite groups is not acylindrically hyperbolic;
see \cite[Corollary~7.2(b)]{Osin}.
Any subgroup of $G_-$ isomorphic to $\F$ must be proper, since
$G_-$ contains $\Z^2$, whereas a free group does not.
A hyperbolically embedded copy of $\F$ would therefore imply
acylindrical hyperbolicity by \cite[Theorem~1.2]{Osin}.
Thus \textup{(iii)} fails, and Proposition~\ref{prop:AtoH} implies
that \textup{(i)} fails as well.

\emph{Condition \textup{(iii)} implies neither \textup{(i)} nor
\textup{(ii)}.}
Let $P=\pi_1(Q)$ from the first construction.
Since $Q$ is closed and hyperbolic, $P$ is center-free,
torsion-free, and non-elementary word-hyperbolic.
It has no nontrivial finite normal subgroup and satisfies
\textup{(iii)} by \cite[Theorem~6.14(c)]{DGO}.
Since $P$ is perfect, \textup{(ii)} fails.

The group $P$ is one-ended and has boundary $S^2$.
It has no nontrivial splitting over a finite group, and Bowditch's
splitting criterion excludes a splitting over an infinite cyclic
group: $S^2$ has no local cut point and $P$ is not cocompact
Fuchsian; see \cite[Section~0 and Theorem~6.2]{Bowditch}.
In a splitting as in \textup{(i)}, the edge group would be trivial
or infinite cyclic by \eqref{eq:edgecohom}.
Neither nontrivial amalgam is therefore possible.
The degenerate case would give $P\cong\F$, which is excluded by
$P_{\mathrm{ab}}=0$. Hence \textup{(i)} also fails.
The same splitting obstruction shows that the fundamental group
of any closed real hyperbolic $n$-manifold, $n\ge3$, fails
\textup{(i)}, since its boundary $S^{n-1}$ is connected and has
no local cut point. The homology-sphere choice additionally
ensures failure of \textup{(ii)}.

Thus $Q$ realizes a center quotient satisfying only \textup{(iii)}.
The center quotient of $Q\times S^1$ is again $P$, and
$\pi_1(Q\times S^{n-3})=P$ for $n\ge5$.
This gives such closed manifold realizations in every dimension
at least three.

\emph{Compact manifolds with boundary and symplectic realizations.}
Deleting the interiors of finitely many pairwise disjoint embedded
$n$-balls does not change the fundamental group when $n\ge3$.
Consequently, all the closed examples above also give compact
manifolds with nonempty boundary and the same center quotients.
The boundary conventions of Theorem~\ref{thm:main} apply.

By Gompf's realization theorem \cite{Gompf}, there are closed
connected symplectic four-manifolds $X_+$ and $X_3$ with
fundamental groups $G_+$ and $P$, respectively.
Together with $M_-=\Sigma_2\times\Sigma_2$, these realize all three
patterns by symplectic four-manifolds.
Their products with $\mathbb{CP}^{s}$, $s\ge0$, retain the same
fundamental groups and give symplectic examples in every even
dimension at least four. Removing Darboux balls gives the
corresponding symplectic examples with boundary.
Theorem~\ref{thm:main} applies to all four transformation groups
for these symplectic manifolds.

Finally, all three conditions hold for $\F$ itself: use the
amalgam $\F*_{1}1$, the identity epimorphism, and
$\F\hookrightarrow_h\F$.
For instance, $\#_2(S^1\times S^2)$ and its ball complement
realize this group.
The three preceding counterexamples establish all the claimed
nonimplications, completing the proof.
\end{proof}

Theorem~\ref{thm:surface} gives infinite-dimensional reduced bounded
cohomology of $\Homeo_0(T,\mu)$ in every degree at least two, and
infinite-dimensional reduced exact bounded cohomology in every even
degree. We conclude with the following question.

\begin{question*}
For every odd integer $d>3$, is the reduced exact bounded cohomology
$\rEH^d\bigl(\Homeo_0(T,\mu)\bigr)$ infinite-dimensional?
\end{question*}

\bibliographystyle{plain}
\bibliography{Bd_cohomology_all_degrees_references}
\end{document}